\documentclass[3p,review]{elsarticle}
\usepackage{amsmath}
\usepackage{amsfonts}
\usepackage{amsthm}
\usepackage{accents}

\newtheorem{theo}{Theorem}[section]
\newtheorem{coro}{Corollary}[section]
\newtheorem{lemm}{Lemma}[section]
\newtheorem{prop}{Proposition}[section]

\newdefinition{defn}{Definition}[section]
\newdefinition{rem}{Remark}[section]
\newdefinition{exmpl}{Example}[section]

\newenvironment{assp}[1]%
  {\par\normalfont\trivlist\item {\bfseries Assumption}~#1.\ \ignorespaces}
  {\endtrivlist}
\newenvironment{assps}[1]%
  {\par\normalfont\trivlist\item {\bfseries Assumptions}~#1.\ \ignorespaces}
  {\endtrivlist}

\numberwithin{equation}{section}

\makeatletter
\let\P\@undefined
\makeatother

\DeclareMathOperator{\P}{\mathbb P}        % probability
\DeclareMathOperator{\E}{\mathbb E}        % expectation
\DeclareMathOperator{\Var}{Var}            % variance
\DeclareMathOperator{\Po}{\mathcal P}      % Poisson point process
\DeclareMathOperator{\Pois}{Pois}          % Poisson distribution
\DeclareMathOperator{\1}{\mathbf 1}        % indicator
\DeclareMathOperator{\B}{\mathcal B}       % bounded measurable functions
\DeclareMathOperator{\Meas}{Meas}          % measures
\DeclareMathOperator{\SIG}{SIG}            % sphere of influence graph
\DeclareMathOperator{\NSIG}{NSIG}          % the neighbors in the sphere of influence graph
\DeclareMathOperator{\NNeigh}{NN}
\newcommand{\NN}[1]{\NNeigh^#1}            % `undirected' nearest neighbor set
\newcommand{\NND}[1]{\NNeigh^{#1, \to}}    % `directed' nearest neighbor set
\DeclareMathOperator{\Lin}{Lin}            % linear span
\DeclareMathOperator{\diam}{diam}          % diameter
\DeclareMathOperator{\vol}{vol}            % volume
\DeclareMathOperator*{\esssup}{ess\,sup}   % essential supremum

\DeclareMathOperator{\Dens}{D}             % Assumptions on density
\DeclareMathOperator{\Momla}{M}            % Assumption on a moment with respect to PPP
\DeclareMathOperator{\MomlaOne}{M1}        % Assumption on a moment with respect to PPP
\DeclareMathOperator{\MomGrInt}{MGI}       % Assumption on the growth of moments with respect to integral.
\DeclareMathOperator{\MomGrPt}{MGP}        % Assumption on the growth of moments with respect to subsets.
\DeclareMathOperator{\FinHom}{FH}          % Assumption on almost sure finiteness with respect to
\DeclareMathOperator{\MomHom}{MH}          % Assumption on a moment with respect to homogeneous PPP
\DeclareMathOperator{\MomHomOne}{MH1}      % Assumption on a moment with respect to homogeneous PPP
\DeclareMathOperator{\Gen}{G}              % Assumptions needed in all our main results.
\DeclareMathOperator{\ConvergWLLN}{CWLLN}  % Assumptions needed for the weak law of large numbers
\DeclareMathOperator{\ConvergVar}{CV}      % Assumptions needed for convergence of the variance
\DeclareMathOperator{\ConvergCLT}{CCLT}    % Assumptions needed for CLI

\newcommand{\0}{\mathbf 0}                 % the origin
\newcommand{\N}{\mathbb N}                 % natural numbers
\newcommand{\R}{\mathbb R}                 % real numbers
\newcommand{\X}{\mathcal X}                % point
\newcommand{\Y}{\mathcal Y}                %  sets

\newcommand{\U}{\mathcal U}
\newcommand{\Colon}{\colon\>}              % used for maps (f \Colon A \to B)
\newcommand{\sth}{\mathrel{;}}             % in sets: such that
\newcommand{\dst}{\displaystyle}
\newcommand{\eqd}{\stackrel{\mathcal{D}}{=}}  % equivalence in distribution
\newcommand{\la}{\lambda}                  % scaling factor
\newcommand{\ka}{\kappa}                   % initial intensity of a Poisson point process
\newcommand{\dl}{\mathrm d}                % differential
\newcommand{\eps}{\varepsilon}             % epsilon
\newcommand{\Mom}{M}                       % moment measure ...
\newcommand{\mom}{m}                       %   ... and its density
\newcommand{\cu}{c}                        % cumulant
\newcommand{\Cum}{c}                       % cumulant measure ...
\newcommand{\Clu}{U}                       % cluster measure
\newcommand{\Marks}{\mathcal M}            % mark space
\newcommand{\marked}{\breve}               % marked point or set
\newcommand{\markedname}{breve}            % name for the accent used for marked points or set
\newcommand{\volB}{\omega}                 % volume of the unit ball

\newcommand{\markedAin}{\marked{A}^{\mathrm{in}}}

\newcommand{\bdl}[2][\relax]{\mathrm{\bar d}\ifx#1\relax\else[#1]\fi #2}

\newcommand{\tdl}[2][\relax]{\mathrm{\tilde d}\ifx#1\relax\else[#1]\fi #2}

\newcommand{\muxilaka}{\mu_\la}           % instead of \mu^\xi_{\la \ka}
\newcommand{\barmuxilaka}{\bar\mu_\la}    % instead of \overline\mu^xi_{\la \ka}
\newcommand{\sigmaxika}{\sigma}           % instead of Q^\xi_\ka
\newcommand{\sigmaxikainf}{\sigma_-}      % instead of Q^\xi_\ka
\newcommand{\sigmaxilaka}{\sigma_\la}     % instead of \sigma_\la, keeps the previous notation
\newcommand{\Vxi}{V}                      % instead of V^\xi
\newcommand{\Ixika}[1]{I_{#1}}            % instead of I^\xi_{\ka; #1}
\newcommand{\Ixikameas}{I}                % instead of I^\xi_\ka
\newcommand{\Jxikameas}{J}                % instead of J^\xi_\ka
\newcommand{\Deltaxi}{\Delta}             % should be used for the add-one cost
\newcommand{\Hxi}{H}                      % should be used for the sum of the values of a geometric
\newcommand{\RDeltaxi}{R^\Delta}          % should be used for the radius of stabilization of the
\DeclareMathOperator{\sepp}{sep}          % instead of d expressing seperation
\DeclareMathOperator{\Sep}{Sep}           % instead of \sigma (in $ \sigma(\{ S, T \}) $)
\DeclareMathOperator{\dist}{dist}         % distance
\DeclareMathOperator{\bd}{\partial}       % boundary

\newcommand{\url}{\texttt}                % to refer to an URL

\journal{Annales de l'Institut Henri Poincar\'e (B) Probabilit\'es et Statistiques}

\begin{document}

% BibTeX was used to generate references, but they are currently incorporated.
% 
% \bibliographystyle{elsarticle-num}
% 

\begin{frontmatter}

\title{Moderate deviations for stabilizing functionals in geometric probability}

\author{P.~Eichelsbacher}
%
% \address{Fakult\"at f\"ur Mathematik, Ruhr-Universit\"at Bochum, NA 3/68, 44780 Bochum, Germany}
\address{Fakult\"at f\"ur Mathematik, Ruhr-Universit\"at Bochum, Germany}
\ead{peter.eichelsbacher@rub.de}

\author{M.~Rai\v{c}}
% \address{Faculty of Mathematics and Physics, University of Ljubljana,
% Jadranska 19, SI-1000 Ljubljana, Slovenia}
\address{University of Ljubljana (FMF) and University of Primorska (FAMNIT), Slovenia}
\ead{martin.raic@fmf.uni-lj.si}

\author{T.~Schreiber\fnref{fnTomasz}}
% \address{Faculty of Mathematics and Computer Science,
% Nicholas Copernicus University, Toru\'n, Poland}
\address{Faculty of Mathematics and Computer Science, Nicholas Copernicus University, Toru\'n, Poland}
% \ead{tomeks@mat.uni.torun.pl}

\fntext[fnTomasz]{Research partially supported by the Polish Minister of Science and
Higher Education grant N N201 385234 (2008-2010).\\
\textbf{Tomasz Schreiber passed away on $ \text{1}^{\text{st}} $ December 2010.}}

% *** No longer needed with the style provided by the journal ***
% 
% \date{\today}
% \maketitle
%
% \footnotetext{$~^1$
% Research partially supported by the Polish Minister of Science and
% Higher Education grant N N201 385234 (2008-2010)}
% 
% ***

\begin{abstract} The purpose of the present paper is to establish explicit
 bounds on moderate deviation probabilities for a rather general class
 of geometric functionals enjoying the stabilization property, under
 Poisson input and the assumption of a certain control over the growth of
 the moments of the functional and its radius of stabilization. Our proof
 techniques rely on cumulant expansions and cluster measures and yield
 completely explicit bounds on deviation probabilities. In addition, we
 establish a new criterion for the limiting variance to be non-degenerate.
 Moreover, our main result provides a new central limit theorem, which,
 though stated under strong moment assumptions, does not require bounded
 support of the intensity of the Poisson input.
 We apply our results to three groups of examples: random packing models,
 geometric functionals based on Euclidean nearest neighbors and
 the sphere of influence graphs.
\end{abstract}


\end{frontmatter}

\section{Introduction, main results}
\label{sc:Intr}

\subsection{Introduction}

Stabilization is an important concept expressing in natural geometric terms 
mixing properties of a broad class of functionals of point processes arising
in geometric probability, see \cite{PY1,PY2,BY2}. Even though these processes
are presumably also tractable using more traditional mixing concepts, 
stabilization-based techniques proved extremely convenient in studying
the asymptotic behavior of large random geometric systems. This is due
to the geometric nature of these methods which makes them compatible with
many stochastic geometric set-ups in which the target functionals arise.
In particular, stabilization is often helpful in establishing a direct
connection between the microscopic (local) and macroscopic properties of
the processes studied, see ibidem for further details. 
Stabilization has been successfully used to establish laws of large numbers
for many functionals \cite{Pe2007LLN,PY2,PY4} and it has also been employed in a
general setting to establish Gaussian limits for re-normalized functionals
as well as re-normalized spatial point measures \cite{BY2,Pe2,Pe2007CLT,PY1,PW}. The
functionals to which the afore-mentioned theory applies include those
defined by percolation models \cite{Pe2}, random graphs in computational
geometry \cite{BY2,PY1}, random packing models \cite{BY1,PY2}, and the germ--grain
models \cite{BY2}.
%, and the process of maximal points \cite{BY4}.
Large deviation principles for stabilizing functionals have
also been established, see \cite{SY}. Finally, the corresponding
moderate deviation principle, interpolating between the central limit 
theorem and law of large numbers, has been obtained in \cite{BESY} and
\cite{ES10}, for a rather limited sub-class of the above examples though,
namely for empirical functionals of random sequential packing, 
nearest neighbor graphs and germ--grain models.

In this paper, we use stabilization combined with cumulant
expansion techniques in the spirit of \cite{BY2}
to prove moderate deviation bounds for three groups of geometric
functionals: random sequential packing along with birth--growth models,
functionals based on nearest neighbors and sphere of influence graphs.
In each case, we consider a much more general class of geometric functionals,
based on a newly introduced concept of \emph{confinement}.
In particular, in the birth--growth models, we lift
the unnatural lower bound on grain sizes. Moreover, our
deviation probability estimates are much more explicit than those established
in \cite{BESY}. In addition, we derive moderate deviation principles
in the $ \tau $-topology, which is stronger than the weak topology,
which is used in \cite{BESY}.

We assume Poisson input with a density, which is assumed to be bounded
and integrable, but need not have bounded support. This fact adds to
existing central limit theorems, which, to the best of our knowledge,
all require the latter assumption: see Remark~\ref{probabilities-CLT}.
In particular, bounded support is not needed in the first group of
applications, random packing models. On the other hand, in the other
two groups, nearest neighbors and sphere of influence graphs, bounded
support is required in order to ensure stabilization.

On the other hand though, for the applications considered in \cite{BESY},
our approach provides a much narrower range of scaling regimes to which
moderate deviation results apply. In particular, our results in the current
form do not seem to provide moderate deviation principles in the full range.
However, in \cite{BESY}, it is assumed that the random measures associated
to the functionals and the Poisson point process can be suitably coupled with 
their exponential (Gibbs) modifications. In the present paper, we assume no such
structure; as a result, the application of our main results is much more
straightforward. Moreover, future refinements might actually yield full range
moderate deviation principles in some cases: see Remark~\ref{NonOpt}.

However, it is well known that many natural multidimensional
stochastic systems exhibiting various types of exponential mixing often
satisfy the Gaussian moderate deviation principle only up to a certain
point beyond the CLT scale, whereafter the Gaussian behavior breaks down
and gets replaced by phenomena of a different nature. As a spectacular
example, consider the phase separation, condensation and droplet creation
as established for many statistical mechanical models in phase transition
regime, see the seminal monograph \cite{DKS} as well as the survey 
\cite{BIV}. Consequently, we believe that the Gaussian moderate deviation
principle may well be violated by geometric stabilizing functionals 
for ranges far enough from the CLT regime. Even though we are 
definitely not in a position to claim that the ranges of Gaussian
behavior for deviation probabilities established in this paper are
optimal, we should most likely not hope to get a full range Gaussian moderate
deviation principle at the level of generality considered here.

\subsection{Terminology and notation}\label{ssc:terminology}

We begin with some common notation. First, denote $ \N = \{ 1, 2, 3, \ldots \} $
and $ \N_0 = \{ 0, 1, 2, \ldots \} $. Next, for a set $ A $, denote by $ |A| $
its cardinality. Throughout this paper, fix $ d \in \N $. For $ x \in \R^d $,
denote by $ \| x \| $ its Euclidean norm and, furthermore,
for $ A \subseteq \R^d $, denote by $ \dist(x, A) $ the Euclidean distance
from $ x $ to $ A $. Next, denote by $ \0 $ the origin in $ \R^d $.

Unless specified otherwise, the expression `measurable' will mean
`Borel-measurable' when applied to subsets of $ \R^d $.
For a measurable set $ A \subseteq \R^d $, denote by $ \vol(K) $ its
Lebesgue measure.
Throughout this paper, the letter $ \Omega $ will denote a measurable
subset of $ \R^d $, which will be called a \emph{domain}. Next,
the letter $ \ka $ will denote a probability density function
on $ \R^d $, vanishing on $ \R^d \setminus \Omega $. Abusing the notation slightly,
$ \ka $ will sometimes also denote the corresponding probability measure, i.~e.,
$ \ka(x) \,\dl x $. In particular, `$ \ka $-almost everywhere' will mean
`for $ \ka(x) \,\dl x $-almost all $ x $'.
% Let $ \supp(\ka) $ denote its support, i.~e., the smallest closest set
% $ C $ with $ \int_C \ka(x) \,\dl x = 1 $.

Furthermore, we let $\langle f, \mu \rangle$ denote the integral with respect to
a signed finite variation Borel measure $\mu$ of a $\mu$-integrable function $f$.
For a measurable set $ W \subseteq \R^d $, we write $ \B(W) $ for the collection
of bounded measurable $ f \Colon W \to \R $.
Finally, we shall assume that $ 0/0 = 0 $ (and $ a/0 = + \infty $ for $ a > 0 $
and $ - \infty $ for $ a < 0 $). The essential supremum of $ f $ with respect to
a measure $ \mu $ will be denoted by $ \esssup_{\mu(\dl x)} f(x) $. In particular,
for a non-negative function $ g $, $ \esssup_{g(x) \,\dl x} f(x) $ will denote
the essential supremum of $ f $ with respect to the Lebesgue measure restricted
to the set $ \{ x \sth g(x) > 0 \} $.

Next, we introduce \emph{marked points}. Let
$ (\Marks, \mathcal F_\Marks, \P_\Marks) $ be a probability space
(\emph{mark space}). Marked points will be the elements of
$ \marked{\R}^d := \R^d \times \Marks $ and will be usually denoted by
a \markedname\ accent. We shall use the following convention: if a letter with
a \markedname\ accent denotes a marked point and the same letter without the accent
appears in the same context, both will refer to the same location.
More formally, when $ \marked{x} \in \marked{\R}^d $ and $ x $ appear
in the same context, we shall always assume that $ \marked{x} = (x, t) $ for
some $ t \in \Marks $. Similarly, for a set $ \marked{\X} \subseteq \marked{\R}^d $,
$ \X $ in the same context will denote the set $ \{ x \sth (\exists t \in \Marks)
\> (x, t) \in \marked{\X} \} $. A subset $ \marked{A} \subseteq \marked{\R}^d $ will
be called \emph{measurable} if it is measurable with respect to the product of
the Borel $ \sigma $-algebra on $ \R^d $ and the $ \sigma $-algebra
$ \mathcal F_\Marks $.

\begin{trivlist}
\item \textbf{Common Conventions}.
Let $ z \in \R^d $ and $ a \in \R $. For $ \marked{x} = (x, t) \in
\marked{\R}^d $, put $ \marked{x} + z := (x + z, t) $ and
$ a \marked{x} := (a x, t) $. For a set $ \marked{\X} \subseteq \marked{\R}^d $,
put $ \marked{\X} + z := \{ \marked{x} + z \sth \marked{x} \in \marked{\X} \} $ and
$ \la \marked{\X} := \{ \la \marked{x} \sth \marked{x} \in \marked{\X} \} $. 
For sets $ \marked{A} \subseteq \marked{\R}^d $
and $ B \in \R^d $, define $ \marked{A} \cap B := \marked{A} \cap (B \times \Marks) $
and $ \marked{A} \setminus B := \marked{A} \setminus (B \times \Marks) $.
\end{trivlist}

We shall often integrate over $ \marked{\R}^d $. For that purpose, we extend
the meaning of differential to integration with respect to the product
of the Lebesgue measure and $ \P_\Marks $. More formally, we define:
\begin{equation}
\label{eq:diffMarked}
 \int_{\marked{\R}^d} f(\marked{x}) \,\dl \marked{x} :=
 \int_{\R^d} \E f(x, T) \,\dl x \, ,
\end{equation}
where $ T $ is a generic random mark with distribution $ \P_\Marks $.

For $ x \in \R^d $ and $ r > 0 $, $ B_r(x) $ denotes the closed Euclidean ball
of radius $ r $ centered at $ x $.
% Moreover, let $ \marked{B}_r(x) := B_r(x) \times \Marks $.
A set $ \marked{\X} \subseteq \marked{\R}^d $ will be called \emph{configuration}
if it is locally finite with respect to the product
of the Euclidean topology on $ \R^d $ and the indiscrete topology on
$ \Marks $, i.~e., if for each $ x \in \R^d $, the set $ \marked{\X}
\cap B_r(x) $ is finite for some $ r > 0 $ (recall from the Common Conventions
that $ \marked{\X} \cap B_r(x) $ is interpreted as $ \marked{\X}
\cap (B_r(x) \times \Marks) $). A configuration which is
also a subset of $ \marked{\Omega} := \Omega \times \Marks $ will be
called a \emph{configuration on $ \Omega $}. We shall keep the meaning of
$ \marked{\Omega} $ throughout the paper.

On the set of all configurations, we
shall consider the standard $ \sigma $-algebra defined as the smallest
$ \sigma $-algebra, such that the map $ \marked{\X} \mapsto |\marked{\X} \cap \marked{A}| $
is measurable for all measurable sets $ \marked{A} \subseteq \marked{\R}^d $.
Thus, a subset of the set of all configurations will be called measurable if it is
measurable with respect to the latter $ \sigma $-algebra.

Now we turn to our main object, geometric functionals and the associated
random measures. We first refine the definition of a geometric functional
from \cite{BY2, PY4}.

\begin{defn}
A \emph{geometric functional}
defined on a measurable class $ \mathcal{C} $ of configurations is a measurable map
defined for all pairs $ (\marked{x}, \marked{\X}) $, where $ \marked{\X} \in \mathcal{C} $
and $ \marked{x} \in \marked{\X} $. If the class $ \mathcal{C} $ is not specified, we shall
take the class of all \emph{finite} configurations. Another example of a class that we
shall frequently consider are all configurations on a domain $ \Omega $.
\end{defn}

Although the main object of our paper will be real-valued functionals, we shall also need
$ [0, \infty] $-valued as well as even set-valued functionals.

Now we turn to
\emph{stabilization}, which plays an essential role in all that follows.

\begin{defn}
\label{Stab}
A geometric functional $ \xi $ \emph{stabilizes}
at $ \marked{x} $ with respect to a
configuration $ \marked{\X} \subseteq \marked{\Omega} $ inside a domain $ \Omega $,
if there exists a finite $ \rho \ge 0 $, such that:
\begin{equation}
\label{eq:Stab}
 \xi(\marked{x}, \marked{\Y}) = \xi(\marked{x} , \marked{\X} \cap B_\rho(x))
\end{equation}
for all finite configurations $ \marked{\Y} \subseteq \marked{\Omega} $ with $ \marked{\Y} \cap B_\rho(x)
= \marked{\X} \cap B_\rho(x) $.
In other words, the interaction between $ \marked{x} $ and a point set
is unaffected by changes outside $ B_\rho(x) $. In particular, for $ r \ge \rho $,
$ \xi(\marked{x}, \marked{\X} \cap B_r(x)) $ does not depend on $ r $.
If \eqref{eq:Stab} holds, we shall say that $ \xi $ stabilizes at $ \marked{x} $
\emph{within radius $ \rho $} with respect to $ \marked{\X} $ inside $ \Omega $.
If $ \Omega $ is not specified, we shall take $ \Omega = \R^d $.

We shall say that $ \xi $ is \emph{stabilizing} with respect to
configuration $ \marked{\X} $ if it stabilizes with respect to $ \marked{\X} $
at each $ \marked{x} \in \marked{\X} $ (all inside a domain $ \Omega $).

A \emph{radius of stabilization} for a geometric functional $ \xi $ inside
$ \Omega $ is a $ [0, \infty] $-valued geometric functional $ R $ defined on
configurations on $ \marked{\Omega} $, such that for all suitable $ \marked{\X} $
and all $ \marked{x} \in \marked{\X} $, $ \xi $ stabilizes
at $ \marked{x} $ within radius $ R(\marked{x}, \marked{\X}) $, provided that
$ R(\marked{x}, \marked{\X}) $ is finite.
\end{defn}

Geometric functionals, which are initially defined on finite configurations,
can be extended to all configurations (i.~e., locally finite sets) with respect
to which they stabilize. More precisely, if a geometric functional $ \xi $ stabilizes
at $ \marked{x} $ with respect to configuration $ \marked{\X} $, we can define:
\begin{equation}
\label{eq:StabExt}
 \xi(\marked{x}, \marked{\X}) := \lim_{r \to \infty} \xi(\marked{x}, \marked{\X} \cap B_r(x)) \, .
\end{equation}
Clearly, the extended functional stabilizes within the same radius.

Among others, this extension allows us to consider geometric functionals
on Poisson point processes. For a locally integrable function
$ f \Colon \R^d \to [0, \infty) $, denote by $ \marked{\Po}_f $ a Poisson point
process with intensity $ f \otimes \P_\Marks $ (sometimes, we shall abuse the
notation slightly and identify $ f $ with the corresponding measure $ f(x) \,\dl x $).
By locally integrable function, we mean that $ f \in L^1_{\textrm{loc}}(\R^d) $, i.~e.,
for each $ x \in \R^d $, there exists some neighborhood $ U $, such that
$ f \in L^1(U) $.

\begin{rem}
It would also be desirable extend the results to the case when the
Poisson input is replaced by binomial (i.~e., a point process
consisting of independent and identically distributed random points)
-- so called \emph{de-Poissonization}.
However, this is beyond the scope of the present paper. In particular,
the crucial construction~\eqref{eq:clust:PPPEx} seems not to work. Instead,
one can either attempt to find a more sophisticated construction or
refer to suitable convergence of binomial point processes to Poisson.
In this context, coupling can be used to advantage: see
Lemma~4.2 of \cite{PY1} and Lemma~3.2 of \cite{Pe2007LLN}.
\end{rem}

Throughout this paper, the letter $ \xi $ will be reserved for geometric functionals
(but we shall also consider geometric functionals denoted by other letters). For
a geometric functional denoted by $ \xi $, define its scaled versions by:
$$
 \xi_\la(\marked{x}, \marked{\X}) = \xi \bigl( \la^{1/d} \marked{x}, \la^{1/d}
  \marked{\X} \bigr)
$$
(recalling that the $ \la^{1/d} $ scaling only modifies the spatial component, not
the mark). Our principal objects of interest are the following random point measures
on $ \R^d $:
\begin{equation}
\label{eq:muxilaka}
 \muxilaka := \sum_{\marked{x} \in \marked{\Po}_{\la \ka}}
  \xi_\la (\marked{x}, \marked{\Po}_{\la \ka}) \, \delta_x
  \> ; \quad \la > 0 \, ,
\end{equation}
or, equivalently,
\begin{equation*}
% \label{eq:fmuxilaka}
 \langle f, \muxilaka \rangle := \int_{\marked{\R}^d} f(x) \,
  \xi_\la (\marked{x}, \marked{\Po}_{\la \ka})
  \, \marked{\Po}_{\la \ka}(\dl \marked{x}) \, ,
\end{equation*}
along with $ \barmuxilaka := \muxilaka - \E \muxilaka$, the centered versions of
$ \muxilaka $. Thus, in $ \muxilaka $ and $ \barmuxilaka $, we shall always
refer to a $ \R $-valued geometric functional denoted by $ \xi $ and a probability
density function on $ \R^d $ denoted by $ \ka $, which will be suppressed in the
notation.

Throughout this paper, unless specified otherwise, the letter $ R $ will denote a
radius of stabilization for a geometric functional denoted by $ \xi $. Moreover, for
$ \la > 0 $, denote by $ R_\la $ a non-negative geometric functional defined on locally
finite sets, such that $ \la^{-1/d} R_\la $ is a radius of stabilization for $ \xi_\la $
inside a domain denoted by $ \Omega $, or, equivalently, such that the functional
$ (\marked{x}, \marked{\X}) \mapsto R_\la \bigl( \la^{-1/d} \marked{x}, \la^{-1/d} \marked{\X} \bigr) $
is a radius of stabilization for $ \xi $ inside $ \la^{1/d} \Omega $.
Of course, if $ \Omega = \R^d $, one can simply put
$ R_\la(x, \X) := R \bigl( \la^{1/d} \marked{x}, \la^{1/d} \marked{\X} \bigr) $.

%
%C% Not entirely correct, left out.
%
% In some cases, one can construct $ R_\la $ from $ R $. First, if $ \Omega = \R^d $,
% one can simply put $ R_\la(x, \X) := R \bigl( \la^{1/d} \marked{x}, \la^{1/d}
% \marked{\X} \bigr) $. Second, if $ \xi $ is homogeneous in the sense that
% there exists a function $ F $ such that $ \xi_\la(\marked{x}, \marked{\X})
% = F \bigl( \xi(\marked{x}, \marked{\X}), \la \bigr) $ for all $ \marked{x} $ and
% $ \marked{\X} $, we can put $ R_\la(\marked{x}, \marked{\X}) := \la^{1/d} R(\marked{x}, \marked{\X}) $.

In most of our results, we restrict attention to translation invariant
functionals.

\begin{defn}
A geometric functional $ \xi $ is \emph{translation invariant} if
$ \xi(\marked{x} + z, \marked{\X} + z) = \xi(\marked{x}, \marked{\X}) $ for all
$ \marked{x} \in \marked{\R}^d $, all finite configurations $ \X \subset \R^d $
and all $ z \in \R^d $.
\end{defn}

\begin{rem}
Though most of our main results are formulated under the assumption
of translation invariance, the latter is not crucial. In particular,
it is not required in our key Lemma~\ref{growthlemm}. However, our
main results are derived by combining Lemma~\ref{growthlemm} with
\eqref{eq:Varlim}, which is known to hold under either translation
invariance or more complicated assumptions (see \cite{Pe2007CLT}).
Anyway, translation invariance is the case in all our applications.
\end{rem}

%C% Left out, see %C% below.
%
% If $ \xi $ is translation invariant and has a radius of stabilization
% which is at at least one point almost surely finite, then it also has
% a translation invariant radius of stabilization which is almost
% surely finite (at all points).
% %
% %C% Not only finiteness, moments are also important!
% %
% Therefore, if $ \xi $ is translation
% invariant and stabilizing, then it suffices to derive results for
% the case where its radius of stabilization is translation invariant,
% too.

%C% WRONG and therefore left out.
%
% \begin{rem}
% Under Assumption~$\TransInv$, \eqref{eq:xiMomGrInt} and
% \eqref{eq:StabMomGrInt} can be reduced to:
% $$
%  \E \bigl[ \xi_\la((\0, T), \marked{\Po}_{\la \ka}) \bigr]^k \le A^k (k!)^\alpha
%   \quad \text{and} \quad
%  \E \bigl[ R_\la((\0, T), \marked{\Po}_{\la \ka}) \bigr]^k \le B^k (k!)^\beta
%   \, ,
% $$
% where, again, $ T $ denotes a generic random mark with distribution $ \P_\Marks $,
% independent of $ \marked{\Po}_{\la \ka} $.
% \end{rem}

\subsection{Known results}\label{ssc:KnRe}

From \cite{Pe2007CLT} (see also earlier papers \cite{BY2,PY4}), we know
that under relatively mild assumptions (compared to those stated in
Subsection~\ref{ssc:DevProb}), the one and two point correlation functions for
$\xi_{\la}(\marked{x}, \marked{\Po}_{\la \ka })$ converge in the large $\la$ limit, which establishes
the corresponding asymptotics for integrals
$ \E \langle f, \muxilaka \rangle $ and $ \sigmaxilaka^2[f] := 
\Var \bigl( \langle f, \muxilaka \rangle \bigr) $
(\emph{we shall always take $ \sigmaxilaka[f] \ge 0 $}).
Moreover, under additional assumption that $ \ka $ has bounded support, it is known that the
limit of the re-normalized measures $ \la^{-1/2} \barmuxilaka $ is a generalized mean zero
Gaussian field in the sense that the finite-dimensional distributions of
$ \la^{-1/2} \barmuxilaka $ over $ f \in \B(\R^d) $ converge to those of a Gaussian
field (for test functions, we may take all $ f \in \B(\R^d) $, but only the ones
with support coinciding with the support of $ \ka $ really matter).

To state the results in formal terms, we introduce various assumptions. The first one
will be imposed on the density $ \ka $.

%C% Peter thinks that the following is too pedantic, left out.
%
% In its statement, each assumption will be formally \emph{imposed} on some object,
% but might be later also used on other objects. If an assumption is stated as imposed
% on, say, $ \xi $, we shall say that $ \xi' $ \emph{satisfies} that assumption
% if the latter holds with $ \xi' $ in place on $ \xi $.

\begin{assp}{$\Dens$ (\textit{Density})}
$ \ka $ is bounded and Lebesgue-almost everywhere continuous.
\end{assp}

\noindent
Next, we list two assumptions imposed on a family $ (g_\la)_{\la > \la_0} $
of geometric functionals either taking values in $ \R $ or in $ [0, \infty] $.

\begin{assp}{$\Momla(p, \ka)$ (\textit{$ p $-th Moment})}
$ \dst \limsup_{\la \to \infty} \esssup_{\ka(x) \,\dl \marked{x}}
  \E \bigl| g_\la(\marked{x}, \marked{\Po}_{\la \ka}) \bigr|^p < \infty $.
\end{assp}

\begin{assp}{$\MomlaOne(p, \ka)$ (\textit{$ p $-th Moment with One additional point})}\newline
$ (g_\la)_\la $ satisfies Assumption~$\Momla(p, \ka)$ and
$ \dst
 \limsup_{\la \to \infty} \esssup_{\ka(x) \,\dl \marked{x} \otimes \ka(y) \,\dl \marked{y}}
  \E \bigl| g_\la(\marked{x}, \marked{\Po}_{\la \ka} \cup \{ \marked{y} \}) \bigr|^p < \infty \, .
$
\end{assp}

Now recall our conventions on $ R $ and $ R_\la $ from the preceding subsection.
In most of our results, we shall need that $ R $ satisfies Assumption~$\MomlaOne(q, \ka)$
for some $ q $. This is analogous, but not entirely equivalent to the \emph{power-law stabilization
of order $ q $} as defined in \cite{Pe2005,Pe2007CLT,PY6}.

Below we state an assumption formally imposed on $ R $, a radius of stabilization
for $ \xi $. Throughout this subsection, let $ T $ denote a generic random mark with
distribution $ \P_\Marks $, independent of all other random objects we consider.

\begin{assp}{$\FinHom(\tau)$ (\textit{Finiteness with respect to Homogeneous process})}\newline
$ R $ is translation invariant and $ \P \bigl( R((\0, T), \marked{\Po}_\tau) < \infty \bigr) = 1 $.
Here, $ \tau $ denotes a positive real number, not a function. Notice that if $ R $ is not
a priori translation invariant, it can always be replaced by an appropriate translation
invariant radius of stabilization.
\end{assp}

%
%C% Follows from Assumption~$\FinHom(\tau)$ and therefore left out.
%
% \begin{assp}{$\FinHomOne(\tau)$ (\textit{Finiteness with respect to Homogeneous process with One
% additional point})}\newline
% $ R $ satisfies Assumption~$\FinHom(\tau)$ and
% $ R((\0, t), \marked{\Po}_\tau \cup \{ \marked{x} \}) $ is almost surely finite
% $ \P_\Marks(\dl t) \otimes \dl \marked{x} $-almost everywhere.
% \end{assp}

\begin{rem}
\label{tauHom}
Under translation invariance, Assumption~$\FinHom(\tau)$ is essentially equivalent to
what is called \emph{$ \tau $-homogeneous stabilization} in \cite{Pe2005,Pe2007CLT}.
More precisely, a translation invariant geometric functional $ \xi $ is
$ \tau $-homogeneously stabilizing if:
\begin{equation}
\label{eq:tauHom}
    \P \bigl( R((\0, T), \marked{\Po}_\tau) < \infty \bigr)
  = \P \bigl( R((\0, T), \marked{\Po}_\tau \cup \{ \marked{x} \}) < \infty \bigr)
  = 1
\end{equation}
for all $ \marked{x} \in \marked{\R}^d $. However, if the latter only holds for \emph{almost}
all $ \marked{x} $, say, for all $ \marked{x} \in \marked{\R}^d \setminus (\{ \0 \}) $
outside a set $ \marked{N} \subset \marked{\R}^d \setminus (\{ \0 \}) $ of measure zero,
one can modify $ \xi((\0, t), \marked{\X}) $ to $ \xi((\0, t), \marked{\X} \setminus \marked{N}) $
and $ R((\0, t), \marked{\X}) $ to $ R((\0, t), \marked{\X} \setminus \marked{N}) $, keeping
translation invariance. Replacing $ \marked{\X} $ with a marked Poisson point process
$ \marked{\Po}_f $, $ \xi $ and $ R $ have then only been modified on a null event.

Once relaxed to the `almost everywhere' condition, the additional point $ \marked{x} $
can now be left out. This is because for each $ r > 0 $, a Poisson point process can be
represented as a union of a set $ \marked{\U}_{N,r} $ and a Poisson point process
on the complement of $ \marked{B}_r(\0) := B_r(\0) \times \Marks $, where
$ N $ is a Poisson random variable and where for each $ n \in \N_0 $,
$ \marked{\U}_{n,r} $ denotes a set of $ n $ independent points in $ \marked{B}_r(\0) $.
Saying that the first probability in \eqref{eq:tauHom} equals one
is equivalent to saying that $ \P \bigl( R((\0, T), \marked{\U}_{n,r} \cup (\marked{\Po}_\tau
\setminus \marked{B}_r(\0))) < \infty \bigr) = 1 $ for all $ r > 0 $ and all $ n \in \N_0 $,
while for the second probability, this reduces to the statement that the latter
holds for all $ r > 0 $ and all $ n \in \N $. Therefore, the second probability in
\eqref{eq:tauHom} is in fact redundant. To summarize, \emph{any statement on the distributions
of the random measures $ \muxilaka $ which holds under \eqref{eq:tauHom} still holds under
Assumption~$\FinHom(\tau)$ only.}
\end{rem}

\begin{defn}
\label{ExpDec}
A family $ (g_\la)_{\la > \la_0} $ of geometric functionals enjoys
\emph{$ \ka $-almost exponential decay} if there exist $ a \ge 0 $ and $ b > 0 $, such
that $ \esssup_{\ka(x) \,\dl \marked{x}}
\P \bigl( \bigl| g_\la(\marked{x}, \marked{\Po}_{\la \ka}) \bigr| > t \bigr) \le a \, e^{-bt} $
for all $ t \ge 0 $. If this is satisfied for $ g_\la = R_\la $, where, by our convention,
$ \la^{-1/d} R_\la $ is a radius of stabilization for $ \xi_\la $, we shall say that $ \xi $
is \emph{$ \ka $-almost exponentially stabilizing inside the upscaled domain $ \Omega $}.
If $ \Omega $ is not specified, we shall take $ \Omega = \R^d $.
\end{defn}

Next, we list three further assumptions imposed on $ \xi $. Recall that
$ \xi_\la(\marked{x}, \marked{\X}) = \xi \bigl( \la^{1/d} \marked{x},
\la^{1/d} \marked{\X} \bigr) $.

\begin{assps}{$\ConvergWLLN(p, \ka)$ (\textit{Convergence in the sense of WLLN})}
$ \ka $ satisfies Assumption~$\Dens$, $ \xi $ is translation invariant and
the family $ (\xi_\la)_{\la > 0} $ satisfies Assumption~$\MomlaOne(p', \ka)$ for
some $ p' > p $. Moreover, there exists a radius of stabilization $ R $ for $ \xi $
satisfying Assumption~$\FinHom(\ka(x))$ for $ \ka $-almost all $ x \in \R^d $.
\end{assps}

\begin{assps}{$\ConvergVar(\ka)$ (\textit{Convergence of Variance})}
$ \ka $ satisfies Assumption~$\Dens$ and $ \xi $ is translation invariant.
Moreover, there exists a radius of stabilization $ R $ for $ \xi $ satisfying
Assumption~$\FinHom(\ka(x))$ for $ \ka $-almost
all $ x \in \R^d $. Finally, there exist $ p, q > 0 $ with $ 2/p + d/q < 1 $, such
that the family $ (\xi_\la)_{\la > 0} $ satisfies Assumption~$\MomlaOne(p, \ka)$ and
such that there exists a family $ (R_\la)_{\la > \la_0} $ of appropriate radii of
stabilization satisfying Assumption~$\Momla(q, \ka)$.
\end{assps}

\begin{assps}{$\ConvergCLT(\ka)$ (\textit{Convergence in the sense of multivariate CLT})}
$ \ka $ satisfies Assumption~$\Dens$ and has bounded support, and $ \xi $ is
translation invariant. Moreover, there exists a radius of stabilization $ R $ for
$ \xi $ satisfying Assumption~$\FinHom(\ka(x))$ for $ \ka $-almost all $ x \in \R^d $.
Finally, there exists a family $ (R_\la)_{\la > \la_0} $ of appropriate radii of
stabilization, such that either $ (\xi_\la)_{\la > 0} $ satisfies Assumption~$\MomlaOne(p, \ka)$
for some $ p > 2 $ and $ (R_\la)_{\la > \la_0} $ enjoys almost uniform exponential decay
for $ \ka $, or such that $ (\xi_\la)_{\la > 0} $ satisfies Assumption~$\MomlaOne(p, \ka)$ for
some $ p > 3 $ and $ (R_\la)_{\la > \la_0} $ satisfies Assumption~$\Momla(q, \ka)$
for some $ q > d(150 + 6/p) $.
%
%C% Curiously, this condition is not of form $ a/p + b/q < 1 $.
%
\end{assps}

\noindent
Now we are ready to list the following known results:
\begin{itemize}
 \item Take $ p = 1 $ or $ p = 2 $. If $ \xi $ satisfies Assumptions~$\ConvergWLLN(p, \ka)$,
 we have for all $ f \in \B(\R^d) $:
 \begin{equation}
 \label{eq:Elim}
  \frac{\langle f, \muxilaka \rangle}{\la} \xrightarrow[\la \to \infty]{L^p}
   \int_{\R^d} f(x) \E \bigl[ \xi((\0, T), \marked{\Po}_{\ka(x)}) \bigr] \, \ka(x) \,\dl x \, ,
 \end{equation}
where $ T $ is a generic random mark with distribution $ \P_{\Marks} $, independent
of $ \marked{\Po}_{\ka(x)} $,
 \item If $ \xi $ satisfies Assumptions~$\ConvergVar(\ka)$, there exists a measurable
 function $ \Vxi \Colon [0, \infty) \to [0, \infty) $, depending on $ \xi $ but not
 on $ \ka $, such that for all $ f \in \B(\R^d) $, the variance $ \sigmaxilaka^2[f] =
 \Var \bigl( \langle f, \muxilaka \rangle \bigr) $ satisfies:
 \begin{equation}
 \label{eq:Varlim}
  \lim_{\la \to \infty} \frac{\sigmaxilaka^2[f]}{\la} =
   \sigmaxika^2[f] :=
   \int_{\R^d} f^2(x) \, \Vxi(\ka(x)) \, \ka(x) \,\dl x \, ,
 \end{equation}
 (like for $ \sigmaxilaka[f] $, we shall always take $ \sigmaxika[f] \ge 0 $).
 The function $ \Vxi $ is defined by:
 \begin{equation}
 \label{eq:Vxi}
 \begin{split}
  \Vxi(\tau) &:= \E \Bigl[ \xi \bigl( (\0, T), \marked{\Po}_\tau \bigr)^2 \Bigr] + \null
 \\ & \qquad \null
   + \tau \int_{\R^d} \biggl\{ \E \Bigl[
     \xi \bigl( (\0, T), \marked{\Po}_\tau \cup \{ (z, T') \} \bigr) \,
     \xi \bigl( (z, T'), \marked{\Po}_\tau \cup \{ (\0, T) \} \bigr)
   \Bigr] - \Bigl[ \E \xi \bigl( (\0, T), \marked{\Po}_\tau \bigr) \Bigr]^2 \biggr\} \,\dl z
 \end{split}
 \end{equation}
 for $ \tau > 0 $; for convenience, we set $ \Vxi(0) := 0 $. Similarly as before,
 $ T $ and $ T' $ are generic random marks with distribution $ \P_{\Marks} $, independent
 of each other as well as of $ \marked{\Po}_{\tau} $.
 \item If $ \xi $ satisfies Assumptions~$\ConvergCLT(\ka)$, the finite-dimensional distributions
 of $ \la^{-1/2} \barmuxilaka $ converge in distribution as $\la \to \infty$
 to those of a generalized mean-zero Gaussian field with covariance kernel
 \begin{equation}
 \label{eq:Gausslim}
  (f_1,f_2) \mapsto \int_{\R^d} f_1(x) \, f_2(x) \, \Vxi(\ka(x))
   \, \ka(x) \,\dl x \, .
 \end{equation}
\end{itemize}
% see Theorem~2.1 of \cite{Pe2007LLN} and Theorems~2.1--2.3 along with
% the remarks below them in \cite{Pe2007CLT}:

The above results capture the weak law of large numbers
and the Gaussian limit  behavior of the re-normalized measures
$ \la^{-1/2} \barmuxilaka $. For a special case (special mainly with respect to
test functions), which is actually a good starting point, they are all proved
in \cite{BY2} (Theorem~2.1). For the general case (in fact,
much more general that we consider), see \cite{Pe2005} (Theorems~2.1--2.3),
noting that we may relax the supremum over $ \la \ge 1 $ along with the supremum
over the measure-theoretic support of $ \ka $ stated in the conditions given ibidem
to $ \limsup_{\la \to \infty} $ and the $ \ka $-essential supremum. Furthermore,
we can relax $\tau$-homogeneous stabilization to Assumption~$\FinHom(\tau)$: see
Remark~\ref{tauHom}. The convergence of the variance \eqref{eq:Varlim} and
the central limit theorem are also proved in \cite{Pe2007CLT} (Theorems~2.1 and 2.2,
but see also the remarks below Theorem~2.3).

The above-mentioned central limit theorem is based on Theorems~2.3 and
2.5 of \cite{PY6}, where also the curious condition $ q > d(150 + 6/p) $
arises from. The latter results also provide explicit univariate bounds.
An extension to multivariate bounds is derived in \cite{PW}.
Under a different set-up, a multivatiate CLT is also proved in \cite{Pe2}.

As to Theorem~2.2 of \cite{Pe2005} and its counterpart, Theorem~2.1 of \cite{Pe2007CLT},
it is worth mentioning that a factor $ d $ has been accidentally dropped from their
statement, so that there should be $ q > dp/(p-2) $ rather than $ q > p/(p-2) $, just
like in Assumptions~$\ConvergVar(\ka)$: see Lemma~5.2 of \cite{Pe2005} and Lemma~4.2 of
\cite{Pe2007CLT}.

\subsection{Non-degeneracy of the limiting variance}\label{ssc:NDeg}

In view of \eqref{eq:Varlim}, it is important to distinguish between
degenerate and non-degenerate limiting variance, i.~e.,
$ \sigmaxika[f] = 0 $ and $ \sigmaxika[f] > 0 $.
This issue is heavily discussed in \cite{PY1,PY2,BY2}. A further fruitful general
result is derived in \cite{PW}. However, the verification of the conditions guaranteeing
non-degeneracy given in the latter paper might be somewhat involved. Therefore,
we take one step forward and establish a new criterion for the non-degeneracy,
which seems to be easier to verify. Moreover, the result of \cite{PW} is not
supported by an example where earlier results do not apply. We here provide one:
see Example~\ref{NN:ExpLength}.

Like the earlier result from \cite{PY1}, Theorem~2.2 of \cite{PW} is based on the
\emph{add-one cost} of the total mass functional. The \emph{total mass functional}
of a geometric functional $ \xi $ is defined by
$ \Hxi(\marked{\X}) := \sum_{\marked{x} \in \marked{\X}} \xi(\marked{x}, \marked{\X}) $,
while its add-one cost is defined by $ \Deltaxi(\marked{x}, \marked{\X}) :=
\Hxi(\marked{\X}) - \Hxi(\marked{\X} \setminus \{ \marked{x} \}) $,
provided that $ \marked{x} \in \marked{\X} $. Notice that $ \Deltaxi $ is also
a geometric functional. Extending it by our convention from
Subsection~\ref{ssc:terminology}, we have $ \Deltaxi(\marked{x}, \marked{\X})
= \Hxi(\marked{\X} \cup \{ \marked{x} \}) - \Hxi(\marked{\X}) $ for
$ \marked{x} \notin \marked{\X} $. The name `add-one cost' is due to the latter
case.

In order to derive non-degeneracy, a new concept of stabilization, called
\emph{external stabilization}, is introduced in \cite{PW}. We here rewrite the
corresponding definition for marked configurations and, in addition, introduce
the concept of \emph{basic external stabilization}.

\begin{defn}
\label{ExtStab}
Let $ \rho > 0 $. A configuration $ \marked{\X} $ is said to be
\emph{basically $ \rho $-externally stable} at a point $ \marked{x} \in \marked{\R}^d $
with respect to a geometric functional $ \xi $ if
\begin{equation}
\label{eq:ext:out}
 \xi(\marked{z}, \marked{\Y} \cup \{ \marked{x} \}) = \xi(\marked{z}, \marked{\Y} \setminus \{ \marked{x} \})
\end{equation}
for all finite $ \marked{\Y} $ with $ \marked{\Y} \cap B_\rho(\marked{x})
= \marked{\X} \cap B_\rho(\marked{x}) $ and all $ \marked{z} \in \marked{\Y} \setminus
B_\rho(\marked{x}) $.

The configuration $ \marked{\X} $ is said to be \emph{$ \rho $-externally stable} at 
$ \marked{x} $ with respect to $ \xi $ if the following three conditions hold:
first, $ \marked{\X} $ is $ \rho $-externally stable at $ \marked{x} $ with respect to
$ \xi $. Second, $ \xi $ stabilizes at $ \marked{x} $ within radius $ \rho $ with respect
to $ \marked{\X} $. Third,
\begin{equation}
\label{eq:ext:in}
 \xi \bigl( \marked{y}, \marked{\Y} \cup \{ \marked{x} \} \bigr)
   - \xi \bigl( \marked{y}, \marked{\Y} \setminus \{ \marked{x} \} \bigr) =
 \xi \bigl( \marked{y}, (\marked{\X} \cap B_\rho(\marked{x})) \cup \{ \marked{x} \} \bigr)
   - \xi \bigl( \marked{y}, (\marked{\X} \cap B_\rho(\marked{x})) \setminus \{ \marked{x} \} \bigr)
\end{equation}
for all finite $ \marked{\Y} $ with $ \marked{\Y} \cap B_\rho(\marked{x})
= \marked{\X} \cap B_\rho(\marked{x}) $ and all $ \marked{y} \in \marked{\Y}
\cap B_\rho(\marked{x}) $.

We shall say that $ \marked{\X} $ is \emph{externally stable} at $ \marked{x} $
with respect to $ \xi $ if it is $ \rho $-externally stable at $ \marked{x} $
for some finite $ \rho $.
\end{defn}

\begin{rem}
\label{ExtStabDelta}
If $ \marked{\X} $ is $ \rho $-externally stable at $ \marked{x} $ with respect to $ \xi $, then
$ \Deltaxi $ stabilizes at $ \marked{x} $ within radius $ \rho $.
\end{rem}

\begin{rem}
\label{ExtStabInc}
If $ \marked{\X} $ is $ \rho $-externally stable and $ r \ge \rho $, then
$ \marked{\X} $ is also $ r $-externally stable.
\end{rem}

\begin{rem}
\label{ExtStabMatch}
If $ \marked{\X} $ is $ \rho $-externally stable at $ \marked{x} $ and
$ \marked{\Y} \cap B_\rho(\marked{x}) = \marked{\X} \cap B_\rho(\marked{x}) $, then
$ \marked{\Y} $ is also $ \rho $-externally stable at $ \marked{x} $.
\end{rem}

\begin{rem}
Condition~\eqref{eq:ext:in} holds provided that $ \xi $ stabilizes at $ \marked{y} $
with respect to $ \marked{\X} \cup \{ \marked{x} \} $ and to $ \marked{\X} \cup \{ \marked{x} \} $
within radius $ \rho - \| y - x \| $. On the other hand, condition~\eqref{eq:ext:out} cannot
be simply derived from the usual radius of stabilization. Therefore it is referred to as
basic external stabilization.
\end{rem}

In view of the preceding remark, the following construction of the external
radius of stabilization is now immediate.

\begin{prop}
\label{StabExtStab}
Let $ \xi $ be a geometric functional with radius of stabilization $ R $.
Take $ \rho > 0 $, $ \marked{x} \in \marked{\R}^d $ and a configuration
$ \marked{\X} $. If $ \marked{\X} $ is basically $ \rho $-externally stable
at $ \marked{x} $ with respect to $ \xi $, it is also $ r $-externally
stable at $ \marked{x} $ with respect to $ \xi $, where:
$$
 r := \rho + \max \bigl\{ R(\marked{y}, \marked{\X} \cup \{ \marked{x} \}),
                          R(\marked{y}, \marked{\X} \setminus \{ \marked{x} \})
        \sth \marked{y} \in (\marked{\X} \cup \{ \marked{x} \}) \cap B_\rho(\marked{x}) \bigr\} \, .
$$
In particular, if $ \marked{\X} $ is basically $ \rho $-externally stable
at $ \marked{x} $ with respect to $ \xi $ and $ \xi $ stabilizes with respect
to $ \marked{\X} \cup \{ \marked{x} \} $ and $ \marked{\X} \setminus \{ \marked{x} \} $
at all $ \marked{y} \in (\marked{\X} \cup \{ \marked{x} \}) \cap B_\rho(\marked{x}) $,
then $ \marked{\X} $ is externally stable at $ \marked{x} $ with respect to $ \xi $.\qed
\end{prop}

In the rest of this subsection as well as in Subsection~\ref{ssc:NDeg:Proof}, we
fix a translation invariant geometric functional $ \xi $, let $ \Hxi $ be the total mass
functional of $ \xi $ and let $ \Deltaxi $ be its add-one cost.
We also recall the convention $ \xi_\la(\marked{x}, \marked{\X}) :=
\xi_\la \bigl( \la^{1/d} x, \la^{1/d} \X \bigr) $; set also
$ \Deltaxi_\la(\marked{x}, \marked{\X}) := \Deltaxi_\la \bigl( \la^{1/d} x, \la^{1/d} \X \bigr) $.
Let $ R $ denote a radius of stabilization for $ \xi $ and let
$ \la^{-1/d} R_\la $ be a radius of stabilization for $ \xi_\la $. Finally,
let $ T $ denote a generic random mark with distribution $ \P_\Marks $, independent
of all other random variables.

Now we state the version of Theorem~2.2 of \cite{PW} for marked configurations.
One can easily check that the proof given ibidem still carries
through.
% We here thank Mathew Penrose for confirming this in a personal
% communication.
%
%.% This will not be confirmed.

\begin{theo}
\label{NDeg:PW}
Suppose that $ \xi $ is translation invariant and that there exists
$ \rho > 0 $, such that with strictly positive probability, the marked
homogeneous Poisson process $ \marked{\Po}_1 $ is $ \rho $-externally
stable at $ (\0, T) $ and $ \Deltaxi \bigl( (\0, T), $
\newline
$ \marked{\Po}_1 \cap B_\rho(\0) \bigr)
\ne 0 $. Next, take a probability density function $ \ka $ on $ \R^d $ and
a function $ f \in \B(\R^d) $, which is Lebesgue-almost everywhere
continuous. Suppose that $ f \ka $ is not Lebesgue-almost everywhere zero.
Let $ X $ be a random variable with density $ \ka $.
Suppose that for some $ s > 2 $, we have:
\begin{equation}
\label{eq:NDeg:PW:DeltafMom}
 \limsup_{\la \to \infty} \E \biggl|
   \sum_{\marked{x} \in \marked{\Po}_{\la \ka} \cup \{ (X, T) \}}
     f(x) \, \xi\bigl( \marked{x}, \marked{\Po}_{\la \ka} \cup \{ (X, T) \} \bigr) -
   \sum_{\marked{x} \in \marked{\Po}_{\la \ka} \setminus \{ (X, T) \}}
     f(x) \, \xi\bigl( \marked{x}, \marked{\Po}_{\la \ka} \setminus \{ (X, T) \} \bigr)
 \biggr|^s < \infty \, .
\end{equation}
In addition, suppose either that $ f $ is an indicator function of a measurable
subset of $ \R^d $ with vanishing boundary (i.~e., the Lebesgue measure of the
boundary equals zero) or that for all $ K > 0 $ and $ \ka $-almost all $ x $, we
have:
\begin{equation}
\label{eq:NDeg:PW:xiMom}
 \int_{B_K(\0)} \E \bigl| \xi(\marked{y}, \marked{\Po}_{\ka(x)} \cup \{ (\0, T) \})
   - \xi(\marked{y}, \marked{\Po}_{\ka(x)}) \bigr| \,\dl \marked{y} < \infty \, .
\end{equation}
Then $ \liminf_{\la \to \infty} \sigmaxilaka^2[f]/\la  > 0 $.\qed
\end{theo}

As already mentioned, the verification of the conditions required in the
preceding result might be somewhat involved. In particular, this holds for the moment
conditions \eqref{eq:NDeg:PW:DeltafMom} and \eqref{eq:NDeg:PW:xiMom}. The main goal of
this subsection is to show that these conditions can be simply
replaced by Assumptions~$\ConvergVar(\ka)$ introduced in Subsection~\ref{ssc:KnRe},
taking $ \ka $ to be the uniform density on a suitable domain and restricting attention
to the case $ f \equiv 1 $. Observe first that in this case, the moment condition
\eqref{eq:NDeg:PW:xiMom} can be left out, while the moment condition \eqref{eq:NDeg:PW:DeltafMom}
reduces to:
\begin{equation}
\label{eq:NDeg:PW:DeltaMom}
 \limsup_{\la \to \infty} \E \bigl| \Deltaxi_\la \bigl( (X, T), \marked{\Po}_{\la \ka} \bigr) \bigr|^s
   < \infty \, .
\end{equation}
Considering only the case $ f \equiv 1 $ is \emph{not} a big restriction:
under Assumptions~$\ConvergVar(\ka)$, formula~\eqref{eq:Varlim} allows us to derive
the limiting variance from the function $ \Vxi $; notice that for each $ \tau > 0 $,
$ \Vxi(\tau) $ is precisely the limiting variance for $ f \equiv 1 $, taking
$ \ka $ to be a suitable uniform density (see also Remark~\ref{MomHomUnif}).

For convenience, we list two additional assumptions imposed on
a family $ (g_\la)_{\la > \la_0} $ of geometric functionals. Below we
show that they are essentially equivalent to Assumptions~$\Momla$ and
$\MomlaOne$; moreover, with proper parameters and one additional
assumption, they are also equivalent to Assumptions~$\ConvergVar(\ka)$.

\begin{assp}{$\MomHom(p, \tau, \Omega)$ (\textit{$ p $-th Moment with respect to
Homogeneous process restricted to $ \Omega $})}\newline
$ \dst \limsup_{\la \to \infty} \esssup_{\1(x \in \Omega) \, \dl \marked{x}}
  \E \bigl| g_\la(\marked{x}, \marked{\Po}_{\la \tau} \cap \Omega) \bigl|^p < \infty
$.
\end{assp}

\begin{assp}{$\MomHomOne(p, \tau, \Omega)$ (\textit{$ p $-th Moment with respect
to Homogeneous process restricted to $ \Omega $ with one additional point})}
$ (g_\la)_{\la > \la_0} $ satisfies Assumption~$\MomHom(p, \tau, \Omega)$ and\newline
$ \dst \limsup_{\la \to \infty} 
  \esssup_{\1(x, y \in \Omega) \, \dl \marked{x} \otimes \dl \marked{y}}
  \E \bigl| g_\la \bigl( \marked{x}, (\marked{\Po}_{\la \tau} \cap \Omega)
    \cup \{ \marked{y} \} \bigr) \bigl|^p < \infty
$.
\end{assp}

\begin{rem}
\label{MomHomUnif}
Letting $ v := \vol(\Omega) $, $ \Omega^* := (\tau v)^{-1/d} \Omega $,
$ g_{\la^*}^*(\marked{x}^*, \marked{\X}^*) := g_{\la} \bigl( (\tau v)^{1/d} \marked{x}^*,
(\tau v)^{1/d} \marked{\X}^* \bigr) $, where $ \la = \la^*/(\tau v) $,
\emph{the family $ (g_\la)_{\la \ge \la_0} $ satisfies Assumption~$\MomHom(p, \tau, \Omega) $
if and only if the family $ (g^*_{\la^*})_{\la^* \ge \la^*_0} $ satisfies
Assumption~$\Momla(p, \ka)$}, taking $ \ka $ to be the uniform density on $ \Omega^* $ and
$ \la^*_0 := \tau v \la_0 $. This is true because:
\begin{equation*}
\begin{split}
 g_\la(\marked{x}, \marked{\Po}_{\la \tau} \cap \Omega)
    &=  g_{\la^*/(\tau v)} \bigl( (\tau v)^{1/d} \marked{x}^*,
          \marked{\Po}_{\la^*/v} \cap (\tau v)^{1/d} \Omega^* \bigr)
    =  g_{\la^*}^* \bigl( \marked{x}^*, (\tau v)^{-1/d} \marked{\Po}_{\la^*/v} \cap \Omega^* \bigr)
  \eqd
\\
 &\eqd g_{\la^*}^* \bigl( \marked{x}^*, \marked{\Po}_{\la^* \tau} \cap \Omega^* \bigr)
    =  g_{\la^*}^* \bigl( \marked{x}^*, \marked{\Po}_{\la^* \ka} \bigr) \, ,
\end{split}
\end{equation*}
where $ \marked{x}^* := (\tau v)^{-1/d} \marked{x} $ and where $ \eqd $ denotes equivalence in distribution.
Similarly, \emph{the family $ (g_\la)_{\la \ge \la_0} $ satisfies Assumption~$\MomHomOne(p, \tau, \Omega) $
if and only if the family $ (g^*_{\la^*})_{\la^* \ge \la^*_0} $ satisfies
Assumption~$\MomlaOne(p, \ka)$}. Notice also that \emph{the family $ (\xi_\la)_{\la > 0} $ satisfies
Assumption~$\MomHom(p, \tau, \Omega)$, respectively $\MomHomOne(p, \tau, \Omega)$, if and only
if it satisfies Assumption~$\Momla(p, \ka)$, respectively $\MomlaOne(p, \ka)$}.

Moreover, keeping the relationship between $ \la $ and $ \la^* $, \emph{$ \la^{-1/d} R_\la $ is a radius
of stabilization for $ \xi_\la $ inside domain $ \Omega $ if and only if $ (\la^*)^{-1/d} R^*_{\la^*} $
is a radius of stabilization for $ \xi_{\la^*} $ inside $ \Omega^* $}, where
$ R_{\la^*}^*(\marked{x}^*, \marked{\X}^*) := R_{\la} \bigl( (\tau v)^{1/d} \marked{x}^*,
(\tau v)^{1/d} \marked{\X}^* \bigr) $. As a result,
\emph{if the family $ (\xi_\la)_{\la > 0} $ satisfies Assumption~$\MomHomOne(p, \tau, \Omega)$
and if there exists a radius $ R $ of stabilization for $ \xi $ satisfying Assumption~$\FinHom(\tau)$
and a suitable family $ (R_\la)_{\la > \la_0} $ of scaled radii of stabilization
satisfying Assumption~$\MomHom(q, \tau, \Omega)$, where $ 2/p + d/q < 1 $, then
$ \xi $ satisfies Assumption~$\ConvergVar(\ka)$}.
\end{rem}

As mentioned above, the main goal of this subsection is to show that the moment condition
\eqref{eq:NDeg:PW:DeltaMom} can be replaced by Assumptions~$\ConvergVar(\ka)$. In addition,
we also simplify the condition on non-vanishing $ \Delta $. In order to do this, we state
the following definition.

\begin{defn}
A predicate $ P $ defined on pairs $ (t, \marked{\X}) $, where $ t \in \Marks $ and
$ \marked{\X} $ is a finite configuration, is said to hold \emph{for notably many}
pairs $ (t, \marked{\X}) $ if there exists $ n \in \N_0 $ and a
$ \P_\Marks(\dl t) \otimes \dl \marked{x}_1 \otimes \cdots \otimes \dl \marked{x}_n $-nonnull
set, such that $ P(t, \{ \marked{x}_1, \ldots, \marked{x}_n \}) $ holds for all
$ (t, \marked{x}_1, \ldots, \marked{x}_n) $ in that set.
\end{defn}

\begin{rem}
In particular, for $ P(t, \marked{\X}) $ to hold for notably many
pairs $ (t, \marked{\X}) $, it suffices that
$ P(t, \emptyset) $ holds on a $ \P_\Marks(\dl t) $-nonnull set.
\end{rem}

Now we are ready to formulate our result on non-degeneracy. We defer the proof to
Subsection~\ref{ssc:NDeg:Proof}.

\begin{theo}
\label{NDeg}
Let $ \xi $ be translation invariant, let $ 0 < \tau < \infty $ and suppose that there
exists a radius of stabilization $ R $ for $ \xi $ satisfying Assumption~$\FinHom(\tau)$.
Let $ \Omega $ be a domain with $ 0 < \vol(\Omega) < \infty $ and with
$ \vol(\bd \Omega) = 0 $. Take $ p, q > 0 $ with $ 2/p + d/q < 1 $ and $ \la_0 > 0 $.
Suppose that the family $ (\xi_\la)_{\la > \la_0} $ satisfies Assumption~$\MomHomOne(p, \tau, \Omega)$
and that there exists a family $ (R_\la)_{\la > \la_0} $ satisfying
Assumption~$\MomHom(q, \tau, \Omega)$, such that for each $ \la > \la_0 $,
$ \la^{-1/d} R_\la $ is a radius of stabilization for $ \xi_\la $.
Finally, suppose that for notably many pairs $ (t, \marked{\X}) $,
$ \Deltaxi \bigl( (\0, t), \marked{\X} \bigr) \ne 0 $ and $ \marked{\X} $ is
externally stable at $ (\0, t) $ with respect to $ \xi $. With $ \Vxi $ as in
\eqref{eq:Vxi}, we then have $ \Vxi(\tau) > 0 $.
\end{theo}

%C% It is extremely unnatural that a result regarding homogeneous processes
%%% requires domains. However, this is due to Theorem~\ref{NDeg:PW}, which
%%% is designed for domains. Getting rid of domains seems to require
%%% considerable extra effort, which we leave for the future.

\subsection{Estimates on deviation probabilities}
\label{ssc:DevProb}

Starting from the known results it is natural to investigate the asymptotics of
\emph{deviation probabilities} on a scale larger than that of the central limit
theorem. To this end we consider a fixed test-function 
$f \in \B(\R^d)$ and strive to get precise information on bounds
of the relative error
\begin{equation}
\label{eq:ldpprob}
 \frac{\P ( \langle f, \barmuxilaka \rangle \geq x)}{1- \Phi(x/\sigmaxilaka[f])},
 \mbox{    as well as    }
 \frac{\P ( \langle f, \barmuxilaka \rangle \leq -x)}{\Phi(-x/\sigmaxilaka[f])},
 \;\; x > 0, 
\end{equation}
where, as in the preceding Subsection~\ref{ssc:KnRe}, 
$\sigmaxilaka^2[f] $ denotes the variance and where, as usual,
$$
 \Phi(x) = \frac{1}{\sqrt{2 \pi}} \int_{- \infty}^x e^{-t^2/2} \, \dl t
$$
is the distribution function of the standard normal.
In particular, we are interested in conditions under which the relative error \eqref{eq:ldpprob}
converges to 1 uniformly in the interval $0 \leq x \leq F(\la)$, where $F(\la)$ is a
nondecreasing function such that $F(\la) \to \infty$. Of course $F(\la)$ will depend
not only on $\la$ but also on other characteristics of our models, in particular their
dimensionality. For the sake of readability, the dependence on other quantities is
suppressed in all of our notation.

Since we will refine the cumulant expansion method of \cite{BY2} to establish more precise
rates of growth on the cumulants, in both their scale parameter and their order, we will
be able to apply a powerful and general lemma on deviation probabilities due to Rudzkis,
Saulis and Statulevi\v{c}ius \cite{RSS}, whose version specialized for our purposes
is stated as Lemma~\ref{RSSLemma} in the sequel for the convenience of the reader.

Before formulating the results, we list the following two key assumptions imposed
on a family $ (g_\la)_{\la \ge \la_0} $ of geometric functionals:

\begin{assp}{$\MomGrPt(\alpha, \ka)$, $ \alpha \ge 0 $ (\textit{Moment Growth with additional Points})}
There exist $ A \ge 0 $ and $ q > 0 $, such that for all $ \la \ge \la_0 $,
all $ k \in \N $ and all $ r \le qk $,
\begin{equation}
\label{eq:MomGrPt}
 \esssup_{
  \ka(x_1) \,\dl \marked{x}_1 \otimes \cdots \otimes \ka(x_r) \,\dl \marked{x}_r
% \marked{x}_1, \ldots, \marked{x}_r \in \marked{\R}^d
  }
   \E \bigl| g_\la(\marked{x}_1, \marked{\Po}_{\la \ka} \cup \{ x_1, \ldots, x_r \})
    \bigr|^k \, \ka(x_1) \,\dl \marked{x}_1
  \le A^k (k!)^\alpha
\end{equation}
(recall that $ \esssup_{f(x) \,\dl x} g(x) $ denotes the essential supremum of $ g $ with respect
to the Lebesgue measure restricted to the set $ \{ x \sth f(x) > 0 \} $).
\end{assp}

\begin{assp}{$\MomGrInt(\alpha, \ka)$, $ \alpha \ge 0 $ (\textit{Moment Growth with respect to Integral})}
There exists $ A \ge 0 $, such that for all $ \la \ge \la_0 $
and all $ k \in \N $,
\begin{equation}
\label{eq:MomGrInt}
 \int_{\R^d}
   \E \bigl| g_\la(\marked{x}, \marked{\Po}_{\la \ka}) \bigr|^k \, \ka(x) \,\dl \marked{x}
  \le A^k (k!)^\alpha \, .
\end{equation}
\end{assp}

\begin{rem}
Clearly, Assumption~$\MomGrInt(\alpha, \ka)$ is weaker than $\MomGrPt(\alpha, \ka)$.
\end{rem}

\begin{rem}
If there exists a family $ (R_\la)_{\la \ge \la_0} $ according to our convention from
Subsection~\ref{ssc:terminology}, satisfying Assumption~$\MomGrInt(\alpha, \ka)$ for
some $ \alpha \ge 0 $, then $ \xi $ is almost surely stabilizing with respect to
$ \Po_{\la \ka} $ for all $ \la \ge \la_0 $. As a result, the random measures
$ \muxilaka $ are almost surely defined.
\end{rem}

It is worth to point out two special cases. First, observe that if the functionals
$ |g_\la| $ are uniformly bounded, then the corresponding family satisfies
Assumptions~$\MomGrPt(0, \ka)$ and $\MomGrInt(0, \ka)$. When this
is true for $ g_\la = R_\la $, we shall say that $ \xi $ is
\emph{uniformly stabilizing inside the upscaled domain $ \Omega $}.
The second special case is when the family enjoys $ \ka $-almost
exponential decay (Definition~\ref{ExpDec}). In this case, we can estimate:
\begin{equation}
\label{eq:ExpDec1}
 \E \bigl| g_\la(\marked{x}, \marked{\Po}_{\la \ka}) \bigr|^k
  = \int_0^\infty k \, t^{k-1}
  \P \bigl( \bigl| g_\la(\marked{x}, \marked{\Po}_{\la \ka}) \bigr| > t \bigr) \,\dl t
  \le \frac{a \, k!}{b^k} \, .
\end{equation}
As a result, the family $ (g_\la)_{\la \ge \la_0} $ satisfies
Assumption~$\MomGrInt(1, \ka)$.

Now we state our central assumption, which is imposed on and will be used exclusively
for $ \xi $:

% Using the afore-mentioned techniques we shall establish the following estimate
% for deviation probabilities of the centered empirical measures
% $ \barmuxilaka $:

\begin{assps}{$\Gen(\gamma, \ka)$ (\textit{General conditions for our main results})}
$ \ka $ satisfies Assumption~$\Dens$ and there exist $ \alpha, \beta \ge 0 $,
$ \la_0 > 0 $ and a family $ (R_\la)_{\la \ge \la_0} $ according to our convention
from Subsection~\ref{ssc:terminology}, satisfying Assumption~$\MomGrInt(\beta, \ka)$
along with at least one of the following two conditions fulfilled:
\begin{list}{\textbullet}{\topsep 0pt\parsep 0pt\itemsep 0pt}
\item $ (\xi_\la)_{\la \ge \la_0} $ satisfies Assumption~$\MomGrInt(\alpha, \ka)$ and
$ 1 + \alpha + \beta d = \gamma $.
\item $ (\xi_\la)_{\la \ge \la_0} $ satisfies Assumption~$\MomGrPt(\alpha, \ka)$ and 
$ \max\{ \alpha, 1 \} + \beta d = \gamma $.
\end{list}
\end{assps}

The following result concerns deviation probabilities.

\begin{theo}\label{probabilities} %(deviation probabilities for Poisson samples)
Suppose that $ \xi $ satisfies Assumptions~$\Gen(\gamma, \ka)$ and take $ f \in \B(\R^d) $.
\newline
Let $ \sigmaxikainf[f] := \liminf_{\la \to \infty} \la^{-1/2} \sigmaxilaka[f] $.
\begin{enumerate}
\item[$(1)$] Suppose that, in addition, $ \sigmaxikainf[f] > 0 $. Then, for all
$ \la \ge \la_1 $ and $ 0 \leq x \leq C_1 \la^{(1+\gamma)/(1+2\gamma)} $, we have:
\begin{align}
\label{eq:large+}
 \left| \log \frac{\P\bigl( \langle f, \barmuxilaka \rangle \geq x \bigr)}%
  {1-\Phi(x/\sigmaxilaka[f])} \right|
  &\leq C_2 \left( \frac{1}{\la^{1/(2+4\gamma)}} + \frac{x^3}{\la^{(2+3\gamma)/(1+2\gamma)}} \right) \, ,
\\
\label{eq:large-}
 \left| \log \frac{\P\bigl( \langle f, \barmuxilaka \rangle \leq - x \bigr)}%
   {\Phi(-x/\sigmaxilaka[f])} \right|
  &\leq C_2 \left( \frac{1}{\la^{1/(2+4\gamma)}} + \frac{x^3}{\la^{(2+3\gamma)/(1+2\gamma)}} \right) \, ,
\end{align}
where $ \la_1 $ only depends on $ f $, $ \ka $, $ \xi $ and $ R $, whereas
$ C_1 $ and $ C_2 $ only depend on the \emph{ratio} $ \| f \|_\infty/\sigmaxikainf[f] $
along with $ \ka $, $ \xi $ and $ R $.
\item[$(2)$] Suppose that $ 0 < \sigmaxilaka[f] \le C_3 \la^{1/2} $. Then, for all $ x \ge 0 $, we have:
\begin{equation}
\label{eq:large0}
 \P \bigl( \pm \langle f, \barmuxilaka \rangle \geq x \bigr)
  \le \exp \left( - \min \left\{ C_4 \, \frac{x^2}{\sigmaxilaka^2[f]}, \> C_5 \, x^{1/(1 + \gamma)}, \>
   C_6 \left( \frac{x^3}{\la} \right)^{1/(2 + \gamma)} \right\} \right) \, ,
\end{equation}
\end{enumerate}
where $ C_3 $--$ C_6 $ only depend on $ f $, $ \ka $, $ \xi $ and $ R $.
\end{theo}

\begin{rem}
The second part of the theorem above is especially useful for degenerate cases, i.~e.,
$ \sigmaxika[f] = 0 $. As an example, one can consider the total number of edges in
the Voronoi graph: see Subsection~8.2 of \cite{PY1}.
\end{rem}

\begin{rem}
\label{probabilities-CLT}
In particular, under Assumptions~$\Gen(\gamma, \ka)$ and provided that $ \sigmaxikainf[f] > 0 $,
Theorem~\ref{probabilities} provides a central limit theorem. Comparing to
Assumptions~$\ConvergCLT(\ka)$, none of them implies
the others. Assumptions~$\Gen(\gamma, \ka)$ roughly include much stronger moments conditions,
but do not require boundedness of the support of $ \ka $. Similarly, to the best of our knowledge,
none of the existing central limit theorems has been proved under conditions weaker than
Assumptions~$\Gen(\gamma, \ka)$. Thus, Theorem~\ref{probabilities} also adds to existing CLT's.
\end{rem}

\begin{rem}
\label{NonOpt}
At least in certain cases, there appears scope for improvement. Some related results
indicate that if $ \xi $ satisfies Assumption~$\MomGrInt(0, \ka)$ and
$ R $ satisfies Assumption~$\MomGrPt(0, \ka)$ (i.~e., both are almost surely uniformly bounded),
Theorem~\ref{probabilities} should actually hold
for $ \gamma = 0 $ (full range large deviation principles, see next subsection) rather
than $ \gamma = 1 $. Results leading to full range large deviation principles are derived in
\cite{Gorch} for sums of locally dependent random variables (provided that the random variables
as well as the vertex degrees in the dependence graph are uniformly bounded), in \cite{LH2005}
for germ--grain models and in \cite{LH2009} for a more general case, where
germs are affine subspaces instead of points. Notice that in the latter case,
the corresponding geometric functional is even not stabilizing in the
sense of the present paper. However, it is uniformly stabilizing if we replace
balls by direct sums of $(d-k)$-dimensional balls and $k$-dimensional subspaces.
This is due to different behavior of the variance.

Indeed, clever modification and application of Lemma~1 of \cite{Gorch} might
relax the expression $ \max \{ \alpha, 1 \} + \beta d $ in Assumptions~$\Gen(\gamma, \ka)$
to some continuous function $ \psi(\alpha, \beta) $ with $ \psi(0, 0) = 0 $ and
$ \lim_{\alpha, \beta \to \infty} \bigl( \alpha + \beta d - \psi(\alpha, \beta) \bigr) = 0 $.
Details may appear in forthcoming work.
\end{rem}

\subsection{Moderate deviation principles}
\label{ssc:MDP}

It is natural to investigate the asymptotics of
$ (\barmuxilaka)_\la$ on intermediate scales between those appearing
in Gaussian and law of large numbers behavior.
This leads us to moderate deviation principles (MDPs).
In this paper we are able to deduce
moderate deviation principles from Theorem~\ref{probabilities} for a typically
\emph{partial} intermediate regime for stabilizing
$\xi$ (for the full scale, we have to assume $ \alpha = \beta = 0 $; see
Theorem~\ref{MDPTWI} below for a formal statement). We remark that in \cite{BESY},
moderate deviation principles were obtained for an essentially smaller
set of examples, including the prototypical random sequential packing and some 
spatial birth-and-growth models as well as for empirical functionals of nearest
neighbor graphs, but they were obtained on \emph{every} intermediate scale. 

We say that a family of probability measures $(\nu_{\la})_{\la}$ on $ \mathcal T $,
which is a measurable as well as a topological space, obeys a large deviation principle
(LDP) with speed $ a_\la $ and good rate function $I(\cdot) \Colon \mathcal T \to
[0, \infty] $ as $ \la \to \la_0 $ if
\begin{itemize}
\item $I$ is lower semi-continuous and has compact level sets
$ N_L := \{ x \in \mathcal T \Colon I(x) \le L\} $, for every $L \in  [0, \infty)$.
\item For every measurable set $ \Gamma $, we have:
\begin{equation}
\label{eq:DefLDP}
 - \inf_{x \in \accentset{\circ}{\Gamma}} I(x)
  \le \liminf_{\la \to \la_0} \frac{1}{a_\la} \log \nu_{\la}(\Gamma)
  \le \limsup_{\la \to \la_0} \frac{1}{a_\la} \log \nu_{\la}(\Gamma)
  \le - \inf_{x \in \overline{\Gamma}} I(x) \, ,
\end{equation}
where $ \accentset{\circ}{\Gamma} $ denotes the topological interior of $ \Gamma $
and $ \overline\Gamma $ denotes its closure.
\end{itemize}
Notice that we do not assume that the measures are Borel. In other words, open sets
are not necessarily measurable.

Similarly we will say that a family of $ \mathcal T $-valued random variables
$(Y_{\la})_{\la}$ obeys a large deviation principle with speed
$a_\la$ and good rate function $I(\cdot) \Colon \mathcal T \to
[0, \infty] $ if the sequence of their distributions does.
Formally a moderate deviation principle is nothing but an LDP. However, we
will speak about a moderate deviation principle
for a
sequence of random variables whenever the scaling of the
corresponding random variables is between that of an ordinary Law
of Large Numbers and that of a Central Limit Theorem.
\medskip

Take $ \gamma \ge 0 $ (arising from Assumptions~$\Gen(\gamma, \ka)$) and
consider $ \la \in (0, \infty) $, $ \la \to \infty $. Let
$(a_{\la})_{\la>0}$ be such that
\begin{equation}
\label{eq:assal}
 \lim_{\la \to \infty} a_{\la} = \infty \quad \text{and} \quad
  \lim_{\la \to \infty} \frac{a_{\la}}{\la^{1/(2+4\gamma)}}=0 \, .
\end{equation}
Under these assumptions, we first state the following MDP for
$ \barmuxilaka $:

\begin{theo}\label{MDPTWI} %(MDP on Poisson samples)
Suppose that $ \xi $ satisfies Assumptions~$\Gen(\gamma, \ka)$ and $\ConvergVar(\ka)$,
and take $ f \in \B(\R^d) $.
% with $ \sigmaxika[f] > 0 $.
Then, for each $(a_{\la})_{\la >0}$ satifying
\eqref{eq:assal},
the family of random variables $(a_{\la}^{-1} \la^{-1 \slash 2}
\langle f, \barmuxilaka \rangle )_{\la}$
satisfies on $\R$ the moderate deviation principle
with speed $a_{\la}^2$
and good rate function
\begin{equation}
\label{eq:MDPRF}
 \Ixika{f}(t) := \frac{t^2}{2 \sigmaxika^2[f]} \, ,
\end{equation}
where $ \sigmaxika $ is as in \eqref{eq:Varlim} and where possible division by zero is
handled according to our convention at the end of Subsection~\ref{ssc:terminology}.
\end{theo}

The next result is a MDP on the level of \emph{measures}.
Denote by $ \Meas(\R^d)$ the real vector space of finite signed measures on $\R^d$.
Equip $ \Meas(\R^d)$ with the \emph{$ \tau $-topology} generated by the sets:
$$
 U_{f,x,\delta} := \bigl\{ \nu \in \Meas(\R^d) \sth | \langle f, \nu \rangle -x | < \delta \bigr\} \, ,
$$
where $ f \in \B(\R^d) $, $ x \in \R $ and $ \delta > 0 $.
It is well known that since the collection of linear functionals
$\{\nu \mapsto \langle f, \nu \rangle \sth f \in \B(\R^d) \}$
is separating in $ \Meas(\R^d)$, this topology makes $ \Meas(\R^d)$
into a locally convex, Hausdorff topological vector space, whose topological dual
is the preceding collection, hereafter identified with $ \B(\R^d)$. With
this notation we establish the following \emph{measure-level} MDP for
$ \barmuxilaka $:

\begin{theo} \label{MDPmeasure} %(measure level MDP)
Suppose that $ \xi $ satisfies Assumptions~$\Gen(\gamma, \ka)$ and $\ConvergVar(\ka)$.
Then for any family $(a_{\la})_{\la >0}$ satisfying \eqref{eq:assal},
the family $ (a_{\la}^{-1} \, \la^{-1/2} \, \barmuxilaka)_{\la} $
satisfies the MDP on $ \Meas(\R^d)$, endowed with the $ \tau $-topology, with
speed $a_{\la}^2$ and the convex, good rate function given by
\begin{equation}
\label{eq:rateemp}
 \Ixikameas(\nu) := \frac{1}{2} \sigmaxika^2 \biggl[
  \frac{\dl \nu}{\Vxi(\ka(x)) \, \ka(x) \, \dl x}
 \biggr]
\end{equation}
if $\nu \in \Meas(\R^d)$ is absolutely continuous with
respect to $ \Vxi(\ka(x)) \, \ka(x) \, \dl x $, and by
$ \Ixikameas(\nu) := +\infty$ otherwise. Again,
$ \sigmaxika $ is as in \eqref{eq:Varlim}.
\end{theo}

\begin{rem}
Theorem~\ref{MDPmeasure} provides a MDP with respect to the
$ \tau $-topology, which is based on \emph{measurable} bounded
test functions. Therefore, this result has a stronger nature
than the corresponding Theorem~2.2 of \cite{BESY}, which
is stated in the weak topology, based on \emph{continuous} bounded
test functions.
\end{rem}

\section{Applications}\label{sc:APPLIC}

We here provide three groups of applications of our deviation bounds and moderate
deviation principles: models related to random sequential packing,
functionals related to $ k $ nearest neighbors and sphere of influence graphs. These
applications have been considered in detail in the context of central limit theorems
\cite{PY1, BY2, PY6} and in the context of laws of large numbers in \cite{PY2, PY4}.
In the context of moderate deviation principles, packing and nearest neighbors were
considered in \cite{BESY}. For all groups of applications,
we establish results of a relatively universal nature: to the best of
our knowledge, they are more general than those stated in the literature. We show
where our large deviation results improve and generalize over \cite{BESY}.

To set up the framework under which our results are stated, we here introduce
a new concept, which we shall call \emph{confinement}. Similarly as in stabilization,
the idea is that the value of a functional at $ \marked{x} $ depends only on some
`neighborhood' of $ \marked{x} $. The concept of stabilization is based on
metric neighborhoods, while the concept of confinement is entirely based on sets.
In precise terms, it goes as follows.

\begin{defn}
\label{confin}
Let $ h $ be a set-valued geometric functional, such that $ h(\marked{x}, \marked{\X})
\subseteq \marked{\X} $ for all $ \marked{\X} $ and all $ \marked{x} \in \marked{\X} $.
A geometric functional $ \xi $ is \emph{confined to $ h $} if $ \xi(\marked{x}, \marked{\X})
= \xi(\marked{x}, h(\marked{x}, \marked{\X})) $ for all $ \marked{\X} $ and
all $ \marked{x} \in \marked{\X} $.
\end{defn}

\begin{rem}
\label{StabConfin}
Let $ \xi $ be confined to $ h $. Then any radius of stabilization for $ h $
is also a radius of stabilization for $ \xi $. Similarly, if
$ \marked{\X} $ is basically $ \rho $-externally stable at $ \marked{x} $ with
respect to $ h $, it is also basically $ \rho $-externally stable at $ \marked{x} $
with respect to $ \xi $.
\end{rem}

\begin{rem}
\label{ConfinSubs}
Let $ h_1 $ and $ h_2 $ are set-valued geometric functionals, such that
$ h_1(\marked{x}, \marked{\X}) \subseteq h_2(\marked{x}, \marked{\X})
\subseteq \X $ for all relevant $ \marked{x} $ and $ \marked{\X} $.
Suppose that $ h_1 $ is stable in the sense that $ h_1(\marked{x}, \marked{\Y})
= h_1(\marked{x}, \marked{\X}) $ if $ h_1(\marked{x}, \marked{\X}) \subseteq
\marked{\Y} \subseteq \marked{\X} $. Then any geometric functional
confined to $ h_1 $ is also confined to $ h_2 $.
\end{rem}

Now we are ready to focus on each group of applications separately.

\subsection{Random sequential packing and related models}
\label{ssc:RSA}

The following prototypical random sequential packing/adsorption (RSA)
model arises in diverse disciplines, including physical, chemical,
and biological processes. See \cite{PY2} for a discussion of the
many applications, the many references, and also a discussion of
previous mathematical analysis. In one dimension, this model is
often referred to as the R\'enyi car parking model \cite{Re}.

Consider a finite set $\X \subset \R^d$
and to each $x \in \X$ attach a ball with some fixed diameter $ \rho $
centered at $x$. Moreover, to all points in $\X$ attach i.i.d.\ uniform
time marks taking values in some finite time interval, say, $ [0,1] $.
This establishes a
\emph{chronological} order on the points of $\X$. As usual, denote by
$ \marked{\X} $ the configuration of points of $ \X $ along with their time marks.
Declare the first point in the chronological ordering \emph{accepted}
and proceed recursively, each time accepting the next point if the ball
it carries does not overlap the previously accepted (packed) balls and
rejecting it otherwise. The functional $\xi(\marked{x},\marked{\X})$ is defined to be $1$ if
the ball centered at $x$ has been accepted and $0$ otherwise.
This defines the prototypical random sequential packing/adsorption
(RSA) process.

One can also consider infinite periods of packing, i.~e., input point
processes on $ \marked{\R}^d = \R^d \times [0, \infty) $. Take a Poisson point
process with density $ \ka(x) \,\dl x \otimes \dl t $,
where $ \ka $ is a probability density function with bounded support. Then,
clearly, only finitely many points can be accepted, so that all points that appear
after a certain time are rejected. Moreover, almost surely, there is actually no
more available space for packing. This is called \emph{jamming}: see
\cite{Pe1}. This setting allows to define the random measures $ \muxilaka $
in just the same way as in \eqref{eq:muxilaka}, although the measure on the
mark space is infinite. However, the latter fact prevents us from applying
our results directly. Although one might use truncation of time, jamming will
not be considered in this paper.

The RSA model can be extended in numerous other ways: see~\cite{Pe2007CLT,PY2}.
In particular, the decision whether to accept or reject a particle can
depend on additional characteristics attached to the particle (e.~g., mass),
it can depend on time (in particular, after a certain time, a particle may
be desorbed) and it can even be random. As an example, we consider the
\emph{spatial birth--growth model}: the balls attached to subsequent independently
time-marked points, i.~e., particles, are allowed to have their initial radii bounded
random i.i.d.\ rather than fixed. Moreover, at the moment of its birth each
particle begins to \emph{grow} radially with constant speed $v$ until it hits
another particle or reaches a certain maximal admissible size $ \rho $ -- in both
these cases it stops growing. In analogy to the basic RSA, a particle is accepted
if it does not overlap any previously accepted one and is rejected otherwise.

The mark of a point now consists of the time stamp plus the initial radius
of the corresponding ball. The functional of interest is
again given by $\xi(\marked{x},\marked{\X}) = 1$ if the particle centered at $x$
has been accepted and $0$ otherwise.
This model, going also under the
name of the \emph{Johnson--Mehl growth process} in the particular case where
the initial radii are $ 0 $, has attracted a lot of interest in the
literature, see \cite{PY2,BY2} and the references therein.

Let $ \marked{\X} $ be a configuration of marked particles, where a random mark
consists of a pair $ (t, s) $, where $ t $ is the time stamp and where $ s $
is some additional feature of the point. As suggested in \cite{PY2}, consider an
oriented graph with vertex set $ \marked{\X} $,
where an edge from $ \marked{x} $ to $ \marked{y} $ exists if the particle
$ x $ has arrived before $ y $ and if $ \| x - y \| \le \rho $.
Given $ \marked{x} \in \marked{\X} $, denote by
$ \markedAin_\rho(\marked{x}, \marked{\X}) $ the set
of all particles in $ \marked{\X} $ from which $ \marked{x} $ can be reached
by a directed path in this graph, along with $ \marked{x} $ itself.

Some thought shows that the functional $ \xi $ considered in the basic
RSA model as well as in the spatial birth--growth model is confined to
the functional $ \markedAin_\rho $ according to Definition~\ref{confin}.
Moreover, this is true for all examples
considered in \cite{PY2}: the key point is that particles are only influenced
by the configuration at their arrival, but not by the particles arriving later.

Now let $ \xi $ be any geometric functional confined to $ \markedAin_\rho $.
Denoting $ D(\marked{x}, \marked{\X}) := \sup \{ \| y - x \| \} \sth y \in \X \} $,
observe that the functional:
$$
 R(\marked{x}, \marked{\X}) := D(\marked{x}, \markedAin_\rho(\marked{x}, \marked{\X})) + \rho
$$
is a radius of stabilization for $ \markedAin_\rho $ and, according to Remark~\ref{StabConfin},
also for $ \xi $. Moreover, letting $ R_\la(\marked{x}, \marked{\X}) :=
R \bigl( \la^{1/d} \marked{x}, \la^{1/d} \marked{\X} \bigr) $,
observe that $ \la^{-1/d} R_\la $ is a radius of stabilization for
$ \xi_\la $.

Now let $ \marked{\Po}_f $ be a marked Poisson process with bounded intensity $ f $;
for the random marks (i.~e., the probability measure $ \P_\Marks $), assume that the
time stamp is continuously distributed
(without loss of generality, we may then assume that it is uniform over $ [0, 1] $).
Percolation estimates (Section~4 of \cite{PY2}) then yield the bound:
\begin{equation}
\label{eq:RSA:ExpDec}
 \P(D(\marked{x}, \markedAin_\rho(\marked{x}, \marked{\Po}_f)) \ge t) \le a \, e^{-bt}
\end{equation}
for all $ \marked{x} \in \marked{\R}^d $ and all $ t \ge 0 $, where the constants
$ a \ge 0 $ and $ b > 0 $ depend only on $ \| f \| $ and $ \rho $. As a result, $ \xi $
is $ \ka $-almost exponentially stabilizing for any bounded density $ \ka $.

\noindent
This puts us into the position to formulate the following result:

\begin{theo}
\label{RSA}
Let $ \rho > 0 $ and let $ \xi $ be a geometric functional on marked
points, where the marks along with $ \P_\Marks $ are as described above.
Suppose that $ \xi $ is confined to $ \markedAin_\rho $ and that
$ \ka $ satisfies Assumption~$\Dens$. Take $ \la_0 > 0 $.
\begin{enumerate}
\item[$(1)$] If the family $ (\xi_\la)_{\la > \la_0} $ satisfies
Assumption~$\MomlaOne(p, \ka)$ for some $ p > 2 $, then
$ \xi $ satisfies
\newline
Assumptions~$\ConvergVar(\ka)$. Consequently,
\eqref{eq:Varlim} holds.
\item[$(2)$] Let $ 0 < \tau < \infty $. Suppose that the family
$ (\xi_\la)_{\la > \la_0} $ satisfies Assumption~$\MomHomOne(p, \tau, \Omega)$
for some $ p > 2 $ and some convex domain $ \Omega $ with $ 0 < \vol(\Omega) < \infty $.
Next, suppose that for notably many triples $ (t, s, \marked{\X}) $,
we have $ \xi \bigl( (\0, t, s), \marked{\X} \bigr) \ne 0 $
and $ t > \max \{ t' \sth (x', t', s') \in \marked{\X} \} $.
Then, with $ \Vxi $ as in \eqref{eq:Vxi}, we have $ \Vxi(\tau) > 0 $.
\item[$(3)$] Let $ \alpha \ge 0 $. If $ \xi $ satisfies Assumption~$\MomGrInt(\alpha, \ka)$,
then it satisfies Assumption~$\Gen(1 + \alpha + d, \ka)$; if $ \xi $ satisfies
Assumption~$\MomGrPt(\alpha, \ka)$, it satisfies Assumption~$\Gen(\max \{ \alpha, 1 \} + d, \ka)$.
Consequently, the conclusions of Theorems~\ref{probabilities}, \ref{MDPTWI} and \ref{MDPmeasure}
hold with suitable $ \gamma $.
\end{enumerate}
\end{theo}

\begin{proof}
\textit{(1)}: From the exponential bound~\eqref{eq:RSA:ExpDec}, it follows that
the family $ (R_\la)_{\la > \la_0} $ satisfies Assumption~$\Momla(q, \ka)$ for all
$ q \ge 0 $. Similarly, $ R $ satisfies Assumption~$\FinHom(\tau)$ for all $ \tau > 0 $.
As a result, $ \xi $ satisfies
Assumptions~$\ConvergVar(\ka)$.

\medskip
\textit{(2)}: Again from the bound~\eqref{eq:RSA:ExpDec}, it follows that the family
$ (R_\la)_{\la > \la_0} $ satisfies Assumption~$\MomHom(q, \tau, \Omega)$ for all
$ q \ge 0 $. Thus, by Theorem~\ref{NDeg}, it suffices to show that $ \Delta \bigl( (\0, t, s),
\marked{\X} \bigr) \ne 0 $ and that $ \marked{\X} $ is externally stable at $ (\0, t, s) $
with respect to $ \xi $ provided that $ \marked{\X} $ is finite,
$ \xi \bigl( (\0, t, s), \marked{\X} \bigr) \ne 0 $,
$ t > \max \{ t' \sth (x', t', s') \in \marked{\X} \} $ and, in addition,
without loss of generality, $ (\0, t, s) \notin \marked{\X} $.
Since the time stamp $ t $ is the largest of all, we have
$ \markedAin_\rho(\marked{x}, \marked{\X}) = \markedAin_\rho(\marked{x},
\marked{\X} \cup \{ (\0, t, s) \}) $. As $ \xi $ is confined to
$ \markedAin_\rho $, we also have
$ \xi(\marked{x}, \marked{\X}) = \xi(\marked{x}, \marked{\X} \cup
\{ (\0, t, s) \}) $ for all $ \marked{x} \in \marked{\X} $, so that
$ \Deltaxi \bigl( (\0, t, s), \marked{\X} \bigr)
= \xi \bigl( (\0, t, s), \marked{\X} \bigr) \ne 0 $.
Moreover, letting $ r := \max_{\marked{x} \in \marked{\X}} \| x \| + \rho $,
observe that $ \marked{\X} $
is $ r $-externally stable at $ (\0, t, s) $ with respect to $ \xi $. This proves the
desired assertion.

\medskip
\textit{(3)}: It suffices to observe that $ \ka $-uniform exponential decay
of the family $ (R_\la)_{\la > \la_0} $ implies Assumption~$\MomGrInt(d, \ka)$.
\end{proof}

Now we return to our two examples, the RSA and the spatial birth--growth model.
As $ \xi $ is then bounded, the family $ (\xi_\la)_{\la > \la_0} $ satisfies
$ (\xi_\la)_{\la > \la_0} $ satisfies Assumptions~$\MomlaOne(p, \ka)$ and
$\MomHomOne(p, \tau, \Omega)$ for all $ p \ge 0 $, all
$ 0 < \tau < \infty $ and all suitable domains $ \Omega $.
Since the first particle is always accepted, we have
$ \xi \bigl( (\0, s, t), \emptyset \bigr) = 1 \ne 0 $. As a result,
the limiting variance is non-degenerate, i.~e., $ \Vxi(\tau) > 0 $ for all $ \tau > 0 $.
Finally, the family $ (\xi_\la)_{\la \ge \la_0} $ satisfies Assumption~$\MomGrPt(0, \ka)$
and therefore Assumption~$\Gen(d + 1)$. Thus, the conclusions of
Theorems~\ref{probabilities}, \ref{MDPTWI} and \ref{MDPmeasure} hold with
$ \gamma = d + 1 $.

Compared to the results in \cite{BESY}, there are three general novelties: first, we provide
more explicit bounds in Theorem~\ref{probabilities}. Second, we consider a much
more general class of geometric functionals. Third, we consider a broader
class of intensities $ \ka $: in particular, they need not have bounded support
(in contrast to Theorem~\ref{NN} where bounded support is required because the
density has to be bounded away from zero). Thus, in the basic RSA model, our present
results add to existing central limit theorems \cite{BY1, BY2, DR, PY2}, weak laws of
large numbers \cite{CFJP, PY2, PY4} and large deviations \cite{BESY, SY} for random
packing functionals.

Regarding the spatial birth--growth model, note that the paper \cite{BESY} only
succeeds to treat this model under an unnatural positive lower bound for initial
particle sizes, which excludes for instance the crucial Johnson-Mehl set-up. Here
this condition is no longer required. Our present results add to existing central
limit theorems \cite{BY2,CQ,Pe2007CLT,PY2} as well as to the large deviation principle \cite{SY}.

\subsection{Nearest Neighbors}
\label{ssc:NG}

Let $ \marked{\X} $ be a locally finite point configuration in $ \marked{\R}^d $. Take
$ \marked{x} \in \marked{\X} $ and $ k \in \N $. We define the set of $ k $ nearest
neighbors of $ \marked{x} $ in $ \marked{\X} $ to be
the set of all $ \marked{y} \in \marked{\X} \setminus \{ \marked{x} \} $, such that
$ \| z - x \| < \| y - x \| $ for strictly less than $ k $ points $ \marked{z} \in
\marked{\X} \setminus \{ \marked{x} \} $. Thus, if
$ \marked{\X} $ consists of a point $ \marked{x} $, a point $ \marked{y} $ with
$ \| y - x \| = 1 $, two more points $ \marked{z}, \marked{w} $ with $ \| z - x \| =
\| w - x \| = 2 $ and possibly some more points with the distance to $ x $ strictly
larger than $ 2 $, the set of two nearest neighbors of $ \marked{x} $ in $ \marked{\X} $
actually consists of three points: $ \marked{y} $, $ \marked{z} $ and $ \marked{w} $.

Let $ k \in \N $. Define the two set-valued geometric functional $ \NND{k} $ and $ \NN{k} $
as follows: let $ \NND{k}(\marked{x}, \marked{\X}) $ be the set consisting of $ \marked{x} $
and the set of $ k $ nearest points of $ \marked{x} $ in $ \marked{\X} $; let
$ \NN{k}(\marked{x}, \marked{\X}) $ be the union of $ \NND{k}(\marked{x}, \marked{\X}) $
plus the set of all $ \marked{y} \in \marked{\X} $, such that $ \marked{x} $ is among the
$ k $ nearest neighbors of $ \marked{y} $ in $ \marked{\X} $. We shall consider geometric
functionals confined to $ \NND{k} $ or $ \NN{k} $ according to Definition~\ref{confin}.
Notice that by Remark~\ref{ConfinSubs}, any geometric functional confined to
$ \NND{k} $ is also confined to $ \NN{k} $.

Fix a domain $ \Omega \subseteq \R^d $. The construction
of a radius of stabilization for $ \NN{k} $ inside $ \Omega $ is well-known.
Following \cite{Pe2007CLT}, consider a collection $ C_1, \ldots , C_s $
of infinite open cones with angular radius $ \pi/12 $ and apex at $ \0 $, with
union $ \R^d \setminus \{ \0 \} $. Let $ C_i^+ $ be the open cone concentric to $ C_i $ and
with angular radius $ \pi/6 $. For a configuration $ \marked{\X} \subset \marked{\Omega}
= \Omega \times \Marks $ and $ \marked{x} \in \marked{\X} $, define
$ R^{k,\Omega,i}(\marked{x}, \marked{\X}) $ to be the distance from $ x $ to its
$ k $-th nearest point in $ \marked{\X} \cap (C_i^+ + x) $ if such a point exists and
this distance is less than $ \diam((C_i + x) \cap \Omega) $; otherwise,
set $ R^{k,\Omega,i}(\marked{x}, \marked{\X}) := \diam((C_i + x) \cap \Omega) $.

Let $ R^{k,\Omega}(\marked{x}, \marked{\X}) := \max_i R^{k,\Omega,i}(\marked{x}, \marked{\X}) $.
From elementary geometry (see \cite{Pe2007CLT}), it follows that if
$ \marked{\X} \subset \marked{\Omega} $, $ \marked{x} \in \marked{\X} $ and
$ \| y - x \| > R^{k,\Omega}(\marked{x}, \marked{\X}) $,
then neither $ \marked{y} \in \NND{k}(\marked{x}, \marked{\X}) $ nor
$ \marked{x} \in \NND{k}(\marked{y}, \marked{\X}) $.
As a result, $ R^{k,\Omega} $ is a radius of stabilization inside $ \Omega $
for $ \NN{k} $ as well as for any geometric functional confined to $ \NN{k} $.

If $ 0 < \tau < \infty $, then the homogeneous Poisson process $ \Po_\tau $ almost
surely contains infinitely many points in every cone $ C_i $. Therefore,
for $ \Omega = \R^d $, $ R^{k, \Omega} $ satisfies $ \FinHom(\tau) $.

Now assume that $ \Omega $ is bounded and convex with $ \vol(\Omega) > 0 $.
Take $ \la > 0 $ and let $ \ka $ be a probability density function
vanishing outside $ \Omega $, but with $ \inf_{x \in \Omega} \ka(x) > 0 $.
We will show that $ R^{k,\Omega} $ enjoys super-exponential tail decay.
Basically, we follow \cite{Pe2007CLT}, but it turns out that one has to
be a bit more careful. Take $ \marked{x} \in \marked{\Omega} $ and $ i = 1, \ldots, s $.
It is easy to see that if $ R^{k,\Omega,i}(\marked{x}, \marked{\Po}_{\la \ka}) > \rho $, then
there are less than $ k $ points in $ \marked{\Po}_{\la \ka} \cap B_\rho(x) \cap (C_i^+ + x) $,
but also at least one point $ y \in \Omega \cap (C_i + x) $
with $ \| y - x \| > \rho $. By convexity, there also exists a point
$ z \in \Omega \cap (C_i + x) $ with $ \| z - x \| = \rho/2 $.
Setting $ \eta := \sin \frac{\pi}{12} $, we have
$ B_{\eta \rho/2}(z) \subseteq B_\rho(x) \cap (C_i^+ + x) $.

The continuation of the argument in \cite{Pe2007CLT} works provided that 
$ B_{\eta \rho/2}(z) \subseteq \Omega $, but this
is not necessarily true. However, letting $ D := \diam(\Omega) $, we have
$ \Omega' := \bigl\{ \bigl( 1 - \frac{\eta \rho}{2D} \bigr) z +
\frac{\eta \rho}{2D} \, w \sth w \in \Omega \bigr\} \subseteq \Omega $ as well
as $ \Omega' \subseteq B_{\eta \rho/2}(z) \subseteq
B_\rho(x) \cap (C_i^+ + x) $. Therefore, the set $ \Omega' \times \Marks $
contains less than $ k $ points in $ \marked{\Po}_{\la \ka} $.
Since $ \vol(\Omega') = \bigl( \frac{\eta \rho}{2D} \bigr)^d \vol(\Omega) $ and
$ \Omega' \subseteq \Omega $, the probability that $ \Omega' \times \Marks $ contains
less than $ k $ points in $ \marked{\Po}_{\la \ka} $ is bounded from above by:
$$
 \sum_{l=0}^{k-1} \frac{1}{l!} \left[ \la m \left( \frac{\eta \rho}{2D} \right)^d \vol(\Omega) \right]^l
  \exp \left[ - \la m \left( \frac{\eta \rho}{2D} \right)^d \vol(\Omega) \right] \, ,
$$
where $ m = \inf_{x \in \Omega} \ka(x) $.
This is also an upper bound on $ \P(R^{k,\Omega,i}(\marked{x}, \marked{\Po}_{\la \ka}) > \rho) $.
Consequently, there exist $ a \ge 0 $ and $ b > 0 $ depending only on $ \Omega $, $ \ka $ and
$ k $, such that $ \P(R^{k,\Omega}(\marked{x}, \marked{\Po}_{\la \ka}) \ge \rho)
\le a \, e^{-b \la \rho^d} $ for all $ \rho \ge 0 $, all $ \la > 0 $ and all
$ \marked{x} \in \marked{\Omega} $.

Recalling that $ \xi_\la(\marked{x}, \marked{\X}) = \xi \bigl( \la^{1/d} \marked{x},
\la^{1/d} \marked{\X} \bigr) $, observe that if $ \xi $ is confined to $ \NN{k} $, then
$ \xi_\la $ is also confined to $ \NN{k} $. Therefore,
$ R^{k,\Omega} $ is a radius of stabilization for $ \xi_\la $ inside $ \Omega $,
so that we can set $ R_\la(\marked{x}, \marked{\X}) := \la^{1/d} R^{k,\Omega}(\marked{x}, \marked{\X}) $.
Then we have $ \P(R_\la(\marked{x}, \marked{\Po}_{\la \ka}) > \rho) \le a \, e^{-b \rho^d} $,
with $ a $ and $ b $ uniform in $ \marked{x} $, $ \la $ and $ \rho $. Consequently,
\begin{equation}
\label{eq:NG:SuperExpDec}
\begin{split}
 \E \bigl( R_\la(\marked{x}, \marked{\Po}_{\la \ka}) \bigr)^j
  &= \int_0^\infty j \, \rho^{j-1}
  \P \bigl( R_\la(\marked{x}, \marked{\Po}_{\la \ka}) > \rho \bigr) \,\dl \rho
  \le a j \int_0^\infty \rho^{j-1} \, e^{- b \rho^d} \,\dl \rho =
\\
 &= \frac{a j}{d} \int_0^\infty t^{j/d-1} \, e^{- bt} \,\dl t
  = \frac{aj}{b^{j/d} d} \, \Gamma \left( \frac{j}{d} \right)
  \le B^k (j!)^{k/d}
\end{split}
\end{equation}
for some $ B $ depending only on $ a $ and $ b $. Thus,
the family $ (R_\la)_{\la > 0} $ satisfies Assumption~$\MomGrInt(1/d, \ka)$
(and, since $ R_\la(\marked{x}, \marked{\Y}) \le R_\la(\marked{x}, \marked{\X}) $
for $ \marked{\Y} \supseteq \marked{\X} $, even Assumption~$\MomGrPt(1/d, \ka)$).
This puts us into the position to formulate the following result:

\begin{theo}
\label{NN}
Let $ k \in \N $ and let $ \xi $ be a geometric functional confined to
$ \NN{k} $. Take a convex bounded domain $ \Omega $ and
$ \ka $ satisfying Assumption~$\Dens$ and with
$ \inf_{x \in \Omega} \ka(x) > 0 $. Let $ \la_0 > 0 $ and
let the cones $ C_i $ be as above.
\begin{enumerate}
\item[$(1)$] If the family $ (\xi_\la)_{\la > \la_0} $ satisfies
Assumption~$\MomlaOne(p, \ka)$ for some $ p > 2 $, then
$ \xi $ satisfies
\newline
Assumptions~$\ConvergVar(\ka)$. Consequently,
\eqref{eq:Varlim} holds.
\item[$(2)$] Let $ 0 < \tau < \infty $. Suppose that the family
$ (\xi_\la)_{\la > \la_0} $ satisfies Assumption~$\MomHomOne(p, \tau, \Omega)$
for some $ p > 2 $. Next, suppose that notably many pairs $ (t, \marked{\X}) $
satisfy the following two conditions: first,
$ \Deltaxi \bigl( (\0, t), \marked{\X} \bigr) \ne 0 $ (recalling the definition
of the add-one cost from Subsection~\ref{ssc:NDeg}); second,
there exists $ \rho > 0 $, such that for each cone $ C_i $, each of the
sets $ \marked{\X} \cap C_i \cap B_\rho(\0) $ and
$ (\marked{\X} \cap C_i) \setminus B_{\rho/\eta}(\0) $
contains at least $ k $ points; here, $ \eta = \sin \frac{\pi}{12} $.
Then, with $ \Vxi $ as in \eqref{eq:Vxi}, we have $ \Vxi(\tau) > 0 $.
\item[$(3)$] Let $ 0 < \tau < \infty $. Suppose that $ \xi $ is confined
to $ \NND{k} $ and that the family
$ (\xi_\la)_{\la > \la_0} $ satisfies Assumption~$\MomHomOne(p, \tau, \Omega)$
for some $ p > 2 $. Next, suppose that notably many pairs $ (t, \marked{\X}) $
satisfy the following two conditions: first,
$ \xi \bigl( (\0, t), \marked{\X} \bigr) \ne 0 $; second,
there exists $ \rho > 0 $, such that $ \marked{\X} \cap B_\rho(\0) = \emptyset $
and such that each intersection $ \marked{\X} \cap B_{\rho \sqrt{3}}(\0) \cap C_i $
contains at least $ k+1 $ points. Then we have $ \Vxi(\tau) > 0 $.
\item[$(4)$] Let $ \alpha \ge 0 $. If $ \xi $ satisfies Assumption~$\MomGrInt(\alpha)$,
then it satisfies Assumption~$\Gen(2 + \alpha)$; if $ \xi $ satisfies
Assumption~$\MomGrPt(\alpha)$, it satisfies Assumption~$\Gen(\max \{ \alpha, 1 \} + 1)$.
Consequently, the conclusions of Theorems~\ref{probabilities}, \ref{MDPTWI} and \ref{MDPmeasure}
hold with suitable $ \gamma $.
\end{enumerate}
\end{theo}

\begin{proof}
\textit{(1)}: From \eqref{eq:NG:SuperExpDec}, it follows that the family
$ (R_\la)_{\la > \la_0} $ satisfies Assumption~$\Momla(q, \ka)$ for all $ q \ge 0 $.
As a result, $ \xi $ satisfies Assumptions~$\ConvergVar(\ka)$.

\medskip
\textit{(2)}: Similarly as before, we find that the family
$ (R_\la)_{\la > \la_0} $ satisfies Assumption~$\MomHom(q, \tau, \Omega)$ for all $ q \ge 0 $.
By Theorem~\ref{NDeg}, it remains to show that for each $ t \in \Marks $,
any finite configuration $ \marked{\X} \subset \marked{\R}^d \setminus \{ (\0, t) \} $
satisfying the specified conditions is externally stable at $ (\0, t) $ with respect to $ \xi $.
First, take a finite configuration $ \marked{\Y} $ with $ \marked{\Y} \cap B_\rho(\0)
= \marked{\X} \cap B_\rho(\0) $ and observe first
that no point in $ \marked{\Y} \setminus B_\rho(\0) $ is
among the $ k $ nearest neighbors of $ (\0, t) $ in $ \marked{\Y} \cup \{ (\0, t) \} $;
similarly, $ (\0, t) $ is not among the $ k $ nearest neighbors of any point
in $ \marked{\Y} \setminus B_\rho(\0) $. Therefore, $ \marked{\X} $ is
basically $ \rho $-externally stable at $ (\0, t) $ with respect to $ \xi $.

Next, let $ x \in B_\rho(\0) $, $ y \in C_i \setminus B_{\rho/\eta}(\0) $ and take
$ u $ on the axis of $ C_i $, inside $ C_i $. Then the angle between
$ y $ and $ u $ is smaller than $ \pi/12 $. By elementary geometry,
the angle between $ y $ and $ y - x $ is also smaller than $ \pi/12 $,
for $ \| x \|/\| y \| < \eta $.
Consequently, the angle between $ y - x $ and $ u $ is smaller than
$ \pi/6 $. As a result, we have $ C_i \setminus B_{\rho/\eta}(\0) \subseteq
(C_i^+ + x) $ for all $ x \in C_i \cap B_\rho(\0) $. Thus, taking
$ \marked{x} \in (\marked{\X} \cap B_\rho(\0)) \cup \{ (\0, t) \} $,
each of the sets $ (C_i^+ + x) \cap \marked{\X} $ contains
at least $ k $ points, so that $ \xi $ stabilizes with respect to $ \marked{\X} $
and $ \marked{\X} \cup \{ (\0, t) \} $ at all
$ \marked{x} \in (\marked{\X} \cap B_\rho(\0)) \cup \{ (\0, t) \} $.
By Proposition~\ref{StabExtStab}, $ \marked{\X} $ is externally stable
at $ (\0, t) $ with respect to $ \xi $. This proves the desired assertion.

\medskip
\textit{(3)}: Set $ A := B_{\rho \sqrt{3}}(\0) \setminus B_\rho(\0) $.
Similarly as in the preceding point, it suffices to show that
for each pair $ (t, \marked{\X}) $ satisfying the specified conditions,
we have $ \Delta \bigl( (0, t), \marked{\X} \bigr) \ne 0 $
and $ \marked{\X} $ is externally stable at $ (\0, t) $ with respect to $ \xi $.
First, by elementary geometry, we have
$ \diam( C_i \cap A) = \rho $. Therefore, $ (\0, t) \notin
\NND{k}(\marked{y}, \marked{\X} \cup \{ (\0, t) \}) $
for all $ \marked{y} \in (C_i \cap A) \times \Marks $ and therefore for all
$ \marked{y} \in A \times \Marks $. However, this is also true if
$ \marked{y} \notin B_{\rho \sqrt{3}}(\0) \times \Marks $ and therefore
for all $ \marked{y} \in \marked{\X} $. Since
$ \xi $ is confined to $ \NND{k} $, we then have
$ \xi( \marked{y}, \marked{\X} \cup \{ (\0, t) \})
= \xi( \marked{y}, \marked{\X}) $. As a result,
$ \Deltaxi \bigl( (\0, t), \marked{\X} \bigr) =
\xi \bigl( (\0, t), \marked{\X} \bigr) \ne 0 $.

Now take a finite configuration $ \marked{\Y} $ with
$ \marked{\Y} \cap B_{\rho \sqrt{3}}(\0) = \marked{\X} \cap B_{\rho \sqrt{3}}(\0) $.
Similarly as above, we find that for all $ \marked{z} \in \marked{\Y} \setminus
B_{\rho \sqrt{3}}(\0) $, we have $ (\0, t) \notin
\NND{k}(\marked{y}, \marked{\Y} \cup \{ (\0, t) \}) $. Therefore,
$ \marked{\X} $ is basically $ \rho \sqrt{3} $-externally stable at $ (\0, t) $
with respect to $ \xi $.
Moreover, since $ C_i \cap A \cap \marked{\X} $ contains at least $ k+1 $ points
and since $ \diam( C_i \cap A) = \rho $ for all $ i $, we have
$ \NND{k}(\marked{y}, \marked{\X} \cup \{ (\0, t) \}) \subseteq
B_\rho(y) \times \Marks $ for all $ \marked{y} \in \marked{\X} \cap B_{\rho \sqrt{3}}(\0)
= \marked{\X} \cap A $. Therefore, $ \xi $ stabilizes at $ \marked{y} $ within radius
$ \rho $ with respect to $ \marked{\X} $ as well as to
$ \marked{\X} \cup \{ (\0, t) \} $.
By Proposition~\ref{StabExtStab}, $ \marked{\X} $ is externally stable
at $ (\0, t) $ with respect to $ \xi $. This proves the desired assertion.

\medskip
\textit{(4)}: This follows immediately from the fact that
the family $ (R_\la)_{\la > 0} $ satisfies Assumption~$\MomGrInt(1/d, \ka)$.
\end{proof}

Theorem~\ref{NN} adds to the existing results on non-degeneracy of the limiting
variance (see \cite{PY1}), central limit theorems (see Chapter~4 of \cite{Pe} as
well as \cite{BY2}) and, of course, large deviation results.
For the latter, observe that the paper \cite{BESY} was only able to deal with
the \emph{empirical functionals of nearest neighbors graphs}, where
$ \xi(x, \X) $ is the indicator of the event that the total edge length
exceeds a certain threshold, or of some event involving 
the degree of the graph at $x$ and possibly also the edge length,
such as `the total length of edges incident to $x$ exceeds a certain
multiplicity of the graph degree of $x$' etc. No marks have been considered
in \cite{BESY}. Our result includes a broader class of intensities $ \ka $ and a
much more general collection of geometric functionals. The following example
serves as a classical one.

\begin{exmpl}
\label{NN:Length}
Define $ \xi(x, \X) $ to be the sum of the distances from $ x $ to its
$ k $ nearest neighbors (there are no marks). Then $ \xi $ is obviously
confined to $ \NND{k} $. Recalling that $ \Omega $ is a bounded convex domain
with $ \vol(\Omega) > 0 $, that $ \inf_\Omega \ka > 0 $ and that $ \ka $ vanishes
outside $ \Omega $, a similar argument as the one used for the radius
of stabilization shows that the family $ (\xi_\la)_{\la > \la_0} $ satisfies
Assumption~$\MomGrPt(1/d, \ka)$, as well as Assumptions~$\MomlaOne(p, \ka)$
and $\MomHomOne(p, \tau, \Omega) $ for all $ p, \tau > 0 $. Therefore,
$ \xi $ satisfies Assumptions~$\ConvergVar(\ka)$ and, more importantly,
Assumptions~$\Gen(2, \ka)$. This means that the range where moderate deviation
results apply is independent of the dimension. Finally, it is obvious
that for all non-empty configurations $ \X $, we have $ \xi(\0, \X) \ne 0 $.
As a result, the limiting variance is non-degenerate,
i.~e., $ \Vxi(\tau) > 0 $ for all $ \tau > 0 $.

With a little extra effort, one can prove the same results for the `undirected'
case, that is, for the sum of all distances from $ x $ to the points in $ \NN{k}(x, \X) $
(this is twice the total edge length of the undirected $ k $ nearest neighbors graph).
However, for the non-degeneracy of the limiting variance, one has to refer to
part~$(2)$ rather than to part~$(3)$ of Theorem~\ref{NN}. This requires a bit more
involved argument, which we shall not provide here because the non-degeneracy
is proved explicitly in \cite{PW}; it can also be deduced from the earlier paper
\cite{PY1}.
\end{exmpl}

Curiously, the paper \cite{PW} provides no explicit application of its general
result on non-degeneracy of the limiting variance, i.~e., Theorem~2.2 ibidem.
For the $ k $ nearest neighbors, the limiting variance is computed explicitly.
Therefore, we give another example, where we prove non-degeneracy, but the
construction from the earlier paper \cite{PY1} does not work.

\begin{exmpl}
\label{NN:ExpLength}
Define $ \xi(x, \X) := \exp \bigl( - \sum_{y \in \NN{k}(x, \X)} \| y - x \| \bigr) $.
Turning first to large deviations, observe that $ \xi $ is bounded. Therefore,
taking $ \Omega $ and $ \ka $ as before, $ \xi $ satisfies Assumptions~$\ConvergVar(\ka)$
as well as Assumptions~$\Gen(2, \ka)$. Thus, we obtain just the same range of
moderate deviation results as in the previous example.

Now we turn to the limiting variance. Consider first a configuration $ \X $
containing some point $ y \ne \0 $ and no point in $ B_{\| y \|}(\0) \setminus
\{ y \} $. Letting $ A_\rho := B_{\rho \sqrt{3}}(\0) \setminus B_\rho(\0) $, assume
also that each intersection $ \X \cap C_i \cap A_{\| y \|} $ contains at least two points.
Then $ y $ is the nearest neighbor of $ \0 $. Moreover, similarly as
in the proof of Theorem~\ref{NN}~$(3)$, we find that $ \0 $ is not the nearest
neighbor of any point in $ \X $. Therefore, we have
$ \Deltaxi(\0, \X) = e^{- \| y \|} - (1 - e^{- \| y \|}) \xi(y, \X) $.

Now take a finite configuration $ \X $, which satisfies the condition from
the preceding paragraph. If, in addition, $ \| y \| < \log 2 $, then,
clearly, $ \Delta(\0, \X) > 0 $. Moreover, taking $ r := \| y \| \sqrt{3}/\sin \frac{\pi}{12} $,
assume that each of the sets $ (\X \cap C_i) \setminus B_r(\0) $ is non-empty.
One can easily check that notably many configurations $ \X $ satisfy
this condition. From Theorem~\ref{NN}~$(2)$, it then follows that
$ V(\tau) > 0 $ for all $ \tau > 0 $.

In \cite{PY1}, the argument used to show non-degeneracy of the limiting variance
requires, among others, that $ \Deltaxi $ is stabilizing. A relatively simple
construction of a radius of stabilization for $ \Deltaxi $,
much similar to the above-mentioned construction of radius of stabilization
for $ \xi $, is provided for the total edge length
of the nearest neighbor graph (considered also in Example~\ref{NN:Length}) in
the plane ($ d = 2 $); a much more complicated construction is used for the number
of components. Here, we demonstrate that the construction used in \cite{PY1} for
the total edge length does not work here.

In \cite{PY1}, the radius of stabilization at the origin is obtained by means of six
disjoint equilateral triangles, such that the origin is a vertex of each triangle. If
this construction works for some geometric functional, one can also take a covering
of $ \R^2 \setminus \{ \0 \} $ by a family of open angles. Clearly, one can assume
that their measures are at most $ \pi/6 $. Now take any family $ C'_1, \ldots, C'_m $
of open cones with angular radii at most $ \pi/3 $, $ k \ge 2 $, $ a > 1 + \sqrt{3}/3 $
and define $ \RDeltaxi(x, \X) $ to be $ a $ times the minimal $ \rho $, such that any
set $ \X \cap (C'_i + x) $ contains at least $ k $ points. Below we show that
for any $ \tau > 0 $, the probability that $ \RDeltaxi $ is not a radius of stabilization
for $ \Deltaxi $ is strictly positive.

Take a configuration $ \X $ containing a point $ y \in C'_1 $, no point in
$ B_{\| y \|}(\0) \setminus \{ y \} $, some point $ z \in C_1 $ with
$ \| z \|/\| y \| > a \sqrt{3} $, no point in $ B_{\| z - y \|}(z) \setminus \{ y \} $
and no point in $ B_{a \| y \| \sqrt{3}}(\0) \setminus B_{\| y \| \sqrt{3}}(\0) $.
Defining $ A_\rho $ as before, assume also that each of the sets
$ \X \cap A_{\| y \|} \cap C_i $ contains at least $ k+1 $ points. Some thought shows
that such a configuration occurs with non-zero probability in any homogeneous Poisson
point process $ \Po_\tau $.

Letting $ \Y := \X \cap B_{a \| y \| \sqrt{3}}(\0) $, observe that for any point
in $ \Y $, there exists another point in the same set within distance $ \| y \| $.
Moreover, since $ a > 1 + \sqrt{3}/3 $ and since $ \X \cap \bigl(
B_{a \| y \| \sqrt{3}}(\0) \setminus B_{\| y \| \sqrt{3}}(\0) \bigr) = \emptyset $,
there exists no point in $ \X \setminus B_{a \| y \| \sqrt{3}}(\0) $ within distance
$ \| y \| $. Therefore, the nearest neighbor in $ \X \cup \{ \0 \} $ of any point
in $ \Y $ also lies in $ \Y $. As a result, $ \NN{1}(y, \X) \supseteq \NN{1}(y, \Y) $.
Moreover, the inclusion is proper because $ y $ is the nearest neighbor
of $ z $ in $ \X $. Therefore, $ \NN{1}(y, \X) < \NN{1}(y, \Y) $.
Since $ \Deltaxi(\0, \X) = e^{- \| y \|} - (1 - e^{- \| y \|}) \xi(y, \X) $ and
analogously for $ \Y $, we have $ \Deltaxi(\0, \X) > \Deltaxi(\0, \Y) $.
Therefore, $ \Delta $ does not stabilize at $ \0 $ within radius $ a \| y \| \sqrt{3} $,
nor does it stabilize within radius $ \RDeltaxi(\0, \X) $ because
$ \RDeltaxi(\0, \X) \le a \| y \| \sqrt{3} $.

Thus, we can conclude that the construction from \cite{PY1} does not work in
this case. However, this does not mean that $ \Deltaxi $ does not stabilize.
In fact, it almost surely stabilizes at $ \0 $ with respect to $ \Po_\tau $:
examining the proof of Theorem~\ref{NN}~$(2)$ and using the basic properties
of homogeneous Poisson point processes, we find that it provides an explicit
construction of an external radius of stabilization, which is almost surely
finite. By Remark~\ref{ExtStabDelta}, this is also a radius of
stabilization for $ \Delta $.
\end{exmpl} 

\subsection {Sphere of Influence Graphs}
\label{ssc:SIG}

Given a locally finite set $ \marked{\X} \subset \marked{\R}^d$, the
\emph{sphere of influence graph} $\SIG(\marked{\X})$ is a graph with vertex set
$\marked{\X}$ constructed as follows: for each $\marked{x} \in \marked{\X}$, let
$B(\marked{x}, \marked{\X})$ be a ball around $x$ with radius equal to
$\min_{\marked{y} \in \marked{\X} \setminus \{\marked{x}\}} \{\| y - x \| \}$
(in particular, the ball is degenerate if two points with different marks
share the same location). Then $B(\marked{x}, \marked{\X})$ is
called the \emph{sphere of influence} of $\marked{x}$. Draw an edge between $\marked{x}$
and $\marked{y}$ iff the balls $B(\marked{x}, \marked{\X})$ and $B(\marked{y}, \marked{\X}) $
overlap. The collection of such edges is the \emph{sphere of influence graph} (SIG) on
$\marked{\X}$ and is denoted by $\SIG(\marked{\X})$.

In \cite{PY1}, non-degeneracy of the limiting variance and central limit
theorems are derived for a variety of functionals, i.~e., the total number of edges,
the total edge length, the number of vertices of fixed degree and, most remarkably,
for the number of components. Except for the latter functional, these results are
extended in \cite{BY2} to random measures.

Here we shall only consider functionals confined to the functional $ \NSIG $,
where $ \NSIG(\marked{x}, \marked{\X}) $ denotes the set of points which are adjacent
to $ \marked{x} $ in $ \SIG(\marked{\X}) $ (including $ \marked{x} $). Notice that the
total number of edges, the total edge length and the number of vertices of fixed degree
can all be expressed in terms of suitable functionals $ \xi $ confined to $ \NSIG $,
while for the number of components, this seems not to be possible.
% In the sequel, we shall derive large deviation results for functionals confined
% to $ \NSIG $.

First, we turn to stabilization. A construction of a radius
of stabilization is given in \cite{PY1} and is also used in \cite{BY2}. However,
the results ibidem do not entirely fit the concept of stabilization
and external stabilization used here. In particular, they do not include a domain
$ \Omega $. Therefore, we here refine the construction in a similar way as in the
case of nearest neighbors. First, we rewrite the stabilization result from p.~1030 of
\cite{PY1}.

\begin{prop}
\label{SIG:Stab}
Let $ \marked{\X} \subset \R^d $ be a finite configuration. Take $ \marked{x}
\in \marked{\X} $, $ \rho > 0 $ and an open cone $ C $ in $ \R^d $ with angular
radius $ \pi/12 $ and apex at $ x $. Assume that the intersection
$ (\marked{\X} \setminus \{ \marked{x} \}) \cap B_\rho(x) $ is non-empty and
that there also exists a marked point
$ \marked{y} \in \marked{\X} \cap (C \setminus B_{3 \rho}(x)) $. Let
$ r = \| y - x \| $. Then no point in $ \marked{\X} \cap (C \setminus B_r(x)) $
is adjacent to $ \marked{x} $ in $ \SIG(\marked{\X}) $. Moreover, for any finite
configuration $ \marked{\Y} $ with
$ \marked{\Y} \cap B_{2r}(x) = \marked{\X} \cap B_{2r}(x) $,
we have $ \NSIG(\marked{x}, \marked{\Y}) \cap C = \NSIG(\marked{x}, \marked{\X})
\cap C $.
\end{prop}

\begin{proof}
Take $ \marked{z} \in \marked{\X} \cap (C \setminus B_{r}(x)) $ and let
$ z' := x + \frac{r}{\| z - x \|} $. Since $ y, z' \in C $, we have
$ \| z' - y \| < 2 \eta r $, where $ \eta = \sin \frac{\pi}{12} $ (but not necessarily
$ \| z' - y \| < r/2 $, as estimated in display $(7.4)$ of \cite{PY1}).
Therefore, $  \| z - y \| \le \| z - z' \| + \| z' - y \| < \| z - x \| - (1 - 2 \eta) r $,
so that $ B(\marked{z}, \marked{\X}) $ does not overlap with $ B_{(1 -  2 \eta)r}(\marked{x}) $.
Since $ (1 - 2 \eta) r > 3 (1 - 2 \eta) \rho > \rho $, it does not overlap with
$ B(\marked{x}, \marked{\X}) $ either.

Finally, if $ \marked{\X} \cap B_{2r}(\marked{x}) = \marked{\Y} \cap B_{2r}(\marked{x}) $, then
any two points in $ \marked{\X} \cap B_r(\marked{x}) = \marked{\Y} \cap B_r(\marked{x}) $ are
adjacent in $ \SIG(\marked{\X}) $ if and only if they are adjacent in $ \SIG(\marked{\Y}) $.
Combined with the above, this proves the result.
\end{proof}

%.%
%
% \begin{rem}
% Alternative: take $ C $ with angular radius $ \pi/6 $, $ y \in \X $ with $ \| y - x \| > 4 \rho $
% and $ z \in \X $ with $ \| z - x \| > 2 \| y - x \| $. Then we have
% $ \| y - x \|^2 + 2 \rho \| z - x \| < \| y - x \| \| z - x \| < \rho^2 + \| y - x \| \| z - x \| $,
% or, equivalently, $ \| y - x \|^2 - \| y - x \| \| z - x \| + \| z \|^2 < (\| z - x \| - \rho)^2 $.
% Since the angle between $ y - x $ and $ z - x $ is smaller than $ \pi/3 $, this implies
% $ \| y - z \|^2 < (\| z - x \| - \rho)^2 $ or, equivalently, $ \rho + \| y - z \| < \| z - x \| $.
% Hence $ B(x, \X) $ and $ B(z, \X) $ do not overlap.
% 
% We claim that $ (x', \X') \mapsto \NSIG(x', \X') \cap C $ stabilizes at $ x $ within radius
% $ 4 \| y - x \| $. Again, consider addition of a point $ z $ with $ \| z - x \| >
% 4 \| y - x \| $. Clearly, $ z $ cannot be adjacent to $ x $. Next, if $ w \in \X \cap C $
% is adjacent to $ x $, then $ \| w - x \| \le 2 \| y - x \| $.
% If $ \| z - x \| > 4 \| y - x \| $, then also $ \| z - w \| > 2 \| y - x \| \ge \| w - x \| $,
% so that the addition of $ z $ cannot shrink $ B(w, \X) $.
% \end{rem}

This allows us to construct a radius of stabilization inside a domain $ \Omega \subseteq \R^d $
in a similar way as in the case of nearest neighbors. Consider a collection $ C_1, \ldots , C_s $
of infinite open cones with angular radius $ \pi/24 $ and apex at $ \0 $, with
union $ \R^d \setminus \{ \0 \} $. Let $ C_i^+ $ be the open cone concentric to $ C_i $ and
with angular radius $ \pi/12 $. Take a configuration $ \marked{\X} \subset \marked{\Omega}
= \Omega \times \Marks $ and $ \marked{x} \in \marked{\X} $. Suppose that $ \marked{\X} $
contains at least one more point and denote by $ \rho $ the distance from $ x $ to the
nearest neighbor of $ \marked{x} $ in $ \marked{\X} $ (which equals zero if there is
another marked point at the same location). Next, suppose that the set
$ \bigl( \marked{\X} \cap (C_i^+ + x) \bigr) \setminus B_{3 \rho}(x) $ is
non-empty and denote by $ r $ the distance from $ x $ to its nearest neighbor in
$ \bigl( \marked{\X} \cap (C_i^+ + x) \bigr) \setminus B_{3 \rho}(x) $. Set 
$ R^{\Omega,i}(\marked{x}, \marked{\X}) := 2r $ if this construction works and
$ 2 r < \diam((C_i + x) \cap \Omega) $; otherwise,
set $ R^{\Omega,i}(\marked{x}, \marked{\X}) := \diam((C_i + x) \cap \Omega) $.
Let $ R^{\Omega}(\marked{x}, \marked{\X}) := \max_i R^{\Omega,i}(\marked{x}, \marked{\X}) $.
Proposition~\ref{SIG:Stab} and some thought show that $ R^\Omega $ is a radius of
stabilization inside $ \Omega $ for the functional $ \NSIG $.

If $ 0 < \tau < \infty $, then the homogeneous Poisson process $ \Po_\tau $ almost
surely contains a point in every cone $ C_i $ arbitrarily far from the origin.
Therefore, for $ \Omega = \R^d $, $ R^{\Omega} $ satisfies $ \FinHom(\tau) $.

Now assume that $ \Omega $ is bounded and convex with $ \vol(\Omega) > 0 $.
Take $ \la > 0 $ and let $ \ka $ be a probability density function
vanishing outside $ \Omega $, but with $ \inf_{x \in \Omega} \ka(x) > 0 $.
Again, we will show that $ R^{\Omega} $ enjoys super-exponential tail decay.
Take $ \marked{x} \in \marked{\Omega} $ and $ i = 1, \ldots, s $.
It is easy to see that if $ R^{\Omega,i}(\marked{x}, \marked{\Po}_{\la \ka}) > u $,
then, first, either the set $ \bigl( B_{u/9}(x) \cap (C_i^+ + x) \bigr) \times \Marks $
or the set $ \bigl( (B_{u/2}(x) \setminus B_{u/3}(x)) \cap (C_i^+ + x)) \bigr) \times \Marks $
contains no point in $ \marked{\Po}_{\la \ka} $, and, second there is at least one point in
$ y \in \Omega \cap (C_i + x) $ with $ \| y - x \| > u $.
Now let $ 0 \le \theta \le 1 $. By convexity, there also exists a point
$ z_\theta \in \Omega \cap (C_i + x) $ with
$ \| z_\theta - x \| = \theta u $. Setting $ \eps := \sin \frac{\pi}{24} $, we have
$ B_{\theta \eps u}(z_\theta) \subseteq \bigl( B_{\theta (1 + \eps) u}(x) \setminus
\mathring{B}_{\theta (1 - \eps) u}(x) \bigr) \cap (C_i^+ + x) $, where
$ \mathring{B}_r(x) $ denotes the open ball of radius $ r $ centered at $ x $.

Letting $ D := \diam(\Omega) $, we have $ \Omega_\theta :=
\bigl\{ \bigl( 1 - \frac{\theta \eps u}{D} \bigr) z_\theta + \frac{\theta \eps u}{D} \, w
\sth w \in \Omega \bigr\} \subseteq \Omega $ as well
as $ \Omega_\theta \subseteq B_{\theta \eps u}(z_\theta) \subseteq
\bigl( B_{\theta(1 + \eps) u}(x) \setminus \mathring{B}_{\theta(1 - \eps) u}(x) \bigr) \cap (C_i^+ + x) $.
In particular, routine calculation shows that $ \Omega_{1/11} \subseteq B_{u/9}(x) $
and $ \Omega_{2/5} \subseteq B_{u/2}(x) \setminus B_{u/3}(x) $.
Therefore, either $ \Omega_{1/11} \times \Marks $ or $ \Omega_{2/5} \times \Marks $
contains no points in $ \marked{\Po}_{\la \ka} $. Since $ \vol(\Omega_\theta)
= \bigl( \frac{\theta \eps u}{D} \bigr)^d \vol(\Omega) $ and
$ \Omega_\theta \subseteq \Omega $, the probability that $ \Omega_\theta \times \Marks $
contains no point in $ \marked{\Po}_{\la \ka} $ is bounded from above by
$ \exp \bigl[ - \la m \bigl( \frac{\theta \eps u}{D} \bigr)^d \vol(\Omega) \bigr] $,
where $ m = \inf_{x \in \Omega} \ka(x) $. Consequently,
\begin{equation}
\label{eq:SIG:Stab:SuperExp}
 \P \bigl( R^{\Omega,i}(\marked{x}, \marked{\Po}_{\la \ka}) > u \bigr)
  \le \exp \left[ - \la m \left( \frac{\eps u}{11 D} \right)^d \vol(\Omega) \right]
        + \exp \left[ - \la m \left( \frac{2 \eps u}{5 D} \right)^d \vol(\Omega) \right]
  \le 2 \, e^{- b \la u^d} \, ,
\end{equation}
where $ b = m \bigl( \eps/(11 D) \bigr)^d \vol(\Omega) $. Thus, the radius of stabilization satisfies
$ \P(R^{\Omega}(\marked{x}, \marked{\Po}_{\la \ka}) > u)
\le 2 s \, e^{- b \la u^d} $.

Recalling that $ \xi_\la(\marked{x}, \marked{\X}) = \xi \bigl( \la^{1/d} \marked{x},
\la^{1/d} \marked{\X} \bigr) $, observe that if $ \xi $ is confined to $ \NSIG $, then
$ \xi_\la $ is also confined to $ \NSIG $. Therefore,
$ R^{\Omega} $ is a radius of stabilization for $ \xi_\la $ inside $ \Omega $,
so that we can set $ R_\la(\marked{x}, \marked{\X}) := \la^{1/d} R^{\Omega}(\marked{x}, \marked{\X}) $.
Then we have $ \P(R_\la(\marked{x}, \marked{\Po}_{\la \ka}) > u) \le 2s \, e^{-b u^d} $.
Similarly as in \eqref{eq:NG:SuperExpDec}, it follows that the family $ (R_\la)_{\la > 0} $
satisfies Assumption~$\MomGrInt(1/d, \ka)$. This puts us into the position to formulate the
following result:

\begin{theo}
\label{SIG}
Let $ k \in \N $ and let $ \xi $ be a geometric functional confined to
$ \NSIG $. Take a convex bounded domain $ \Omega $ and
$ \ka $ satisfying Assumption~$\Dens$ and with
$ \inf_{x \in \Omega} \ka(x) > 0 $. Let $ \la_0 > 0 $ and
let the cones $ C_i $ be as above.
\begin{enumerate}
\item[$(1)$] If the family $ (\xi_\la)_{\la > \la_0} $ satisfies
Assumption~$\MomlaOne(p, \ka)$ for some $ p > 2 $, then
$ \xi $ satisfies
\newline
Assumptions~$\ConvergVar(\ka)$. Consequently,
\eqref{eq:Varlim} holds.
\item[$(2)$] Let $ 0 < \tau < \infty $ and suppose that the family
$ (\xi_\la)_{\la > \la_0} $ satisfies Assumption~$\MomHomOne(p, \tau, \Omega)$
for some $ p > 2 $. Next, suppose that notably many pairs $ (t, \marked{\X}) $
satisfy the following two conditions: first,
$ \Deltaxi \bigl( (\0, t), \marked{\X} \bigr) \ne 0 $; second,
there exist $ \rho > 0 $ and $ r > 4 \rho $, such that none of the sets
$ \marked{\X} \cap C_i \cap B_\rho(\0) $,
$ \marked{\X} \cap C_i \cap (B_r(\0) \setminus B_{4 \rho}(\0)) $ and
$ \marked{\X} \cap (C_i \setminus B_{4(r+\rho)}(\0)) $ is empty.
Then, with $ \Vxi $ as in \eqref{eq:Vxi}, we have $ \Vxi(\tau) > 0 $.
\item[$(3)$] Let $ \alpha \ge 0 $. If $ \xi $ satisfies Assumption~$\MomGrInt(\alpha)$,
then it satisfies Assumption~$\Gen(2 + \alpha)$; if $ \xi $ satisfies
Assumption~$\MomGrPt(\alpha)$, it satisfies Assumption~$\Gen(\max \{ \alpha, 1 \} + 1)$.
Consequently, the conclusions of Theorems~\ref{probabilities}, \ref{MDPTWI} and \ref{MDPmeasure}
hold with suitable $ \gamma $.
\end{enumerate}
\end{theo}

\begin{proof}
Parts~$(1)$ and $(3)$ follow exactly in the same way as parts~$(1)$ and $(4)$
of Theorem~\ref{NN}. Now we turn to part~$(2)$. Clearly, the family
$ (R_\la)_{\la > \la_0} $ satisfies Assumption~$\MomHom(q, \tau, \Omega)$
for all $ q \ge 0 $.
By Theorem~\ref{NDeg}, it remains to show that for each $ t \in \Marks $,
any finite configuration $ \marked{\X} \subset \marked{\R}^d \setminus \{ (\0, t) \} $
satisfying the specified conditions is externally stable at $ (\0, t) $ with respect to $ \xi $.

First, we claim that $ \marked{\X} $ is basically $ (r + \rho) $-externally stable at
$ (\0, t) $ with respect to $ \NSIG $. Take a finite configuration $ \marked{\Y} $
with $ \marked{\Y} \cap B_{r + \rho}(\0) = \marked{\X} \cap B_{r + \rho}(\0) $ and
$ \marked{z} \in \marked{\Y} \setminus B_{r + \rho}(\0) $. What we have to show is
that inserting a marked point at the origin into $ \Y $ does not affect the set of
marked points adjacent to $ \marked{z} $ in the sphere of influence graph.

Inserting $ (\0, t) $ into $ \marked{\Y} $ can affect the set of points adjacent to
$ \marked{z} $ in two ways: either it can make a new edge between $ (\0, t) $ and
$ \marked{z} $, or it can make some other point $ \marked{x} $ no longer adjacent to
$ \marked{z} $. The latter can happen if $ \marked{x} $ is adjacent to $ \marked{z} $ in
$ \SIG(\marked{\Y}) $ and $ \0 \in B(\marked{x}, \marked{\Y}) $. Therefore, it suffices
to show that $ \0 \in B(\marked{x}, \marked{\Y}) $ for no $ \marked{x} \in \marked{\Y} \setminus B_\rho(\0) $,
that no point $ \marked{x} \in \marked{\Y} \cap B_\rho(\0) $ is adjacent to
$ \marked{z} $ in $ \SIG(\marked{\Y}) $ and that $ (\0, t) $ is not adjacent to
$ \marked{z} $ in $ \SIG(\marked{\Y} \cup \{ (\0, t) \}) $.

Take $ \marked{x} \in \marked{\Y} \setminus B_\rho(\0) $. Recall that $ x \in C_i $
for some $ i $ and that $ \marked{\Y} \cap C_i \cap B_\rho(\0) =
\marked{\X} \cap C_i \cap B_\rho(\0) $ contains at least one point, say,
$ \marked{y} $. By elementary geometry, $ \| x - y \| < \| x \| $. Therefore,
$ \0 \notin B(\marked{x}, \marked{\Y}) $.

Now take $ \marked{x} \in \marked{\Y} \cap (B_\rho(\0) \setminus \{ \0 \}) $. Again,
choose $ i $ with $ x \in C_i $. In addition, choose $ v \in \R^d \setminus \{ \0 \} $,
such that the angle between $ x $ and $ v $ equals $ \pi/4 $.
Clearly, $ v \notin C_i $, but $ v \in C_j $ for some $ j $.
There exists $ \marked{w} \in \marked{\Y} \cap C_j \cap B_\rho(\0) $. Then
$ \marked{w} \ne \marked{x} $, but the angle between $ w $ and $ x $ is less than $ \pi/3 $, so that
$ \| w - x \| < \rho $. In other words, the intersection $ (\marked{\Y} \setminus \{ \marked{x} \})
\cap B_\rho(\marked{x}) $ is non-empty. Now choose $ k $, such that $ z \in C_k $, and choose
$ \marked{y} \in \marked{\Y} \cap C_k \cap B_r(\0) \setminus B_{4 \rho}(\0)) $. By
elementary geometry, we have $ C_k \setminus B_{4 \rho}(\0) \subseteq C_k^+ + x $,
so that $ y \in (C_k^+ + x) \setminus B_{3 \rho}(x) $. Observe that $ z \in C_k^+ + x $,
but also $ \| z - x \| \ge \| z \| - \| x \| > r \ge \| y - x \| $, so that
$ z \in (C_k^+ + x) \setminus B_{\| y - x \|}(x) $. By Proposition~\ref{SIG:Stab},
$ \marked{x} $ and $ \marked{z} $ are not adjacent in $ \SIG(\marked{\Y}) $.

Finally, take $ \marked{x} = (\0, t') \in \marked{\Y} \cup \{ (\0, t) \} $.
Again, let $ z \in C_k $. As none of the sets $ (\marked{\Y} \setminus \{ \marked{x} , (\0, t) \})
\cap B_\rho(\0) $ and $ \marked{\Y} \cap C_k \cap (B_r(\0) \setminus B_{4 \rho}(\0)) $
is empty, the conditions of Proposition~\ref{SIG:Stab} are fulfilled, so that
$ \marked{x} $ and $ \marked{z} $ are adjacent neither in $ \SIG(\marked{\Y}) $
nor in $ \SIG(\marked{\Y} \cup \{ (\0, t) \}) $. Thus, we conclude that
$ \marked{\X} $ is basically $ (r + \rho) $-externally stable at $ (\0, t) $ with
respect to $ \NSIG $.

By Proposition~\ref{StabExtStab} and confinement, it remains to show that
$ \NSIG $ stabilizes with respect to $ \marked{\X} $ and $ \marked{\X} \cup \{ (\0, t) \} $
at all $ \marked{x} \in (\marked{\X} \cup \{ (\0, t) \}) \cap B_{r+\rho}(\0) $.
Clearly, for all such $ \marked{x} $, the intersection $ (\marked{\X} \setminus \{ \marked{x} \})
\cap B_{r + \rho}(\marked{x}) $ is non-empty. Take $ i = 1, 2, \ldots , s $ and recall
that there exists $ \marked{y} \in \marked{\X} \cap (C_i \setminus B_{4(r + \rho)}(\0)) $.
However, we then have $ y \in (C_i^+ + x) \setminus B_{3(r + \rho)}(x) $. By
Proposition~\ref{SIG:Stab}, we then have $ \NSIG(\marked{x}, \marked{\Y}) \cap (C_i^+ + x)
= \NSIG(\marked{x}, \marked{\X}) \cap (C_i^+ + x) $ and
$ \NSIG(\marked{x}, \marked{\Y} \cup \{ (\0, t) \}) \cap (C_i^+ + x)
= \NSIG(\marked{x}, \marked{\X} \cup \{ (\0, t) \}) \cap (C_i^+ + x) $ for all
$ \marked{\Y} $ with $ \marked{\Y} \cap B_{2 \| y - x \|}(x) = \marked{\X} \cap
B_{2 \| y - x \|}(x) $. Since this can be deduced for all $ i = 1, \ldots , s $,
$ \NSIG $ stabilizes at $ \marked{x} $ with respect to $ \marked{\X} $
and $ \marked{\X} \cup \{ (\0, t) \} $. As a result, $ \marked{\X} $
is externally stable at $ (\0, t) $ with respect to $ \xi $. The proof is now
completed by Theorem~\ref{NDeg}.
\end{proof}

Theorem~\ref{SIG} adds to the existing results on non-degeneracy of the limiting
variance (see \cite{PY1}), central limit theorems (see Chapter~4 of \cite{Pe} as
well as \cite{BY2}) and, of course, large deviation results. The sphere of influence
graphs were not considered in \cite{BESY}.

\begin{exmpl}
\emph{Total number of edges}. Define $ \xi(x, \X) $ to be half the degree of
$ x $ in $ \SIG(\X) $ (assume that there are no marks). Then $ \langle 1, \muxilaka \rangle $
is precisely the total number of edges in $ \SIG(\X) $.

First, we turn to moment bounds. Take a bounded convex domain $ \Omega $
with $ \vol(\Omega) > 0 $ and a probability density function $ \ka $ with
$ \inf_\Omega \ka > 0 $, but vanishing outside $ \Omega $.
From the construction of the radius of stabilization,
it follows that for all $ \X \subseteq \Omega $, we have:
$$
 \bigl| \xi_\la(x, \X) \bigr|
    = \bigl| \xi(x, \X) \bigr|
  \le \bigl| \X \cap B_{R^\Omega(x, \X)}(x) \bigr|
    = \sum_{y \in \X} \1 \bigl( R^\Omega(x, \X) \ge \| y - x \| \bigr) \, .
$$
Let $ k \in \N $. Applying Lemma~\ref{PowSum} with $ a = k $ and $ b = 2k $ combined
with \eqref{eq:SIG:Stab:SuperExp} (notice that $ b $ in Lemma~\ref{PowSum} is different
from $ b $ in \eqref{eq:SIG:Stab:SuperExp}), we find that for some $ A_1 $ and $ A_2 $ not
depending on $ k $, we have:
\begin{equation}
\label{eq:SIG:Tot:SuperExp}
\begin{split}
 \Bigl( \E \bigl| \xi_\la(x, \Po_{\la \ka}) \bigr|^k \Bigr)^{1/k}
  &\le A_1 k \left\{ \left[
      \la \int_{\R^d} e^{- b \lambda \| y - x \|^d} \, \ka(y) \,\dl y
    \right]^{1/(2k)}
    + \la \int_{\R^d} e^{- b \lambda \| y - x \|^d/(2k)} \, \ka(y) \,\dl y
    \right\} \le
\\
  &\le A_2 k^2 \, .
\end{split}
\end{equation}
Combining this estimate with the observation $ |\xi_\la(x, \X \cup \Y)| \le |\xi_\la(x, \X)|
+ |\Y| $ and Stirling's formula, we find that the family $ (\xi_\la)_{\la > 0} $ satisfies
Assumption~$\MomGrPt(2, \ka)$. By part~$(3)$ of Theorem~\ref{SIG}, it then satisfies
Assumptions~$\Gen(3, \ka)$. Again, the range where moderate deviation results apply is
independent of the dimension.

Although non-degeneracy of the limiting variance is already proved in \cite{PY1}, we here
demonstrate that it also follows from part~$(2)$ of Theorem~\ref{SIG}: choose any
$ \rho > 0 $ and $ r > 4 \rho $. Letting $ c = \cos \frac{\pi}{12} $,
observe that if $ \X \cap B_{\rho/(2 c)}(\0) $ is empty, but each of the sets
$ \X \cap C_i \cap \bigl( B_\rho(\0) \setminus B_{\rho/(2 c)}(\0) \bigr) $ contains at
least two points, then $ \0 $ is the nearest neighbor of no point in $ \X $. Therefore,
insertion of the origin cannot remove any edges in $ \SIG $, but it adds at least
the edge between $ \0 $ and its nearest neighbor in $ \X $, so that
$ \Delta(\0, \X) \ne 0 $. Clearly, this condition along with
non-emptiness of the sets $ \X \cap C_i \cap (B_r(\0) \setminus B_{4 \rho}(\0)) $ and
$ \X \cap (C_i \setminus B_{4(r+\rho)}(\0)) $ is fulfilled for notably many
configurations $ \X $.

Next, from \eqref{eq:SIG:Tot:SuperExp} and again the fact that
$ |\xi_\la(x, \X \cup \Y)| \le |\xi_\la(x, \X)|
+ |\Y| $, it follows that the family $ (\xi_\la)_{\la > 0} $ satisfies
Assumption~$\MomHomOne(p, \tau, \Omega) $ for all $ p, \tau > 0 $. By
part~$(2)$ of Theorem~\ref{SIG}, $ \Vxi(\tau) > 0 $ for
all $ \tau > 0 $.
\end{exmpl}

\section{Proofs of the results}

\subsection{Moment measures and Palm distributions}
\label{ssc:MomPalm}

For a random measure $ \mu $ taking values in the space of
Borel measures over $ \R^d $, define its \emph{$ k $-th moment measure}
$ \Mom^k(\mu) $ as the one characterized by:
\begin{equation}
\label{eq:Mk}
 \bigl\langle f_1 \otimes \cdots \otimes f_k, \Mom^k(\mu) \bigr\rangle
  = \E \Bigl[ \langle f_1, \mu \rangle \cdots \langle f_k, \mu \rangle \Bigr]
\end{equation}
for all $ f_1, \ldots , f_k \in \B(\R^d) $, where $ f_1 \otimes \cdots \otimes f_k
\Colon (\R^d)^k \to \R $ is given by $ f_1 \otimes \cdots \otimes f_k(v_1, \ldots, v_k)
= f_1(v_1) \cdots f_k(v_k) $ (formula~$(5.4.3)$ on p.~133 of \cite{DV1});
the $ k $-th moment measure exists if the \emph{mixed moments} in the right-hand side of
\eqref{eq:Mk} exist for all $ f_1, \ldots , f_k \in \B(\R^d) $.

It will be helpful to consider products of $ \R^d $ and $ \marked{\R}^d $ indexed by
arbitrary finite sets: for a finite index set $ L $ consisting of distinct
elements $ i_1, \ldots, i_l $, denote by $ (\R^d)^L $ (respect.\ $ (\marked{\R}^d)^L $)
the product of $ l $ copies of $ \R^d $ (resp.\ $ \marked{\R}^d $).
Thus, for functions $ f_i \Colon \R^d \to \R $, $ i \in L $,
$ \bigotimes_{i \in L} f_i \Colon (\R^d)^L \to \R $ is
the counterpart of the function $ f_{i_1} \otimes \cdots \otimes f_{i_l} \Colon
(\R^d)^l \to \R $. In the special case where all functions are equal,
define $ f^{\otimes l} := \underbrace{f \otimes \cdots \otimes f}_l $
and its counterpart $ f^{\otimes L} := \bigotimes_{i \in L} f $.
For a random measure $ \mu $ on $ \R^d $, let $ \Mom^L(\mu) $ be the measure
on $ (\R^d)^L $, which is the counterpart of $ \Mom^l(\mu) $, i.~e.,
$ \bigl\langle \bigotimes_{i \in L} f_i, \Mom^L(\mu) \bigr\rangle =
\E \Bigl[ \prod_{i \in L} \langle f_i, \mu \rangle \Bigr] $.

For the random measures which are the subject of the present paper,
write $ \Mom^k_\la := \Mom^k(\muxilaka) $ for $ k \in \N $ and
$ \Mom^L_\la := \Mom^L(\muxilaka) $ for a finite set $ L $.
% (for the sake of readability, $ \xi $ and $ \ka $ will be suppressed in
% our notation).
These moment measures can be expressed in terms of singular measures,
see~\eqref{eq:prod}. That formula, also stated in \cite[p.~143]{DV1}, is
a special case of the Palm disintegration formula for a product of $ k $ copies
of $ \marked{\Po}_{\la \ka} $. First, recall the Palm formula for
$ \marked{\Po}_{\la \ka} $: for each functional $ G $, such that the integral and
the expectation below exist, we have:
\begin{equation}
\label{eq:Palm}
 \E \int_{\marked{\R}^d} G(\marked{x}, \marked{\Po}_{\la \ka}) \,
  \marked{\Po}_{\la \ka}(\dl \marked{x})
 = \la \int_{\marked{\R}^d} \E G(\marked{x}, \marked{\Po}_{\la \ka} \cup
   \{ \marked{x} \}) \, \ka(x) \,\dl \marked{x}
\end{equation}
(for the unmarked case, see \cite[pp.~280--281]{DV2}; the extension to marked
Poisson processes can be achived by conditioning on the marks; see Section~6.4 of
\cite{DV1}). To generalize this disintegration formula to $ k $-fold
integrals, we need singular measures. First, recall \eqref{eq:diffMarked} and for a
measurable function $ g \Colon \R^d \to \R $, define the
singular differential $ \bdl[g]{\marked{v}} $ of a $ (\marked{\R}^d)^k $-valued
variable as being characterized by the relation:
$$
 \int_{(\marked{\R}^d)^k} F(\marked{v}_1, \marked{v}_2, \ldots, \marked{v}_k)
  \, \,\bdl[g]{(\marked{v}_1, \ldots, \marked{v}_k)}
  = \int_{\marked{\R}^d} F(\marked{x}, \marked{x}, \ldots, \marked{x}) \, g(x) \, \dl \marked{x}
$$
for all measurable $ F \Colon (\marked{\R}^d)^k \to \R $.
Next, for $ \marked{v} = (\marked{v}_1, \ldots, \marked{v}_k) $ running over
$ (\marked{\R}^d)^k $, put:
$$
 \tdl[g]{\marked{v}} := \sum_{L_1, \ldots, L_p \preceq \{ 1, \ldots, k \}}
  \bdl[g]{\marked{v}_{L_1}} \cdots \bdl[g]{\marked{v}_{L_p}} \, ,
$$
where $ \marked{v}_L := (\marked{v}_l)_{l \in L} $; by $ \sum_{L_1, \ldots, L_p \preceq L} $,
we shall denote the sum of all \emph{unordered} partitions of a set $ L $.
Below we prove the following assertion, which generalizes the disintegration formula
\eqref{eq:Palm} to the $ k $-fold integral (see also p.~83 of \cite{Krk}):

\begin{prop}
\label{MultiPalm}
For each functional $ G $, such that the integral and the expectation below exist,
we have:
\begin{equation}
\label{eq:MultiPalm}
 \E \int_{(\marked{\R}^d)^k} G(\marked{v}, \marked{\Po}_{\la \ka}) \,
  \marked{\Po}_{\la \ka}(\dl \marked{v}_1) \cdots \marked{\Po}_{\la \ka}(\dl \marked{v}_k)
 = \int_{(\marked{\R}^d)^k} \E G(\marked{v}, \marked{\Po}_{\la \ka} \cup
   \{ \marked{v}_1, \ldots , \marked{v}_k \}) \, \tdl[\la \ka] \marked{v} \, ,
\end{equation}
where $ \marked{v} = (\marked{v}_1, \ldots, \marked{v}_k) $.
\end{prop}

\begin{proof}
As a first step, we prove \eqref{eq:MultiPalm} for the case where $ G(\marked{v}, \marked{\X}) $
vanishes if any two components $ v_i $ and $ v_j $ are equal. This can be proved by induction.
For $ k = 1 $, this is merely the formula~\eqref{eq:Palm}. For the induction step
from $ k $ to $ k + 1 $, use \eqref{eq:Palm} with
$ \int_{(\marked{\R}^d)^k} G(\marked{v}, \marked{\Po}_{\la \ka}) \,
\marked{\Po}_{\la \ka}(\dl \marked{v}_1) \cdots \marked{\Po}_{\la \ka}(\dl \marked{v}_k) $
in place of $ G(\marked{v}, \marked{\Po}_{\la \ka}) $ and notice that
the integration over $ \Po_{\la \ka} \cup \{ v_{k+1} \} $ coincides with the
integration over $ \Po_{\la \ka} $.

Next, observe the following straightforward extension. Let
$ L_1, \ldots , L_p $ be a partition of $ \{ 1, \ldots, k \} $. We say that
a point $ \marked{v} = (\marked{v}_1, \ldots, \marked{v}_k) $ \emph{follows}
this partition if any two components $ v_i $ and $ v_j $ are equal if and
only if the indices $ i $ and $ j $ lie in the same set $ L_r $. Now take
arbitrary $ G $, and define $ G_{L_1, \ldots , L_p}(\marked{v}, \marked{\X}) $
to be $ G(\marked{v}, \marked{\X}) $ if $ \marked{v} $ follows $ L_1, \ldots, L_p $
and zero otherwise. Then we have:
\begin{equation*}
\begin{split}
 & \E \int_{(\marked{\R}^d)^k} G_{L_1, \ldots, L_p}(\marked{v}, \marked{\Po}_{\la \ka}) \,
  \marked{\Po}_{\la \ka}(\dl \marked{v}_1) \cdots \marked{\Po}_{\la \ka}(\dl \marked{v}_k) = \null
\\ & \kern 3em \null
 = \int_{(\marked{\R}^d)^k} \E G_{L_1, \ldots, L_p}(\marked{v}, \marked{\Po}_{\la \ka} \cup
   \{ \marked{v}_1, \ldots , \marked{v}_k \}) \, \bdl[\la \ka] \marked{v}_{L_1}
     \cdots \bdl[\la \ka] \marked{v}_{L_p}
\\ & \kern 3em \null
 = \int_{(\marked{\R}^d)^k} \E G(\marked{v}, \marked{\Po}_{\la \ka} \cup
   \{ \marked{v}_1, \ldots , \marked{v}_k \}) \, \bdl[\la \ka] \marked{v}_{L_1}
     \cdots \bdl[\la \ka] \marked{v}_{L_p} \, .
\end{split}
\end{equation*}
Now write $ G = \sum_{L_1, \ldots, L_p \preceq \{ 1, \ldots, k \}} G_{L_1, \ldots, L_p} $,
sum up over all non-trivial partitions of $ \{ 1, \ldots, k \} $ and the proof
is complete.
\end{proof}

Now take a geometric functional $ \xi $ and recall the definition \eqref{eq:muxilaka} of its
associated random measure $ \muxilaka $. From Proposition~\ref{MultiPalm}, we deduce
that the corresponding moment measures $ \Mom_\la^k = \Mom^k(\muxilaka) $ can be expressed as:
\begin{equation}
\label{eq:prod1}
 \int_{(\R^d)^k} F(v) \, \Mom_{\la}^k(\dl v) =
 \int_{(\marked{\R}^d)^k} F(v) \, \mom_{\la}(\marked{v}) \, \tdl[\la \ka]{\marked{v}} \, ,
\end{equation}
where $ v = (v_1, \ldots, v_k) $ and again $ \marked{v} = (\marked{v}_1, \ldots , \marked{v}_k) $,
and where the Radon--Nikod\'ym derivative $ m_{\la} $ is given by:
\begin{equation}
\label{eq:prod}
 \mom_{\la}(\marked{v}_1, \ldots, \marked{v}_k) := \E \left[ \prod_{i=1}^k
 {\xi}_{\la} \bigl( \marked{v}_i, \marked{\Po}_{\la \ka} \cup
  \{ \marked{v}_1, \ldots, \marked{v}_k \} \bigr) \right] \, .
\end{equation}
Analogously, we define $ \mom_\la $ on products indexed by arbitrary index
sets, i.~e., $ \R^L $.

\subsection{The method of cumulants}
\label{ssc:MethCum}

We will refine the method of
cumulants and cluster measures as developed in \cite{BY2} in the
context of the central limit theorem.  We recall the formal
definition of cumulants in the context specified for our purposes.
For a random variable $ Y $ with all moments, expanding
the logarithm of the Laplace transform in a formal power series in $ t $ gives
\begin{equation}
\label{eq:LGEX}
 \log \left[ 1 + \sum_{k=1}^\infty \frac{\E Y^k}{k!} \, t^k
  \right] = \sum_{k=1}^\infty \frac{\cu^k(Y)}{k!} \, t^k \, ,
\end{equation}
where $ \cu^k(Y) $ denotes the \emph{$ k $-th cumulant} of $ Y $.
As the series \eqref{eq:LGEX} is considered as formal, no additional condition on
convergence is required for the cumulants to exist. Defining differentiation,
evaluation at zero, and the exponential and the logarithmic function of a formal
power series in the obvious way, one may also write:
$$
 \cu^k(Y) = \left. \frac{\dl^k}{\dl t^k} \right|_{t = 0} \log \E \exp( t Y )
$$
Similarly as mixed moments, one can also consider \emph{mixed cumulants}.
In the spirit of the above, one can define it by means of formal power
series of several variables:
\begin{equation}
\label{eq:CumGF}
 \cu(Y_1, \ldots, Y_k) = \left. \frac{\partial^k}{\partial t_1 \, \partial t_2 \cdots \partial t_k}
  \right|_{t_1 = t_2 = \cdots = t_k = 0} \log \E \exp( t_1 Y_1 + \cdots + t_k Y_k ) \, .
\end{equation}
In other words, the mixed cumulant of random variables $ Y_1, \ldots, Y_k $
is the coefficient in the formal power series expansion of
$ \log \E \exp( t_1 Y_1 + \cdots + t_k Y_k ) $ at $ t_1 t_2 \cdots t_k $.
Notice also that $ \cu^k(Y) = \cu(\underbrace{Y,\ldots,Y}_k) $.

To define the mixed cumulant of random variables $ Y_1, \ldots, Y_k $,
we do not even need all the moments to exist. All we need is the existence
of the expectations of the products $ \prod_{i \in L} Y_i $, where
$ L \subseteq \{ 1, \ldots, k \} $. This is because one can replace
the exponential function $ \exp( t_1 Y_1 + \cdots + t_k Y_k ) $ by
the polynomial $ g(t_1, \ldots, t_k) = \sum_{L \subseteq \{ 1, \ldots, k \}}
\E \prod_{i \in L} Y_i t_i $: the mixed cumulant $ \cu(Y_1, \ldots, Y_l) $
is then also the coefficient in the formal power series expansion of
$ \log g(t_1, \ldots, t_k) $ at $ t_1 t_2 \cdots t_k $.

In view of the above, mixed cumulants can be expressed in terms of mixed
moments. This can be made explicit by means of the following extension of the
celebrated \emph{Fa\`a di Bruno's formula} to functions of several variables:
\begin{equation}
\label{eq:FDB}
 \frac{\partial^k}{\partial t_1 \cdots \partial t_k} f(g(t_1, \ldots, t_k))
  = \sum_{L_1, \ldots, L_p \preceq \{ 1, \ldots, k \}} f^{(p)}(g(t_1, \ldots, t_k)) \,
   \frac{\partial^{|L_1|} g}{\prod_{i \in L_1} \partial t_i}
   \cdots
   \frac{\partial^{|L_p|} g}{\prod_{i \in L_p} \partial t_i}
\end{equation}
(see \cite{HardyFDB} and notice that although the result ibidem is stated
for real functions, the extension to formal power series is straightforward:
once we know the chain and the product rule, Fa\`a di Bruno's formula is a matter
of combinatorics, no longer analysis).
%
%.% Stanley?
%
Combining \eqref{eq:CumGF} and \eqref{eq:FDB}, we obtain the formula for
mixed cumulants:
\begin{equation}
\label{eq:CumMom}
 \cu(Y_1, \ldots, Y_k) = \sum_{L_1, \ldots, L_p \preceq \{ 1, \ldots, k \}} (-1)^{p-1} (p-1)!
  \E \left[ \prod_{i \in L_1} Y_i \right] \cdots \E \left[ \prod_{i \in L_p} Y_i \right] \, .
\end{equation}
(see p.~12 of \cite{SS}).
For a random measure $ \mu $, its $ k $-th cumulant measure $ \Cum^k(\mu) $
is defined analogously as its $ k $-th moment measure, i.~e.,
$ \langle f_1 \otimes \cdots \otimes f_k, \Cum^k(\mu) \rangle
= \cu(\langle f_1, \mu \rangle, \ldots, \langle f_k, \mu \rangle) $.
In particular, for equal functions, we have:
\begin{equation}
\label{eq:CumMeasCum}
 \langle f^{\otimes k}, \Cum^k(\mu) \rangle = \cu^k(\langle f, \mu \rangle) \, .
\end{equation}
In view of \eqref{eq:CumMom}, cumulant measures can be expressed in
terms of moment measures in the following way:
\begin{equation}
\label{eq:CumMomMeas}
 \Cum^k(\mu) = \sum_{L_1, \ldots, L_p \preceq \{ 1, \ldots, k \}}
  (-1)^{p-1} (p-1)! \, \Mom^{L_1}(\mu) \> \cdots \> \Mom^{L_p}(\mu) \, ,
\end{equation}
where the multiplication denotes the usual product of measures:
for disjoint finite sets $ G $ and $ H $, and for measurable sets
$ A \subseteq (\R^d)^G $ and $ B \subseteq (\R^d)^H $, we have
$ MN(A \times B) = M(A) \, N(B) $, identifying $ (\R^d)^{G \cup H}
\equiv (\R^d)^G \times (\R^d)^H $ (see p.~30 of \cite{MM}).

Although we use the same notation for cumulants as well as for cumulant
measures, this should not lead to a confusion: for a \emph{real-valued}
random variable $ Y $, $ \cu^k(Y) $ denotes a cumulant, while
for a random \emph{measure} $ \mu $, $ \Cum^k(\mu) $ denotes a
cumulant measure. Observe also that the first cumulant measure coincides
with the expectation measure and the second cumulant measure coincides
with the covariance measure.

% It is important to emphasize the following relationship between 
% the cumulant measures of random measures and (mixed) cumulants
% of random variables. Consider general random variables $Y_1,\ldots,Y_k$
% and construct a random measure $ \nu := \sum_{i=1}^k Y_i \delta_i $.
% Then the mixed $k$-th cumulant $ c(Y_1,\ldots,Y_k) $ of
% $ Y_1,\ldots,Y_k $ coincides with $ c^k(\nu)(\{(1,2,\ldots,k)\}) $.
% Moreover, the usual $k$-th cumulant $c^k[Y]$ of $Y$ is clearly
% equal to $c[\underbrace{Y,\ldots,Y}_k].$  

Throughout this subsection, $ \xi $ will (as usual) denote a geometric functional
and $ R $ its radius of stabilization. Recall the random measures
$ \muxilaka $ defined in \eqref{eq:muxilaka} and the corresponding moment measures
$ \Mom_\la^k := \Mom^k(\muxilaka) $. Similarly, consider the cumulant measures
$ \Cum^k_\la := \Cum^k(\muxilaka) $. Recalling the notation
$ \barmuxilaka = \muxilaka - \E \muxilaka $, observe that
$ \Cum^k(\barmuxilaka) = \Cum^k_\la $ for $ k \geq 2 $. Analogously,
define measures $ \Mom_\la^L $ and $ \Cum_\la^L $ defined on product spaces
indexed finite sets $ L $.

Now we can state our result controlling the growth of $ \langle f^{\otimes k}, \Cum_{\la}^k \rangle $,
which is crucial to prove Theorem~\ref{probabilities}.

\begin{lemm}\label{growthlemm}
If $ \xi $ satisfies Assumptions~$\Gen(\gamma, \ka)$, we have:
$$
 \bigl| \langle f^{\otimes k}, \Cum_{\la}^k \rangle \bigr|
  \leq \la \, C^k \| f \|_{\infty}^k (k!)^{1 + \gamma}
$$
for all bounded measurable functions
$ f \Colon \R^d \to \R $, all $ k = 3, 4, \ldots $ and all $ \la \ge \la_0 $,
where the constant $ C $ and the lower endpoint $ \la_0 $ only depend on
$ \ka $, $ \xi $ and $ R $.
\end{lemm}

Before proving the preceding lemma, we need a couple of auxiliary results.
Following \cite{BY2}, we decompose cumulant measures into
\emph{semi-cluster measures}, i.~e., cluster measures multiplied
by moment measures. For non-empty disjoint finite sets $ S $ and $ T $, define
the cluster measure by:
% non-empty disjoint finite sets. Given moment measures $ \Mom^S $, $ \Mom^T $
% and $ \Mom^{S \cup T} $ (of some random measure, which will sometimes be
% suppressed), define the corresponding cluster measure $ \Clu^{S, T}$ on
% $ (\R^d)^S \times (\R^d)^T $ by
$$
 \Clu_\la^{S,T} = \Mom_\la^{S \cup T} - \Mom_\la^S \, \Mom_\la^T
$$
(where multiplication again means product measure). The following
result is a refinement of Lemma~5.1 of \cite{BY2} in the sense that we
provide control over the number of summands.

\begin{lemm}
\label{CumClust}
For each non-trivial partition $ G, H $ of a finite set $ K $, the cumulant measure
$ \Cum_\la^K $ can be decomposed as:
$$
 \Cum_\la^K = \sum_{L_1, \ldots, L_p \preceq K}
  (-1)^{p-1} (p-1)! \, W_\la^{L_1, \ldots , L_p} \, ,
$$
where $ W_\la^{L_1, \ldots, L_p} $ is a sum of at most $ p $ terms of the form
$ \Clu_\la^{S, T} \Mom_\la^{K_1} \Mom_\la^{K_2} \cdots \Mom_\la^{K_r} $, where
$ S \subseteq G $ and $ T \subseteq H $ are non-empty and disjoint, and where
$ S \cup T, K_1, \ldots, K_r $ is a refinement of the partition $ L_1, \ldots, L_p $.
\end{lemm}

\begin{proof}
Starting from \eqref{eq:CumMomMeas}, we first note
that each moment measure $ \Mom_\la^{L_i} $ with $ S := L_i \cap G \ne \emptyset $
and $ T := L_i \cap H \ne \emptyset $ can be expressed as
$ \Clu_\la^{S, T} + \Mom_\la^S \Mom_\la^T $. Repeating the procedure,
we may write:
\begin{equation}
\label{eq:CumClust:MomClust}
 \Mom_\la^{L_1} \cdots \Mom_\la^{L_p} = \Mom_\la^{L_1 \cap G} \Mom_\la^{L_1 \cap H}
   \cdots \Mom_\la^{L_p \cap G} \Mom_\la^{L_p \cap H} + W_\la^{L_1, \ldots , L_p} \, ,
\end{equation}
where the measures $ W_\la^{L_1, \ldots , L_p} $ are as desired and where
we set $ \Mom_\la^\emptyset := 1 $. Now consider the measure:
\begin{equation}
\label{eq:CumClust:split}
 \Cum_\la^{K; G, H} := \sum_{L_1, \ldots, L_p \preceq K}
  (-1)^{p-1} (p-1)! \, \Mom_\la^{L_1 \cap G} \Mom_\la^{L_1 \cap H}
   \cdots \Mom_\la^{L_p \cap G} \Mom_\la^{L_p \cap H}
\end{equation}
and take functions $ f_i \in \B(\R^d) $, $ i \in K $. By Fa\`a di Bruno's formula~\eqref{eq:FDB},
$ \bigl\langle \bigotimes_{i \in K} f_i, \Cum_\la^{K; G, H} \bigr\rangle $
matches the coefficient in the formal power series expansion of
$ \log g_\la^{K; G, H} $ at $ \prod_{i \in K} t_i $, where:
$$
 g_\la^{K; G, H} := \sum_{L \subseteq K} \biggl\langle \bigotimes_{i \in L} f_i, \>
  \Mom_\la^{L \cap G} \Mom_\la^{L \cap H} \biggr\rangle \prod_{i \in L} t_i \, .
$$
However, $ g_\la^{K; G, H} = g_\la^G g_\la^H $, where $ g_\la^Z := \sum_{L \subseteq Z}
\bigl\langle \bigotimes_{i \in L}, \Mom_\la^L \bigr\rangle \prod_{i \in L} t_i $.
Since $ G $ and $ H $ are both non-empty,
the coefficient at $ \prod_{i \in K} t_i $ in the formal power series
expansion of both $ \log g_\la^G $ and $ \log g_\la^H $ vanishes; clearly,
the same is true for $ \log g_\la^{K; G, H} = \log g_\la^G + \log g_\la^H $.
Therefore, $ \Cum_\la^{K; G, H} = 0 $. Combining this with 
\eqref{eq:CumClust:MomClust} and \eqref{eq:CumClust:split}, the result follows.
\end{proof}

Thus, in order to estimate the cumulants, it suffices to estimate
semi-cluster measures. Recalling \eqref{eq:prod1} and \eqref{eq:prod},
it makes sense, as the first step towards the latter estimation,
to bound the differences $ \mom_\la(v_{S \cup T})
- \mom_\la(v_S) \, \mom_\la(v_T) $; throughout this subsection, we shall denote:
$$
 v_L = (v_i)_{i \in L} \, , \quad \marked{v}_L = (\marked{v}_i)_{i \in L}
 \quad \text{and} \quad
 V_L = \{ v_i \sth i \in L \} \, , \quad \marked{V}_L = \{ \marked{v}_i \sth i \in L \}
$$
for vectors $ v = (v_i)_{i \in K} \in (\R^d)^K $ and
$ \marked v = (\marked{v}_i)_{i \in K} \in (\marked{\R}^d)^K $, where $ L \subseteq K $
(the letters $ v $ and $ V $ are fixed, while the letters $ K $ and $ L $ can be arbitrary).
Next, define the \emph{separation} between two subsets $ A $ and $ B $ of a $ \R^d $ by:
$$
 \sepp(A, B) := \inf \{ \| a - b \| \sth a \in A, \> b \in B \} \, .
$$
Now recall the definition of $ \xi_\la $ along with the conventions on $ R $ and $ R_\la $
from Subsection~\ref{ssc:terminology}; in particular, recall that $ \la^{-1/d} R_\la $ a
radius of stabilization for $ \xi_\la $ inside $ \Omega $. In addition, recall that
$ \ka $ vanishes outside $ \Omega $, so that $ \marked{\Po}_{\la \ka} \subseteq \marked{\Omega} $
almost surely. For a finite set $ L $, $ \marked{v} = (\marked{v}_l)_{l \in L} \in (\marked{\R}^d)^L $,
$ i, j \in L $, $ \la > 0 $ and for a function $ \psi \Colon [0, \infty) \to [0, \infty) $,
define:
\begin{equation}
\label{eq:ab}
 a_{\la, i}(\marked{v}) := \Bigl[ \E \bigl|
  \xi_\la(\marked{v}_i, \marked{\Po}_{\la \ka} \cup \marked{V}_L) \bigr|^{2|L|} \Bigr]^{1/(2|L|)}
  \, , \qquad
 b_{\la, j, \psi}(\marked{v}) := \Bigl[ \E \Bigl(
  \psi \bigl( 2 R_\la(\marked{v}_j, \marked{\Po}_{\la \ka} \cup \marked{V}_L) \bigr)
 \Bigr)^{-2} \Bigr]^{1/2}
 \, .
\end{equation}

\begin{lemm}\label{L52refined}
Let $ S $ and $ T $ be non-empty finite disjoint sets and let
$ \psi \Colon [0, \infty) \to [0, \infty) $ be a non-increasing function. Then for
each $ \marked{v} = \marked{v}_{S \cup T} \in (\marked{\R}^d)^{S \cup T} $, we have:
\begin{equation}
\label{eq:L52refined:main}
\begin{split}
 \bigl| \mom_\la(\marked{v}_{S \cup T}) - \mom_\la(\marked{v}_S) \, \mom_\la(\marked{v}_T) \bigr| &\le
  \Biggl[ \prod_{i \in S \cup T} a_{\la, i}(\marked{v}_{S \cup T}) +
    \left( \prod_{i \in S} a_{\la, i}(\marked{v}_S) \right)
    \left( \prod_{i \in T} a_{\la, i}(\marked{v}_T) \right)
  \Biggr] \times
\\ & \kern 3em \null \times
  \Biggl[ \sum_{j \in S} b_{\la, j, \psi}(\marked{v}_S) + \sum_{j \in T} b_{\la, j, \psi}(\marked{v}_T) \Biggr]
 \, \psi(\la^{1/d} \delta) \, ,
\end{split}
\end{equation}
where $ \delta = \sepp \bigl( \{ v_i \sth i \in S \}, \{ v_j \sth j \in T \} \bigr) $ denotes
the separation with respect to the Euclidean metric.
\end{lemm}

\begin{rem}
This is a refinement of Lemma~5.2 of \cite{BY2} in at least two directions: first, we
state a more explicit upper bound, and second, we allow for arbitrary decay of $ R $
(described in terms of $ \psi $), not just exponential. Moreover, a closer
look reveals that the argument used for the proof of that result in fact
needs stronger assumptions than just exponential stabilization (apart from
moment bounds), as claimed ibidem. More precisely, in our notation, one has
to assume suitable stabilization of the functional $ (\marked{x}, \marked{\X})
\mapsto \xi(\marked{x}, \marked{\X} \cup \marked{V}) $, not $ \xi $, for finite
sets $ \marked{V} $ with suitable cardinality. This is due to a confusion
between $ \marked{\Po}_{\la \ka} $ and $ \marked{\Po}_{\la \ka} \cup \marked{V} $
(see below equation~$(5.4)$ ibidem). Moreover, in order to derive large deviation results
from appropriately corrected Lemma~5.2 of \cite{BY2}, one also needs certain
control over the dependence of the stabilization of
$ (\marked{x}, \marked{\X}) \mapsto \xi(\marked{x}, \marked{\X} \cup \marked{V}) $ on
the cardinality of $ \marked{V} $. These additional conditions can be tedious to verify in
actual applications. On the other hand, our argument, though much
more extensive, works under more or less standard conditions
and leads to a neat result.
\end{rem}

\begin{proof}[Proof of Lemma~\ref{L52refined}]
Take independent Poisson point processes $ \marked{\Po}_{\la \ka} $ and
$ \marked{\Po}_{\la \ka}^\dag $ (both with intensity $ \la \ka \times \P_\Marks $)
and define two new point processes:
\begin{equation*}
\label{eq:clust:PPPEx}
\begin{split}
 \marked{\Po}'_{\la \ka} &:= \bigl( \marked{\Po}_{\la \ka} \cap
   B_{\delta/2}(V_S) \bigr) \cup
  \bigl( \marked{\Po}_{\la \ka}^\dag \setminus B_{\delta/2}(V_S) \bigr) \, ,
\\
 \marked{\Po}''_{\la \ka} &:= \bigl( \marked{\Po}_{\la \ka} \setminus
   B_{\delta/2}(V_S) \bigr) \cup
  \bigl( \marked{\Po}_{\la \ka}^\dag \cap B_{\delta/2}(V_S) \bigr) \, ,
\end{split}
\end{equation*}
where, as usual, $ B_r(V) := \bigcup_{v \in V} B_r(v) $.
Observe that $ \marked{\Po}'_{\la \ka} $ and $ \marked{\Po}''_{\la \ka} $ are independent
Poisson point processes with intensity $ \la \ka \times \P_\Marks $. Setting:
$$
  X_i   := \xi_\la ( \marked{v}_i, \marked{\Po}_{\la \ka} \cup \marked{V}_{S \cup T} )
 \, , \quad
  X'_i  := \xi_\la ( \marked{v}_i, \marked{\Po}'_{\la \ka} \cup \marked{V}_S )
 \, , \quad
  X''_i := \xi_\la ( \marked{v}_i, \marked{\Po}''_{\la \ka} \cup \marked{V}_T )
 \, ,
$$
we may write:
$$
 \mom_\la(\marked{v}_{S \cup T}) - \mom_\la(\marked{v}_S) \, \mom_\la(\marked{v}_T)
 = \E \left[
  \left( \prod_{i \in S} X_i \right) \left( \prod_{i \in T} X_i \right) -
  \left( \prod_{i \in S} X'_i \right) \left( \prod_{i \in T} X''_i \right)
 \right] \, .
$$
Now observe that for $ i \in S $, $ X'_i $ agrees with $ X_i $ if $ \xi_\la $ stabilizes
at $ \marked{v}_i $ within radius less than $ \delta/2 $ with respect to
$ \marked{\Po}_{\la \ka} \cup \marked{V}_S $. Similarly, for $ i \in T $, $ X''_i $
agrees with $ X_i $ if $ \xi_\la $ stabilizes at $ \marked{v}_i $
within radius less than $ \delta/2 $ with respect to $ \marked{\Po}_{\la \ka} \cup \marked{V}_T $.
Letting:
$$
 I'_j  := \1 \left( R_\la \bigl( \marked{v}_j, \marked{\Po}'_{\la \ka} \cup \marked{V}_S \bigr)
  \ge \frac{\la^{1/d} \delta}{2} \right) \, , \qquad
 I''_j := \1 \left( R_\la \bigl( \marked{v}_j, \marked{\Po}''_{\la \ka} \cup \marked{V}_T \bigr)
  \ge \frac{\la^{1/d} \delta}{2} \right) \, ,
$$
we can estimate:
$$
 \bigl| \mom_\la(\marked{v}_{S \cup T}) - \mom_\la(\marked{v}_S) \, \mom_\la(\marked{v}_T) \bigr| \le
  \E \left\{
    \left[
      \prod_{i \in S \cup T} |X_i| +
      \left( \prod_{i \in S} |X'_i| \right) \left( \prod_{i \in T} |X''_i| \right)
    \right]
    \left[ \sum_{j \in S} I'_j + \sum_{j \in T} I''_j \right]
  \right\} \, .
$$
Since $\psi$ is non-increasing, we can estimate:
\begin{equation}
 I'_j  \le \frac{\psi(\la^{1/d} \delta)}{\psi \bigl( 2 R_\la(\marked{v}_j,
   \marked{\Po}'_{\la \ka} \cup \marked{V}_S) \bigr)}
  \quad \text{and} \quad
 I''_j \le \frac{\psi(\la^{1/d} \delta)}{\psi \bigl( 2 R_\la(\marked{v}_j,
   \marked{\Po}''_{\la \ka} \cup \marked{V}_S) \bigr)} \, .
\end{equation}
The proof is now completed by application of H\"older's inequality.
\end{proof}

To estimate the semi-cluster measures, we now need to integrate the estimate
\eqref{eq:L52refined:main}. Before tackling this job, we introduce some more
notation. First, we extend the convention on the \markedname\ accents
to the products $ (\R^d)^K $ and $ (\marked{\R}^d)^K $: if $ v $ and $ \marked{v} $
appear in the same context and if $ \marked{v} $ denotes a marked $ K $-tuple
$ (\marked{v}_i)_{i \in K} \in (\marked{\R}^d)^K $, then we shall assume
that $ v = (v_i)_{i \in K} \in (\R^d)^K $.

Now denote by $ \Delta_d^k := \{ (x, x, \ldots, x) \sth x \in \R^d \} $ the
diagonal in $ (\R^d)^k $; similarly, for a finite set
$ K $, denote by $ \Delta_d^K $ the diagonal in $ (\R^d)^K $.
We also consider the marked diagonal $ \marked{\Delta}_d^K = \{ \marked{v}
\in \marked{\R}^d \sth v \in \Delta_d^K \} $. Next, for
$ v \in (\R^d)^K \setminus \Delta_d^K $,
denote by $ \delta(v) $ the maximal separation between the sets
$ V_G $ and $ V_H $, where $ (G, H) $ runs over all non-trivial partitions of
$ K $ (i.~e., $ G $ and $ H $ are non-empty with union $ K $). Finally, denote by
$ \volB $ the volume of the unit ball in $ \R^d $.

The estimation of suitable integrals in the right-hand side of
\eqref{eq:L52refined:main} will be based on the following result.

\begin{lemm}\label{gpsi}
Let $ K_1, K_2, \ldots, K_r $ be finite disjoint sets with union $ K $.
Put $ k_l := |K_l| $ and $ k := |K| $. Take a non-increasing function
$ \psi \Colon [0, \infty) \to [0, \infty) $ with $ \lim_{t \to \infty} \psi(t) = 0 $
and with finite Riemann--Stieltjes integral
$ \int_0^\infty t^{(k-1)d} \, \dl (- \psi)(t) $, $ \la > 0 $, a marked
Poisson point process $ \marked{\Po}_{\la \ka} $, a non-negative geometric functional
$ g $ and $ i \in K_1 $. Then we have:
$$
 \int_{(\marked{\R}^d)^K \setminus \marked{\Delta}_d^K}
  \E g \bigl( \marked{v}_i, \marked{\Po}_{\la \ka} \cup \marked{V}_{K_1} \bigr)
  \psi \bigl( \la^{1/d} \delta(v) \bigr)
 \prod_{l=1}^r \tdl[\la \ka]{\marked{v}_{K_l}} % \le \null
 \le \la \, Q(k, \ka, \psi) \left[ \int_{\marked{\R}^d}
  \E \bigl( g(\marked{x}, \marked{\Po}_{\la \ka}) \bigr)^2 \ka(x) \,\dl \marked{x}
 \right]^{1/2} \, ,
$$
where $ \dst Q(k, \ka, \psi) = 2^{k-1} \, k! \int_0^\infty \bigl( 1 + e \| \ka \|_\infty \volB t^d \bigr)^{k-1}
\, \dl(- \psi)(t) $.
\end{lemm}

\noindent
Before proving Lemma~\ref{gpsi}, we need one more auxiliary result.

\begin{lemm}\label{partfact}
For all $ k \in \N $ and all $ u \in \R $, we have:
\begin{equation}
\label{eq:partfact}
 \sum_{L_1, \ldots, L_p \preceq \{ 1, \ldots, k \}} p! \, |L_1|! \, |L_2|! \cdots |L_p|! \, u^{p-1}
  = (1 + u)^{k-1} \, k! \, .
\end{equation}
\end{lemm}

\begin{proof}
Let $ f(y) = 1/(u(u + 1 - uy)) $, $ g(x) = 1/(1 - x) $ and observe that
the $ k $-th derivative of $ f(g(x)) $ at $ x = 0 $ matches the right hand
side of \eqref{eq:partfact}. Then apply Fa\`a di Bruno's formula \eqref{eq:FDB}.
\end{proof}

\begin{coro}\label{partfactalpha}
For all $ k \in \N $ and $ \alpha \ge 0 $, we have:
\begin{equation}
\label{eq:partfactalpha}
 \sum_{L_1, \ldots, L_p \preceq \{ 1, \ldots, k \}} p! \, (|L_1|!)^\alpha \, (|L_2|!)^\alpha \cdots (|L_p|!)^\alpha
  \le 2^{k-1} \, (k!)^{\max \{ \alpha, 1 \}} \, .
\end{equation}
\qed
\end{coro}

\begin{rem}
Clearly, the exponent $ \max \{ \alpha, 1 \} $ cannot be reduced.
\end{rem}

\begin{proof}[Proof of Lemma~\ref{gpsi}]
Let $ K' := K \setminus \{ i \} $, $ K'_1 := K_1 \setminus \{ i \} $ and $ K'_l := K_l $
for $ l = 2, 3, \ldots, r $. Next, take independent Poisson point processes
$ \marked{\Po}_{\la \ka}^{(1)}, \ldots, \marked{\Po}_{\la \ka}^{(r)} $ (all with intensity
$ \la \ka \times \P_\Marks $). Applying \eqref{eq:MultiPalm}, we may write:
\begin{equation*}
\begin{split}
 J &:= \int_{(\marked{\R}^d)^K \setminus \marked{\Delta}_d^K}
  \E g \bigl( \marked{v}_i, \marked{\Po}_{\la \ka} \cup \marked{V}_{K_1} \bigr)
  \, \psi \bigl( \la^{1/d} \delta(v) \bigr) \prod_{l=1}^r \tdl[\la \ka]{\marked{v}_{K_l}} =
\\
 &\phantom{:}= \E \int_{(\marked{\R}^d)^K \setminus \marked{\Delta}_d^K} g \bigl( \marked{v}_i, \marked{\Po}_{\la \ka}^{(1)} \bigr)
  \, \psi \bigl( \la^{1/d} \delta(v) \bigr)
  \left( \bigotimes_{l=1}^r \bigl( \marked{\Po}_{\la \ka}^{(l)} \bigr)^{\otimes K_l} \right) (\dl \marked{v})
\end{split}
\end{equation*}
(where $ \marked{v} = (\marked{v}_j)_{j \in (\marked{\R}^d)^K} $).
Next, we may asssume without loss of generality that $ \psi $ is left continuous,
so that we can write $ \psi(x) = \int_{[0, \infty)} \1(x \le t) \,\mu(\dl t) $
for some positive measure $ \mu $. Plugging this into the preceding equation, we obtain:
$$
 J = \int_{[0, \infty)} \E \int_{\marked{\R}^d} \int_{(\marked{\R}^d)^{K^\prime}}
  g \bigl( \marked{v}_i, \marked{\Po}_{\la \ka}^{(1)} \bigr) \,
  \1 \bigl( \delta(v) \le \la^{-1/d} t \bigr)
  \left( \bigotimes_{l=1}^r \bigl( \marked{\Po}_{\la \ka}^{(l)} \bigr)^{\otimes K_l^\prime} \right)
 (\dl \marked{v}_{K^\prime}) \, 
  \marked{\Po}_{\la \ka}^{(1)}(\dl \marked{v}_i) \, \mu(\dl t)
$$
(identifying $ v \in (\R^d)^K $ with $ (v_{K^\prime}, v_i) \in
(\R^d)^{K^\prime} \times \R^d) $). Now consider the graph with vertex set $ K $,
where vertices $ j $ and $ l $ are adjacent if $ \| v_j - v_l \| \le \la^{-1/d} t $.
Observe that $ \delta(v) \le \la^{-1/d} t $ if and only if this graph is connected.
Therefore, if $ \delta(v) \le \la^{-1/d} t $, then $ \| v_j - v_i \| \le N(v) \la^{-1/d} t $
for all $ j \in K $, where $ N(v) := | \{ v_j \sth j \in K \} | - 1 $. Next, estimating the
expression under the second integral sign by the Cauchy--Schwarz inequality, we find that:
\begin{equation}
\label{eq:gpsi:J}
 J \le \sqrt{A} \int_{[0, \infty)} \sqrt{B(t)} \, \mu(\dl t) \, ,
\end{equation}
where:
\begin{equation}
\label{eq:gpsi:AB}
\begin{split}
 A &= \E \int_{\marked{\R}^d} \Bigl( g \bigl( \marked{v}_i, \marked{\Po}_{\la \ka}^{(1)} \bigr) \Bigr)^2
  \marked{\Po}_{\la \ka}^{(1)}(\dl \marked{v}_i)
    = \la \int_{\marked{\R}^d} \E \bigl( g(\marked{x}, \Po_{\la \ka}) \bigr)^2 \ka(x) \,\dl \marked{x} \, ,
\\
 B(t) &= \E \int_{\marked{\R}^d} \left[ \int_{(\marked{\R}^d)^{K^\prime}}
  \1 \Bigl( \| v_j - v_i \| \le N(v) \la^{-1/d} t \>\> \text{for all} \>\> j \in K \Bigr)
  \left( \bigotimes_{l=1}^r \bigl( \marked{\Po}_{\la \ka}^{(l)} \bigr)^{\otimes K_l^\prime} \right)
 (\dl \marked{v}_{K^\prime})
 \right]^2
 \marked{\Po}_{\la \ka}^{(1)}(\dl \marked{v}_i) \, .
\end{split}
\end{equation}
To estimate $ B(t) $, let $ K''_1, K''_2, \ldots, K''_r $ be copies of the
sets $ K'_1, K'_2, \ldots, K'_r $, disjoint with $ K $. Put
$ K^\dag := \{ i \} \cup K''_1 \cup K''_2 \cup \cdots \cup K''_r $,
$ \hat K := K \cup K^\dag $ and $ \hat K_l := K_l \cup K''_l $.
Then we may write:
\begin{equation*}
\begin{split}
 B(t) &= \E \int_{(\marked{\R}^d)^{\hat K}}
  \1 \Bigl(
     \| v_j - v_i \| \le N(v_K) \la^{-1/d} t \>\> \text{for all} \>\> j \in K
     \, , \quad
     \| v_l - v_i \| \le N(v_{K^\dag}) \la^{-1/d} t \>\> \text{for all} \>\> l \in K^\dag
 \Bigr) \times \null
\\ & \kern 5em \null \times
 \left(
  \bigotimes_{l=1}^r
  \bigl( \marked{\Po}_{\la \ka}^{(l)} \bigr)^{\otimes \hat K_l}
 \right)(\dl \marked{v}) \, .
\end{split}
\end{equation*}
Noting that $ N(v_K), N(v_{K^\dag}) \le N(v) $ and disintegrating by \eqref{eq:MultiPalm}, we obtain:
\begin{equation*}
\begin{split}
 B(t) &\le \int_{(\marked{\R}^d)^{\hat K}}
  \1 \Bigl( \| v_j - v_i \| \le N(v) \la^{-1/d} t \>\> \text{for all} \>\> j \in \hat K \Bigr)
 \tdl[\la \ka]{\marked{v}_{\hat K_1}}
  \> \cdots \> \tdl[\la \ka]{\marked{v}_{\hat K_r}} \le
\\
 &\le \sum_{L_1, \ldots, L_p \preceq \hat K}
  \int_{(\marked{\R}^d)^{\hat K}}
  \1 \Bigl( \| v_j - v_i \| \le N(v) \la^{-1/d} t \>\> \text{for all} \>\> j \in \hat K \Bigr)
  \bdl[\la \ka]{\marked{v}_{L_1}}
  \> \cdots \> \bdl[\la \ka]{\marked{v}_{L_p}} \le
\\
 &\le \la \sum_{L_1, \ldots, L_p \preceq \hat K} \bigl( (p-1) \| \ka \|_\infty \volB t^d \bigr)^{p-1}
\end{split}
\end{equation*}
(where $ 0^0 := 1 $). Noting that $ (p-1)^{p-1} \le p! \, e^{p-1} $ and $ |\hat K| = 2k - 1 $,
application of Lemma~\ref{partfact} yields:
$$
 B(t) \le \la \, (2k - 1)! \, \bigl( 1 + e \| \ka \|_\infty \volB t^d \bigr)^{2k - 2} \, .
$$
Noting that $ (2k - 1)! \le 4^{k - 1} \, (k!)^2 $ and combining this with \eqref{eq:gpsi:J}
and \eqref{eq:gpsi:AB}, the proof is complete.
\end{proof}

\begin{lemm}
\label{abprod}
Let $ K_1, \ldots, K_r $ be a partition of a non-empty finite set $ K $. Put
$ k_l := |K_l| $ and $ k := |K| $. Take $ j \in K_1 $ and $ \alpha, \beta \ge 0 $.
Suppose that $ \ka $ satisfies Assumption~$\Dens$, that $ R $ satisfies
Assumption~$\MomGrInt(\beta, \ka)$ and that $ \xi $ satisfies either
$\MomGrPt(\alpha, \ka)$ or $\MomGrInt(\alpha, \ka)$. Letting:
$$
 J_{\la, \psi} := \int_{(\marked{\R}^d)^K \setminus \marked{\Delta}_d^K}
  \left( \prod_{l=1}^r \prod_{i \in K_l} a_{\la, i}(\marked{v}_{K_l}) \right)
  b_{\la, j, \psi}(\marked{v}_{K_1}) \, \psi \bigl( \la^{1/d} \delta(v) \bigr)
  \prod_{l=1}^r \tdl[\la \ka] \marked{v}_{K_l} \, ,
$$
there exists a non-increasing function $ \psi \colon [0, \infty) \to [0, \infty) $,
such that for all $ \la \ge \la_0 $,
\begin{align}
\label{eq:abprod:Subs}
 J_{\la, \psi} &\le \la \, C^k (k_1!)^\alpha \, (k_2!)^\alpha \cdots (k_r!)^\alpha \, (k!)^{1 + \beta d}
  && \text{under~$\MomGrPt(\alpha, \ka)$} \, ,
\\
\label{eq:abprod:Int}
 J_{\la, \psi} &\le \la \, C^k (k!)^{1 + \alpha + \beta d}
  &&\text{under~$\MomGrInt(\alpha, \ka)$} \, .
\end{align}
In both estimates, the constant $ C $ and the lower endpoint $ \la_0 $ only depend on
$ \ka $, $ \xi $ and $ R $.
\end{lemm}

\begin{proof}
Let $ \la \ge \la_0 $, where for $ \la_0 $, we take the maximal corresponding
lower endpoint from Assumption~$\MomGrInt(\beta, \ka)$ imposed on $ R $ and
$\MomGrPt(\alpha, \ka)$ or $\MomGrInt(\alpha, \ka)$, whichever imposed on $ \xi $.
If $ \xi $ satisfies Assumption~$\MomGrPt(\alpha, \ka)$, one can estimate, using Jensen's
inequality:
$$
 a_{\la, i}(\marked{v}_{K_l}) \le \Bigl[ \E \bigl|
  \xi_\la(\marked{v}_i, \marked{\Po}_{\la \ka} \cup \marked{V}_{K_l}) \bigr|^{p k_l} \Bigr]^{1/(p k_l)}
  \le A \bigl[ (p k_l)! \bigr]^{\alpha/(p k_l)} \le A p^\alpha \, (k_l!)^{\alpha/k_l} \, ,
$$
where $ p := \max\{ 2, \lceil 1/q \rceil \} $, and where $ A $ and $ q $ are as in
Assumption~$\MomGrPt(\alpha, \ka)$. As a result, we have:
$$
 J_{\la, \psi} \le A^k p^{k \alpha} \left( \prod_{l=1}^k (k_l!)^\alpha \right)
  \int_{(\marked{\R}^d)^K \setminus \marked{\Delta}_d^K}
   \Bigl[ \E \Bigl(
    \psi \bigl( 2 R_\la(\marked{v}_j, \marked{\Po}_{\la \ka} \cup \marked{V}_{K_1}) \bigr)
   \Bigr)^{-2} \Bigr]^{1/2}
   \, \psi \bigl( \la^{1/d} \delta(v) \bigr)
   \prod_{l=1}^r \tdl[\la \ka] \marked{v}_{K_l} \, .
$$
By the Cauchy--Schwarz inequality and Lemma~\ref{gpsi}, we can estimate:
\begin{equation}
\label{eq:abprod:J:Subs}
\begin{split}
 J_{\la, \psi} &\le A^k p^{k \alpha} \left( \prod_{l=1}^k (k_l!)^\alpha \right)
  \left[ \int_{(\marked{\R}^d)^K \setminus \marked{\Delta}_d^K}
   \E \psi \bigl( \la^{1/d} \delta(v) \bigr)
   \prod_{l=1}^r \tdl[\la \ka] \marked{v}_{K_l}
  \right]^{1/2} \times \null
\\ & \kern 3em \null \times
  \left[ \int_{(\marked{\R}^d)^K \setminus \marked{\Delta}_d^K}
   \E \Bigl(
    \psi \bigl( 2 R_\la(\marked{v}_j, \marked{\Po}_{\la \ka} \cup \marked{V}_{K_1}) \bigr)
   \Bigr)^{-2} \,
   \psi \bigl( \la^{1/d} \delta(v) \bigr)
   \prod_{l=1}^r \tdl[\la \ka] \marked{v}_{K_l}
  \right]^{1/2} \le
\\
 &\le \la \, A^k p^{k \alpha} \left( \prod_{l=1}^k (k_l!)^\alpha \right) Q(k, \ka, \psi) \,
  \left[ \int_{\marked{\R}^d}
   \E \Bigl( \psi \bigl( 2 R_\la(\marked{x}, \marked{\Po}_{\la \ka}) \bigr) \Bigr)^{-4} \ka(x) \,\dl \marked{x}
  \right]^{1/4} \, .
\end{split}
\end{equation}
Choosing $ \psi(t) := \bigl( 1 + e \| \ka \|_\infty \omega_d t^d \bigr)^k $, a straightforward
calculation yields $ Q(k, \ka, \psi) = 2^{k-1} k! \, k $. Writing
$ \bigl( \psi(2r) \bigr)^{-4} = \sum_{l=0}^{4k} \binom{4k}{l}
\bigl( 2 e \| \ka \|_\infty \omega_d \bigr)^l r^{ld} $
and recalling that $ R $ satisfies Assumption~$\MomGrInt(\beta, \ka)$, we find that:
\begin{equation*}
\begin{split}
 \int_{\marked{\R}^d}
  \E \Bigl( \psi \bigl( 2 R_\la(\marked{x}, \marked{\Po}_{\la \ka}) \bigr) \Bigr)^{-4} \ka(x) \,\dl \marked{x}
 &\le \sum_{l=0}^{4k} \binom{4k}{l} \bigl( 2 e \| \ka \|_\infty \omega_d B^d \bigr)^l \bigl[ (4ld)! \bigr]^\beta \le
\\
 &\le \bigl( 1 + 2 e \| \ka \|_\infty \omega_d B^d \bigr)^{4k} \bigl[ (4kd)! \bigr]^\beta \le
\\
 &\le (4d)^{4 k \beta d} \bigl( 1 + 2 e \| \ka \|_\infty \omega_d B^d \bigr)^{4k} (k!)^{4 \beta d} \, ,
\end{split}
\end{equation*}
where $ B $ is the constant $ A $ in \eqref{eq:MomGrInt}.
Plugging this into \eqref{eq:abprod:J:Subs} and applying $ k \le 3^{k/3} $, we obtain \eqref{eq:abprod:Subs}.

Now suppose that $ \xi $ satisfies Assumption~$\MomGrInt(\alpha, \ka)$. Then we apply Jensen's and
H\"older's inequality to estimate:
\begin{equation*}
\begin{split}
 J_{\la, \psi} &\le \left[ \prod_{s=1}^r \prod_{i \in K_s}
  \int_{(\marked{\R}^d)^K \setminus \marked{\Delta}_d^K}
   \E \bigl|
    \xi_\la(\marked{v}_i, \marked{\Po}_{\la \ka} \cup \marked{V}_{K_s})
   \bigr|^{2k} \, \psi \bigl( \la^{1/d} \delta(v) \bigr)
   \prod_{l=1}^r \tdl[\la \ka] \marked{v}_{K_l}
 \right]^{1/(2k)} \times \null
\\ & \kern 3em \null \times
 \left[
  \int_{(\marked{\R}^d)^K \setminus \marked{\Delta}_d^K}
   \E \Bigl(
    \psi \bigl( 2 R_\la(\marked{v}_j, \marked{\Po}_{\la \ka} \cup \marked{V}_{K_1}) \bigr)
   \Bigr)^{-2} \, \psi \bigl( \la^{1/d} \delta(v) \bigr)
   \prod_{l=1}^r \tdl[\la \ka] \marked{v}_{K_l}
 \right]^{1/2} \, .
\end{split}
\end{equation*}
Now choose $ \psi $ as before and apply Lemma~\ref{gpsi}. The estimate \eqref{eq:abprod:Int}
follows in more or less the same way that \eqref{eq:abprod:Subs} above.
\end{proof}

Now we are ready to state and prove bounds on semi-cluster measures.
Let $ \Sep(G, H) $ be the set of all $ (G \cup H) $-tuples of points where the
maximum is attained, i.~e.:
\begin{equation}
\label{eq:Sep}
 \Sep(G, H) := \bigl\{ v \in (\R^d)^{G \cup H} \sth \delta(v) = \sepp(V_G, V_H) \bigr\} \, ,
\end{equation}
recalling that $ \delta(v) $ denotes the maximum separation, precisely defined before
Lemma~\ref{gpsi}.

\begin{lemm}
\label{clust}
Let $ G, H $ and $ S, T, K_1, \ldots, K_r $ be two non-trivial
partitions of a finite set $ K $ with $ S \subseteq G $ and
$ T \subseteq H $. Put $ s = |S| $, $ t = |T| $, $ k_l = |K_l| $, $ k = |K| $.
Take $ \alpha, \beta \ge 0 $ and $ f \in \B(\R^d) $. Suppose that
$ \ka $ satisfies Assumption~$\Dens$, that $ R $ satisfies Assumption~$\MomGrInt(\beta, \ka)$,
and that $ \xi $ satisfies either $\MomGrPt(\alpha, \ka)$ or $\MomGrInt(\alpha, \ka)$.
Letting:
$$
 J_\la := \int_{\Sep(G, H)} f^{\otimes K} \,\dl \bigl( \Clu_\la^{S, T}
   \Mom_\la^{K_1} \cdots \Mom_\la^{K_r} \bigr) \, ,
$$
where the product in the right hand side means the usual product of measures
(like in \eqref{eq:CumMomMeas}), we have:
\begin{align}
\label{eq:clust:Subs}
 |J_\la| &\le \la \, C^k \| f \|_\infty^k \bigl( (s + t)! \bigr)^\alpha
  \, (k_1!)^\alpha \, (k_2!)^\alpha \cdots (k_r!)^\alpha \, (k!)^{1 + \beta d}
  && \text{under~$\MomGrPt(\alpha, \ka)$} \, ,
\\
\label{eq:slust:Int}
 |J_\la| &\le \la \, C^k \| f \|_\infty^k (k!)^{1 + \alpha + \beta d}
  &&\text{under~$\MomGrInt(\alpha, \ka)$} \, ,
\end{align}
for all $ \la \ge \la_0 $, where the constant $ C $ and the lower endpoint
$ \la_0 $ only depend on $ \ka $, $ \xi $ and $ R $.
\end{lemm}

\begin{proof}
Applying \eqref{eq:prod1}, write:
\begin{equation}
\label{eq:clust:J}
 J_\la = \int_{(\marked{\R^d})^K} \1(v \in \Sep(G, H)) \,
   f^{\otimes K}(v) \, D_\la(\marked{v}) \, \tilde{\tilde\dl} \marked{v} \, ,
\end{equation}
where:
\begin{equation*}
\begin{split}
 D_\la(\marked{v}) &= \bigl[ \mom_\la(\marked{v}_{S \cup T})
   - \mom_\la(\marked{v}_S) \, \mom_\la(\marked{v}_T) \bigr]
  \prod_{l=1}^r \mom_\la(\marked{v}_{K_l}) \, ,
\\
 \tilde{\tilde\dl} \marked{v} &= \tdl[\la \ka]{\marked{v}_{S \cup T}}
  \, \tdl[\la \ka]{\marked{v}_{K_1}} \cdots \tdl[\la \ka]{\marked{v}_{K_r}} =
  \tdl[\la \ka]{\marked{v}_S} \, 
  \tdl[\la \ka]{\marked{v}_T} \, \tdl[\la \ka]{\marked{v}_{K_1}} \cdots \tdl[\la \ka]{\marked{v}_{K_r}} \, .
\end{split}
\end{equation*}
Observe that since $ \sepp(V_S, V_T) \ge \sepp(V_G, V_H) = \delta(v) > 0 $ for all
$ v \in \Sep(G, H) $, the product differential $ \tdl[\la \ka]{\marked{v}_S} \, \tdl[\la \ka]{\marked{v}_T} $
coincides with $ \tdl[\la \ka]{\marked{v}_{S \cup T}} $.

By Lemma~\ref{L52refined} and the fact that
$ \sepp(V_S, V_T) \ge \delta(v) $ for $ v \in \Sep(G, H) $, the quantity\newline
$ \1(v \in \Sep(G, H)) |D_\la(\marked{v})| $ can be bounded
by a sum of $ 2(s+t) $ terms of the form
$ \left( \prod_{i \in K} \, \tilde a_i \right) \tilde b_j
\psi \bigl( \la^{1/d} \delta(v) \bigr) $, where
either $ \tilde a_i = a_{\la, i}(\marked{v}_S) $
    or $ \tilde a_i = a_{\la, i}(\marked{v}_T) $
    or $ \tilde a_i = a_{\la, i}(\marked{v}_{S \cup T}) $
    or $ \tilde a_i = a_{\la, i}(\marked{v_{K_l}}) $
for some $ l = 1, \ldots, r $, and where
either $ \tilde b_j = b_{\la, j, \psi}(\marked{v}_S) $
    or $ \tilde b_j = b_{\la, j, \psi}(\marked{v}_T) $.
Bounding those terms by Lemma~\ref{abprod} and applying $ 2(s+t) \le 2k \le 2^k $,
the result follows.
\end{proof}

\begin{proof}[Proof of Lemma~\ref{growthlemm}]
Put $ K = \{ 1, \ldots, k \} $ and write:
\begin{equation}
\label{eq:growthlemm:Sep}
 \langle f^{\otimes k}, \Cum_\la^k \rangle = \int_{(\R^d)^k} f^{\otimes k} \,\dl \Cum_\la^k
  = \int_{\Delta_d^K} f^{\otimes k} \,\dl \Cum_\la^K
   + \sum_{G, H \preceq K} \int_{\Sep(G,H)} f^{\otimes K} \,\dl \Cum_\la^K \, ,
\end{equation}
where the sum ranges over all unordered non-trivial partitions of $ K $ into two sets.
For the first term, we directly apply \eqref{eq:CumMomMeas}:
$$
 \int_{\Delta_d^K} f^{\otimes K} \,\dl \Cum_\la^K =
  \sum_{L_1, \ldots, L_p \preceq K}
  (-1)^{p-1} (p-1)! \int_{\Delta_d^K} f^{\otimes K}
  \,\dl \bigl( \Mom_\la^{L_1} \> \cdots \> \Mom_\la^{L_p} \bigr) \, .
$$
However, only the partition into one single set gives a non-zero integral.
Therefore,
$$
 \int_{\Delta_d^K} f^{\otimes K} \,\dl \Cum_\la^K =
 \int_{\Delta_d^K} f^{\otimes K} \,\dl \Mom_\la^K =
 \int_{\marked{\Delta}_d^K} f^{\otimes K}(v) \, \mom_\la(\marked{v}) \tdl[\la \ka] \marked{v}
$$
by \eqref{eq:prod1}. On the diagonal, $ \tdl[\la \ka] \marked{v} $ reduces to
$ \bdl[\la \ka] \marked{v} $, so that:
$$ 
 \int_{\Delta_d^K} f^{\otimes K} \,\dl \Cum_\la^K =
  \la \int_{\marked{\R}^d} \bigl( f(x) \bigr)^k \E \bigl[ \xi_\la(\marked{x},
   \marked{\Po}_{\la \ka}) \bigr]^k \ka(x) \,\dl \marked{x} \, .
$$
Now recall that Assumptions~$\Gen(\gamma, \ka)$ include that $ \xi $ satisfies either
Assumption~$\MomGrPt(\alpha, \ka)$ or the weaker Assumption~$\MomGrInt(\alpha, \ka)$
for some $ \alpha \le \gamma $. The latter one implies:
\begin{equation}
\label{eq:growthlemm:Diag}
 \left| \int_{\Delta_d^K} f^{\otimes K} \,\dl \Cum_\la^K \right| \le \la \, A^k \| f \|^k_\infty (k!)^\alpha
  \le \la \, A^k \| f \|^k_\infty (k!)^\gamma \, .
\end{equation}
Now we turn to the rest of the terms. We shall combine Lemmas~\ref{CumClust} and \ref{clust},
followed by Corollary~\ref{partfactalpha}. If $ \xi $ satisfies Assumption~$\MomGrPt(\alpha, \ka)$,
we can estimate:
\begin{equation*}
\begin{split}
 \left| \int_{\Sep(G,H)} f^{\otimes K} \,\dl \Cum_\la^K \right|
  &\le \la \, C_1^k \| f \|^k_\infty \sum_{L_1, \ldots, L_p \preceq K} p! \,
   (|L_1|!)^\alpha (|L_2|!)^\alpha \cdots (|L_p|!)^\alpha (k!)^{1 + \beta d} \le
\\
  &\le \la \, 2^{k-1} C_1^k \| f \|^k_\infty (k!)^{1 + \max \{ \alpha, 1 \} + \beta d} =
\\
  &= \la \, 2^{k-1} \, C_1^k \| f \|^k_\infty (k!)^{1 + \gamma} \, ,
\end{split}
\end{equation*}
where $ C_1 $ is the constant $ C $ from Lemma~\ref{clust}. Similarly, if $ \xi $
satisfies Assumption~$\MomGrInt(\alpha, \ka)$, we estimate:
\begin{equation*}
\begin{split}
 \left| \int_{\Sep(G,H)} f^{\otimes K} \,\dl \Cum_\la^K \right|
  &\le \la \, C_1^k \| f \|^k_\infty \sum_{L_1, \ldots, L_p \preceq K} p! \,
   (k!)^{2 + \alpha + \beta d} \le
\\
  &\le \la \, 2^{k-1} C_1^k \| f \|^k_\infty (k!)^{2 + \alpha + \beta d} =
\\
  &= \la \, 2^{k-1} \, C_1^k \| f \|^k_\infty (k!)^{1 + \gamma} \, .
\end{split}
\end{equation*}
Plugging the latter bounds along with
\eqref{eq:growthlemm:Diag} into \eqref{eq:growthlemm:Sep}, we find that:
$$
 \bigl| \langle f^{\otimes k}, \Cum_\la^k \rangle \bigr|
  \le \la \bigl( 2 \max \{ A, 2 C_1 \} \bigr)^k  \| f \|^k_\infty
  (k!)^{1 + \gamma} \, .
$$
This completes the proof.
\end{proof}

\subsection{Proof of bounds on deviation probabilities (Theorem~\ref{probabilities})}

As mentioned in Subsection~\ref{ssc:DevProb}, the proof of the result will be based
on the estimation of the cumulants, applying the celebrated lemma of Rudzkis, Saulis
and Statulevi\v{c}ius \cite{RSS}. Consider a general random variable $Y$ with finite
absolute moments of all orders and recall that $ \cu^k(Y) $ stands for the $k$-th cumulant
of $Y$. Below we state a simplified form of the version of that lemma which appears as
Lemma~2.3 on p.~18 of \cite{SS}:
\begin{lemm}
\label{RSSLemma}
Let $ Y $ be a random variable as above, with $ \E Y = 0 $ and $ \Var(Y) = 1 $, and
with its cumulants satisfying:
\begin{equation}
\label{eq:cumulant}
 |\cu^k(Y)| \leq \frac{(k!)^{1 + \gamma}}{\Delta^{k-2}}, \quad k = 3, 4, \ldots
\end{equation} 
for some $ \gamma \ge 0 $ and $ \Delta > 0 $. Then the large deviation relations:
\begin{align}
\label{eq:RSS+}
 \frac{\P(Y \geq y)}{1-\Phi(y)}
  &= \exp \bigl( L_\gamma(y) \bigr) \left( 1 + \theta_1 \psi(y) \frac{y+1}{\Delta_\gamma} \right) \, ,
\\
\label{eq:RSS-}
 \frac{\P(Y \leq - y)}{\Phi(-y)}
  &= \exp \bigl( L_\gamma(-y) \bigr) \left( 1 + \theta_2 \psi(y) \frac{y+1}{\Delta_\gamma} \right)
\end{align}
hold in the interval $ 0 \le y < \Delta_\gamma $. Here:
\begin{align}
\label{eq:Deltagamma}
 \Delta_\gamma &= \frac{1}{6} \left( \frac{\sqrt 2}{6} \, \Delta \right)^{1/(1 + 2 \gamma)} \, ,
\\
\label{eq:RSSf}
 \psi(y) &= \frac{60 \bigl[ 1 + 10 \Delta_\gamma^2
  \exp \bigl( - (1 - y/\Delta_\gamma) \sqrt{\Delta_\gamma} \bigr) \bigr]}{1 - y/\Delta_\gamma} \, ,
\end{align}
the quantities $ \theta_1 $ and $ \theta_2 $ belong to $ [-1, 1] $ and the function $ L_\gamma(y) $,
which is closely related to the Cram\'er--Petrov series, satisfies:
\begin{equation}
\label{eq:Lgamma}
 \bigl| L_\gamma(y) \bigr| \le \frac{|y|^3}{3 \Delta_\gamma}
\end{equation}
for all y with $ |y| \le \Delta_\gamma $.
\end{lemm}

The following weaker form of the preceding result will be used to prove the first part
of Theorem~\ref{probabilities}:

\begin{coro}
\label{RSSCoro}
Under the conditions of Lemma~\ref{RSSLemma}, there exist constants $ C_0 $, $ C_1 $
and $ C_2 $ depending only on $ \gamma $, such that for $ \Delta \geq C_0 $
and $ 0 \leq y \leq C_1 \Delta^{1/(1 + 2 \gamma)} $, we can estimate:
\begin{align}
\label{eq:RSSWeak+}
 \left| \log \frac{\P(Y \geq y)}{1-\Phi(y)} \right|
  &\leq C_2 \, \frac{1 + y^3}{\Delta^{1/(1 + 2 \gamma)}} \, ,
\\
\label{eq:RSSWeak-}
   \left| \log \frac{\P(Y \leq -y)}{\Phi(-y)} \right|
  &\leq C_2 \, \frac{1 + y^3}{\Delta^{1/(1 + 2 \gamma)}} \, .
\end{align}
\end{coro}

\begin{proof}
The key observation is that $ \psi(y) $ from \eqref{eq:RSSf} is bounded for $ 0 \le y \le q \Delta_\gamma $,
where $ q \in [0, 1) $ is fixed. Indeed, for such $ y $, one can estimate
$ \psi(y) \le c_1 + c_2 \Delta_\gamma^2 \exp \bigl( - c_3 \sqrt{\Delta_\gamma} \bigr) $,
where $ c_1 $, $ c_2 $ and $ c_3 $ depend only on $ q $. But the right-hand side of
the last estimate can be bounded uniformly in $ \Delta_\gamma $.

Boundedness of $ \psi $ along with \eqref{eq:RSS+}, \eqref{eq:Deltagamma} and
\eqref{eq:Lgamma} implies that there exist universal constants $ D_1 $, $ D_2 $
and $ D_3 $, such that:
\begin{equation}
\label{eq:RSSWeakProof1}
 \exp \left( - \frac{D_2 y^3}{\Delta^{1/(1 + 2 \gamma)}} \right)
  \left( 1 - \frac{D_3(1 + y)}{\Delta^{1/(1 + 2 \gamma)}} \right)
 \leq \frac{\P(Y \geq y)}{1-\Phi(y)} \leq
  \exp \left( \frac{D_2 y^3}{\Delta^{1/(1 + 2 \gamma)}} \right)
  \left( 1 + \frac{D_3(1 + y)}{\Delta^{1/(1 + 2 \gamma)}} \right)
\end{equation}
for all $ 0 \leq y \leq D_1 \Delta^{1/(1 + 2 \gamma)} $.

Now take $ \Delta \ge (3 D_3)^{1 + 2 \gamma} $ and
$ 0 \le y \le \Delta^{1/(1 + 2 \gamma)}/(3 D_3) $,
so that $ D_3 (1 + y){\Delta^{-1/(1 + 2 \gamma)}} \le 2/3 $. By convexity of
the logarithmic functions including convexity, we have:
\begin{equation}
\label{eq:RSSWeakProof2}
 - \log \left( 1 - \frac{D_3(1 + y)}{\Delta^{1/(1 + 2 \gamma)}} \right)
  \leq \frac{3 \log 3}{2} \, \frac{D_3(1 + y)}{\Delta^{1/(1 + 2 \gamma)}}
 \, .
\end{equation}
An easy exercise shows that $ y \le (2 + y^3)/3 $, so that:
\begin{equation}
\label{eq:RSSWeakProof3}
  \log \left( 1 + \frac{D_3(1 + y)}{\Delta^{1/(1 + 2 \gamma)}} \right)
  \leq - \log \left( 1 - \frac{D_3(1 + y)}{\Delta^{1/(1 + 2 \gamma)}} \right)
  \leq \frac{\log 3}{2} \, \frac{D_3(5 + y^3)}{\Delta^{1/(1 + 2 \gamma)}} \, .
\end{equation}
The estimate~\eqref{eq:RSSWeak+} now follows from \eqref{eq:RSSWeakProof1} and
\eqref{eq:RSSWeakProof3}. Similarly, we obtain \eqref{eq:RSSWeak-} and the proof
is complete.
\end{proof}

For the second part of Theorem~\ref{probabilities}, we shall need another
result, which is due to Bentkus and Rudzkis \cite{BR} and appears as Lemma~2.4
on page~19 of \cite{SS}. Like Lemma~\ref{RSSLemma},
we state it in a simplified form, which appears as a corollary of the
afore-mentioned result.

\begin{lemm}
\label{BRLemma}
Let $ Y $ be a random variable with $ \E Y = 0 $ and with its cumulants
satisfying:
\begin{equation}
\label{eq:BRLemmaCond}
 |\cu^k(Y)| \le \left( \frac{k!}{2} \right)^{1 + \gamma} \frac{H}{\Delta^{k-2}} 
\end{equation}
for some $ \gamma \ge 0 $, $ H > 0 $ and $ \Delta > 0 $. Then for all
$ y \ge 0 $, we have:
$$
 \P(Y \ge y) \le \exp \left( - \frac 1 4 \min \left\{ \frac{y^2}{H},
  (\Delta y)^{1/(1 + \gamma)} \right\} \right) \, .
$$
\end{lemm}

\begin{proof}[Proof of Theorem~\ref{probabilities}]\leavevmode\par
\textit{(1)}: Applying Lemma~\ref{growthlemm} along with \eqref{eq:Varlim}
and \eqref{eq:CumMeasCum} and recalling that $ \sigmaxikainf[f] > 0 $,
we find that for $ \la $ large enough and $k \geq 3$,
the cumulants of $ \langle f, \barmuxilaka \rangle $, i.~e.,
the cumulants of $ \langle f, \muxilaka \rangle $, i.~e.,
$ \langle f^{\otimes k}, \Cum_\la^k \rangle $, satisfy:
\begin{equation}
\label{eq:CumBd1}
 \frac{\bigl| \cu^k \bigl( \langle f, \barmuxilaka \rangle \bigr) \bigr|}%
  {\sigmaxilaka^k[f]}
 \leq \frac{C_1^k}{\la^{(k-2)/2}} \, \frac{\| f \|_\infty^k}{\sigmaxikainf^k [f]} \, (k!)^{1+\gamma}
\end{equation}
for some constant $ C_1 \ge 0 $ depending only on $ \ka $, $ \xi $ and $ R $, where
we recall from Subsection~\ref{ssc:KnRe}
that $ \sigmaxilaka^2[f] $ denotes the variance of 
$ \langle f, \barmuxilaka \rangle $ (and $ \sigmaxilaka^k[f] $ its
$ (k/2) $-th power; similarly, $ \sigmaxikainf^k[f] = (\sigmaxikainf[f])^k $).
To apply Corollary~\ref{RSSCoro}, rewrite the right-hand side of \eqref{eq:CumBd1} as:
$$
 C_1^2 \, \frac{\| f \|_\infty^2}{\sigmaxikainf^2[f]}
  \left( \frac{C_1}{\sqrt{\la}} \, \frac{\| f \|_\infty}{\sigmaxikainf[f]} \right)^{k-2}
  (k!)^{1 + \gamma}
 \le
 \left(
  \max \left\{ 1, \> C_1^2 \, \frac{\| f \|_\infty^2}{\sigmaxikainf^2[f]} \right\}
  \frac{C_1}{\sqrt{\la}} \, \frac{\| f \|_\infty}{\sigmaxikainf[f]}
 \right)^{k-2}
  (k!)^{1 + \gamma}
\, .
$$
Thus, recalling that the first cumulant of the centered measure $ \barmuxilaka $ is
zero whereas its higher order cumulants coincide with those of $ \muxilaka $, we can apply
Corollary~\ref{RSSCoro} to
$ Y := \langle f, \barmuxilaka \rangle/\sigmaxilaka[f] $
with $ y = x/\sigmaxilaka[f] $
and with $ \Delta $ taken to be $ \sqrt{\la} $
multiplied by some constant depending only on $ \ka $, $ \xi $, $ R $ and
the ratio $ \| f \|_\infty/\sigmaxikainf[f] $. It follows that there exist constants
$ \la_1, D_1, D_2 \ge 0 $,
such that for all $ \la \geq \la_1 $ and all $ 0 \le x \le D_1 \, \sigmaxilaka[f] \la^{1/(2+4\gamma)} $,
$$
 \left| \log \frac{\P\bigl( \langle f, \barmuxilaka \rangle \geq x \bigr)}%
  {1-\Phi(x/\sigmaxilaka[f])} \right|
 \leq \frac{D_2}{\la^{1/(2+4\gamma)}} \left[ 1 + \left( \frac{x}{\sigmaxilaka[f]} \right)^3 \right] \, .
$$
Applying \eqref{eq:Varlim} once again, we obtain that there exist constants
$ \la_2, D_3, D_4 \ge 0 $ such that for all $ \la \geq \la_2 $ and all
$ 0 \le x \le D_3 \, \la^{(1+\gamma)/(1+2\gamma)} $,
$$
 \left| \log \frac{\P\bigl( \langle f, \barmuxilaka \rangle \geq x \bigr)}%
  {1-\Phi(x/\sigmaxilaka[f])} \right|
 \leq D_4 \left[ \frac{1}{\la^{1/(2+4\gamma)}} + \frac{x^3}{\la^{(2+3\gamma)/(1+2\gamma)}} \right] \, .
$$
An analogous bound holds for the lower tail probabilities and part $(1)$
follows.

\medskip
\textit{(2)}:
Taking $ C $ from Lemma~\ref{growthlemm}, and letting $ D_5 := 2^{\gamma+1} C^2 \| f \|_\infty^2 $ and
$ D_6 := 1/(C \| f \|_\infty) $, we deduce that the random variable
$ Y := \langle f, \barmuxilaka \rangle $ satisfies \eqref{eq:BRLemmaCond} with:
\begin{equation}
\label{eq:BRlambda}
 H = D_5 \la \quad \text{and} \quad \Delta = D_6 \, .
\end{equation}
Next, if $ \sigmaxilaka^2[f] \le D_5 \la $,
then $ Y $ also satisfies \eqref{eq:BRLemmaCond} with:
\begin{equation}
\label{eq:BRsigma}
 H = \sigmaxilaka^2[f] \quad \text{and} \quad
 \Delta = D_6 \, \frac{\sigmaxilaka^2[f]}{D_5 \la} \, .
\end{equation}
Now take arbitrary $ 0 \le t \le 1 $. Combining \eqref{eq:BRlambda} and \eqref{eq:BRsigma}
with the inequality $ \min \{ a, b \} \le a^{1-t} b^{t} $, we find that
\eqref{eq:BRLemmaCond} is also satisfied with:
$$
 H = \sigmaxilaka^{2t}[f] \, (D_5 \la)^{1-t} \quad \text{and} \quad
 \Delta = D_6 \left( \frac{\sigmaxilaka^2[f]}{D_5 \la} \right)^t \, .
$$
Plugging this into Lemma~\ref{BRLemma}, we obtain:
\begin{equation}
\label{eq:BRLemmaProof}
 \P \bigl( \pm \langle f, \barmuxilaka \rangle \geq  x \bigr)
  \le \min_{0 \le t \le 1} \exp \left[ - \frac 1 4 \min \left\{ \frac{x^2}{\sigmaxilaka^{2t}[f]
   \, (D_5 \la)^{1-t}}, \>
  (D_6 x)^{1/(1+\gamma)} \left( \frac{\sigmaxilaka^2[f]}{D_5 \la} \right)^{t/(1+\gamma)} \right\} \right] \, .
\end{equation}
An easy exercise in optimization shows that for all $ a_1, a_2 > 0 $, $ b \ge 1 $
and $ c_1, c_2 \ge 0 $, we have
\hfill\break
$ \max_{0 \le t \le 1} \min \bigl\{ a_1 b^{c_1 t}, a_2 b^{-c_2 t} \bigr\}
= \min \bigl\{ a_1 b^{c_1}, a_2, a_1^{c_2/(c_1 + c_2)} a_2^{c_1/(c_1 + c_2)}
\bigr\} $. Plugging this into \eqref{eq:BRLemmaProof}, we obtain \eqref{eq:large0}.
\end{proof}

\subsection{Proof of moderate deviations}

The results of Subsection~\ref{ssc:MDP} will follow from the following consequence
of Theorem~\ref{probabilities}.

\begin{lemm}
\label{LimDevProb}
Let $ \barmuxilaka $ be defined as in Section~\ref{sc:Intr} with
$ \xi $ satisfying Assumptions~$\Gen(\gamma, \ka)$ and $\ConvergVar(\ka)$,
let $ a_\lambda $ satisfy \eqref{eq:assal} and take $ f \in \B(\R^d) $. Then, recalling
\eqref{eq:Varlim}, for $ \sigmaxika[f] > 0 $ and $ t \ge 0 $, we have:
\begin{equation}
\label{eq:LimDevProb+}
 \lim_{\la \to \infty} \frac{1}{a_{\la}^2} \log \P \bigl( a_{\la}^{-1} \la^{-1/2} 
  \langle f, \barmuxilaka \rangle \geq t \bigr)
 = \lim_{\la \to \infty} \frac{1}{a_{\la}^2} \log \P \bigl( a_{\la}^{-1} \la^{-1/2} 
  \langle f, \barmuxilaka \rangle > t \bigr)
 = - \frac{t^2}{2 \sigmaxika^2[f]} \, ,
\end{equation}
and for $ \sigmaxika[f] = 0 $ and $ t > 0 $, we have:
\begin{equation}
\label{eq:LimDevProb0}
 \lim_{\la \to \infty} \frac{1}{a_{\la}^2} \log \P \bigl( a_{\la}^{-1} \la^{-1/2} 
  \langle f, \barmuxilaka \rangle \geq t \bigr)
 = \lim_{\la \to \infty} \frac{1}{a_{\la}^2} \log \P \bigl( a_{\la}^{-1} \la^{-1/2} 
  \langle f, \barmuxilaka \rangle > t \bigr)
 = - \infty \, .
\end{equation}
\end{lemm}

\begin{proof}
Suppose first that $ \sigmaxika[f] > 0 $. In this case, we plug
$ x = t \, a_{\la} \, \la^{1/2} $ into \eqref{eq:large+}
and make use of the fact that for all $y \geq 0$,
$$
 \frac{1}{2 + \sqrt{2 \pi} \, y} \leq e^{y^2/2} \bigl( 1 - \Phi(y) \bigr) \leq \frac 1 2 \, .
$$
Combining both, we obtain the bound:
$$
 \bigg| \log \P \bigl( a_{\la}^{-1} \la^{-1/2}
 \langle f, \barmuxilaka \rangle \geq t \bigr) +
  \frac{a_{\la}^2 \la t^2}{2 \sigmaxilaka^2[f]} \biggr| 
 \leq \log \left( 2 + \frac{\sqrt{2 \pi} \, t \, a_{\la} \la^{1/2}}%
  {\sigmaxilaka[f]} \right)
  + C_2 \, \frac{1 + a_{\la}^3 t^3}{\la^{1/(2+2\gamma)}} \, .
$$
Dividing by $ a_\la^2 $, making use of \eqref{eq:Varlim} and applying
condition~\eqref{eq:assal}, this implies \eqref{eq:LimDevProb+} for `greater or equal'.
The corresponding result for the strict inequality follows by continuity.

In the case where $ \sigmaxika[f] = 0 $, plug $ x = t \, a_{\la} \, \la^{1/2} $
into \eqref{eq:large0} to obtain:
$$
 \frac{1}{a_\la^2} \log \P \bigl( a_{\la}^{-1} \la^{-1/2}
   \langle f, \barmuxilaka \rangle \geq t \bigr)
  \le - t^2 \min \left\{ C_4 \, \frac{\la}{\sigmaxilaka^2[f]}, \>
   C_5 \left( \frac{\la}{(t \, a_\la)^{2+4\gamma}} \right)^{{\textstyle \frac{1}{2+2\gamma}}}, \>
   C_6 \left( \frac{\la}{(t \, a_\la)^{2+4\gamma}} \right)^{{\textstyle \frac{1}{4+2\gamma}}}
  \right\}
$$
and the desired limiting behavior follows again from \eqref{eq:Varlim} and \eqref{eq:assal}.
This completes the proof.
\end{proof}

\begin{proof}[Proof of Theorem~\ref{MDPTWI}]
We apply the preceding lemma along with Theorem~4.1.11 of \cite{DZ}, which allows us
to derive a LDP from the limiting behavior of probabilities for a basis of topology.
For the latter, we choose all open intervals $ (u_1, u_2) $, where at least one of the
endpoints is finite and where none of the endpoints lies at the origin. Denote the family
of all such intervals by $ \mathcal U $. From Lemma~\ref{LimDevProb}, it follows that
for each $ U = (u_1, u_2) \in \mathcal U $,
$$
 \mathcal L_U := - \lim_{\la \to \infty} \frac{1}{a_\la^2}
  \log \P \bigl( a_{\la}^{-1} \la^{-1/2}
   \langle f, \barmuxilaka \rangle \in U \bigr) =
 \left\{ \begin{array}{cl}
  u_2^2/ \bigl( 2 \sigmaxika^2[f] \bigr) & ; \quad u_1 < u_2 < 0
 \\
  0 & ; \quad u_1 < 0 < u_2
 \\
  u_1^2/ \bigl( 2 \sigmaxika^2[f] \bigr) & ; \quad 0 < u_1 < u_2
 \end{array} \right. \, ,
$$
for all $ U \in \mathcal U $, recalling our convention on division by zero
from the end of Subsection~\ref{ssc:terminology}. By Theorem~4.1.11 of \cite{DZ},
the random variables $ a_{\la}^{-1} \la^{-1/2} \langle f,
\barmuxilaka \rangle $ satisfy a \emph{weak} LDP (MDP)
as $ \la \to \infty $ with speed $ a_\la^2 $ and rate function:
$$
 t \mapsto \sup_{\substack{U \in \mathcal U \\ t \in U}} \mathcal L_U =
  \frac{t^2}{2 \sigmaxika^2[f]} \, ,
$$
which matches the function $ \Ixika{f} $ from \eqref{eq:MDPRF}. Here, weak
LDP means that the lower bound in \eqref{eq:DefLDP} holds for all measurable sets
$ \Gamma $, while the upper bound holds for all relatively compact measurable
$ \Gamma $. However, from Lemma~\ref{LimDevProb}, it follows that the family
$ a_{\la}^{-1} \la^{-1/2} \langle f, \barmuxilaka \rangle $
is exponentially tight for speed $ a_\la^2 $, i.~e., for each $ M < \infty $,
there exists a measurable relatively compact set $ K $, such that
$ \limsup_{\la \to \infty} a_\la^{-2} \log \P \bigl( a_{\la}^{-1} \la^{-1/2}
\langle f, \barmuxilaka \rangle \notin K \bigr) \le - M $.
By Lemma~1.2.18 of \cite{DZ}, the family $ a_{\la}^{-1} \la^{-1/2} \langle f,
\barmuxilaka \rangle $ must then satisfy a \emph{full} LDP with
the same speed and the same \emph{good} rate function. This completes the proof.
\end{proof}

Now we turn to the proof of Theorem~\ref{MDPmeasure}. We begin with the same
argument based on Theorem~4.1.11 of \cite{DZ}, which requires a certain limiting
behavior of probabilities of sets from a basis of topology. The derivation
of that behavior requires a specific shape of the sets. Thus, we have to
show that certain sets of that shape form a basis of topology.

\begin{lemm}
\label{TopBasis}
Let $ F $ and $ V $ be pairwise dual finite-dimensional vector spaces, equipped
with the usual topology, and let $ Q $ be a positively semi-definite quadratic
form on $ F $. Let $ \mathcal U_0 $ be the family of all open subsets
of the half-spaces $ \{ \nu \in V \sth \langle f, \nu \rangle > b \} $, where
$ b > 0 $ and $ Q(f) = 0 $, and denote by $ \mathcal U_1 $ the family
of all sets of the form
\begin{equation}
\label{eq:TopBasisU}
 \bigl\{ \nu \in V \sth
  \langle f_0, \nu \rangle > b_0,
  \langle f_1, \nu \rangle < b_1,
  \langle f_2, \nu \rangle < b_2,
  \ldots ,
  \langle f_n, \nu \rangle < b_n
 \bigr\} \, ,
\end{equation}
where either $ 0 < b_0/\sqrt{Q(f_0)} < b_i/\sqrt{Q(f_i)} $ for all
$ i = 1, 2, \ldots, n $, or $ b_0 < 0 < b_i $ for all $ i = 1, 2, \ldots, n $.
Then the family $ \mathcal U_0 \cup \mathcal U_1 $ is a basis of the topology on $ V $.
\end{lemm}

\begin{proof}
Define $ F_0 := \{ f \in F \sth Q(f) = 0 \} $ and $ F_0^\bot := \{ \nu \in V \sth
\langle f, \nu \rangle = 0 \text{~for all~} f \in F_0 \} $. Now take
$ \mu \in V $ and its open neighborhood $ W $. We have to show that there
exists $ U \in \mathcal U_0 \cup \mathcal U_1 $, such that $ x \in U \subseteq W $.
We distinguish three cases.

\textit{Case~1: $ \mu \notin F_0^\bot $.} Then there exist $ f \in F_0 $ and
$ a > 0 $, such that $ \langle f, \mu \rangle > a $, so that we can take
$ U := W \cap \{ \nu \in V \sth \langle f, \mu \rangle > a \} \in \mathcal U_0 $.

\textit{Case~2: $ \mu = 0 $.} Then there exists $ \eps > 0 $ and
elements $ f_0, f_1, \ldots, f_n $, such that
$ U := \bigcap_{i=0}^n \bigl\{ \nu \sth \langle f_i, \nu \rangle < \eps \bigr\}
\subseteq W $. Clearly, $ 0 \in U $. Since we can write
$ U = \bigl\{ \nu \sth \langle - f_0, \nu \rangle > - \eps \bigr\}
\cap \bigcap_{i=1}^n \bigl\{ \nu \sth \langle f_i, \nu \rangle < \eps \bigr\} $,
we also have $ U \in \mathcal U_1 $.

\textit{Case~3: $ \mu \in F_0^\bot \setminus \{ 0 \} $.} Recalling our convention
on division by zero from the end of Subsection~\ref{ssc:terminology}, it follows
from standard linear algebra and topology that the map
$ f \mapsto \langle f, \mu \rangle/\sqrt{Q(f)} $ vanishes on $ F_0 $, is
continuous on $ F \setminus \{ 0 \} $, is bounded and that
attains its maximum, say, at $ f_0 $. Since $ \mu \ne 0 $, we also have
$ \langle f_0, \mu \rangle > 0 $ and $ Q(f_0) > 0 $.

There exist functions $ f_1, f_2, \ldots, f_n $ and $ \delta \in \bigl( 0,
\langle f_0, \mu \rangle \bigr) $, such that
$ U_0 := \bigl\{ \nu \in V \sth \langle f_0, \nu - \mu \rangle > - \delta \bigr\} \cap
\bigcap_{i=0}^n \bigl\{ \nu \in V \sth \langle f_i, \nu - \mu \rangle < \delta \bigr\}
\subseteq W $. Now consider the sets:
$$
 U_{\eps, t} := \bigl\{ \nu \in V \sth
  \langle f_0, \nu - \mu \rangle > - \eps \bigr\}
 \cap \bigcap_{i=1}^n \bigl\{
  \langle f_i + t f_i, \nu - \mu \rangle < \eps \bigr\} \, .
$$
Clearly, $ \mu \in U_{\eps, t} $ for all $ \eps, t > 0 $. Next, for each
$ \nu \in U_{\eps, t} $, we can estimate:
$$
 \langle f_i, \nu - \mu \rangle = \frac{1}{t} \Bigl[
  \langle f_0 + t f_i, \nu - \mu \rangle - \langle f_0, \nu - \mu \rangle \Bigr]
 < \frac{2 \eps}{t} \, .
$$
Therefore, if $ \eps < \delta $ and $ 2 \eps/t < \delta $, then
$ U_{\eps, t} \subseteq U_0 $.

Now we turn our attention to the question when $ U_{\eps, t} \in \mathcal U $.
By \eqref{eq:TopBasisU}, this will be surely true if:
\begin{equation}
\label{eq:TopBasisSep}
 \frac{\langle f_0 + t f_i, \mu \rangle + \eps}{\sqrt{Q(f_0 + t f_i)}}
 > \frac{\langle f_0, \mu \rangle - \eps}{\sqrt{Q(f_0)}} \, .
\end{equation}
Since $ \langle f, \mu \rangle/\sqrt{Q(f)} $ is maximal at $ f_0 $,
it follows from smoothness that there exist $ a, t_0 > 0 $, such that
for all $ i = 1, 2, \ldots, n $,
$$
 \frac{\langle f_0 + t f_i, \mu \rangle}{\sqrt{Q(f_0 + t f_i)}}
  \ge \frac{\langle f_0, \mu \rangle}{\sqrt{Q(f_0)}} - a t^2
$$
for all $ 0 \le t \le t_0 $. Therefore, the condition~\eqref{eq:TopBasisSep}
is satisfied if $ t \le t_0 $ and $ a t^2 < \eps/\sqrt{Q(f_0)} $.
Collecting everything together, we conclude after some calculation that if
$ \displaystyle 0 < t < \min \left\{ t_0, \frac{\delta}{2 a \sqrt{Q(f)}},
\sqrt{\frac{\delta}{a \sqrt{Q(f)}}} \right\} $,
then there exists $ \eps > 0 $, such that $ U_{\eps, t} \subseteq U_0 \subseteq W $
and $ U_{\eps, t} \in \mathcal U $, so that the desired set $ U $ exists in this case, too.
This completes the proof.
\end{proof}

\begin{proof}[Proof of Theorem~\ref{MDPmeasure}]
We split the argument into several steps.

\medskip
\textit{Step~1: derive a MDP for finite-dimensional restrictions of the measures.}
Denote by $ \mathcal F $ the set of all finite-dimensional subspaces of
$ \B(\R^d) $, fix $ F \in \mathcal F $ and consider the random measures
$ \muxilaka $ as linear functionals on $ F $. Let $ \mathcal U_0
\cup \mathcal U_1 $ be the basis of the usual topology on $ F' $, where
$ \mathcal U_0 $ and $ \mathcal U_1 $ are defined as in Lemma~\ref{TopBasis},
taking the quadratic form $ Q := \sigmaxika^2 $. For $ U = \bigl\{ \nu \in F' \sth
\langle f_0, \nu \rangle > b_0, \langle f_1, \nu \rangle < b_1, \ldots,
\langle f_n, \nu \rangle < b_n \bigr\} \in \mathcal U_1 $, where the numbers
$ b_0, b_1, \ldots, b_n $ satisfy the conditions below \eqref{eq:TopBasisU},
we have, by Lemma~\ref{LimDevProb}:
$$
 \mathcal L_U := - \lim_{\la \to \infty} a_\la^{-2} \P \Bigl( a_\la^{-1} \la^{-1/2}
  \barmuxilaka \in U \Bigr) = \frac{\bigl( \max \{ b_0, 0 \} \bigr)^2}{2 \sigmaxika^2[f_0]}
 \, ;
$$
for $ U \in \mathcal U_0 $, we have $ \mathcal L_U = \infty $.
By Theorem~4.1.11 of \cite{DZ}, the random functionals $ a_{\la}^{-1} \la^{-1/2}
\barmuxilaka \in F' $ satisfy a weak LDP (MDP) as $ \la \to \infty $
with speed $ a_\la^2 $ and rate function:
$$
 \nu \mapsto \sup_{\substack{U \in \mathcal U \\ \nu \in U}} \mathcal L_U
  = \sup_{f \in F} \frac{\langle f, \nu \rangle^2}{2 \sigmaxika^2[f]} \, .
$$
Finally, by Lemma~1.2.18 of \cite{DZ}, these random functionals satisfy
a full LDP (MDP) with the same good rate function because they are
exponentially tight for speed $ a_\la^2 $. To see this, take a basis
$ f_1, \ldots, f_n $ of $ F $ with $ \sigmaxika[f_i] \le 1 $ for all $ i $ and consider
the compact sets $ K_M := \bigcap_{i=1}^n \bigl\{ \nu \in F' \sth |\langle f_i,
\nu \rangle| \le M \bigr\} $. By Lemma~\ref{LimDevProb}, $ \limsup_{\la \to \infty}
a_\la^{-2} \P \bigl( a_\la^{-1} \la^{-1/2} \barmuxilaka
\notin K_M \bigr) \le - M^2/2 $ and exponential tightness follows. This
completes Step~1.

\medskip
\textit{Step~2: combine the MDP's for finite-dimensional restrictions into a MDP for
entire random measures.} We apply a version of the
Dawson--G\"artner theorem for projective limits, namely Theorem~4.6.9 of \cite{DZ},
naturally embedding $ \Meas(\R^d) $ into $ \bigl( \B(\R^d) \bigr)' $, the algebraic dual
of $ \B(\R^d) $, and identifying the projections of the functionals to finite-dimensional
spaces with their restrictions to finite-dimensional subspaces of $ \B(\R^d) $.
Thus we find that, as $ \la \to \infty $, the random measures
$ a_\la^{-1} \la^{-1/2} \barmuxilaka $ satisfy the MDP in
$ \bigl( \B(\R^d) \bigr)' $ with speed $ a_\la^2 $ and the good
rate function:
$$
 \Jxikameas(\nu) := \sup_{F \in \mathcal F} \sup_{f \in \Lin F}
  \frac{\langle f, \nu \rangle^2}{2 \sigmaxika^2[f]}
  = \sup_{f \in \B(\R^d)} \frac{\langle f, \nu \rangle^2}{2 \sigmaxika^2[f]} \, .
$$

\medskip
\textit{Step~3: compute the rate function.} Take $ \nu \in \bigl( \B(\R^d) \bigr)' $ and
distinguish five separate cases.

\textit{Case~1: $ \nu $ is unbounded with respect to the supremum norm on $ \B(\R^d) $.}
Since the latter is stronger than the seminorm $ \sigmaxika $, we have
$ \Jxikameas(\nu) = + \infty $ in this case.

\textit{Case~2: $ \nu $ is bounded with respect to the supremum norm, but is not a measure.}
This means that there exists a sequence
of bounded functions $ f_n $ with $ f_n \downarrow 0 $ pointwise, such that
the $ \langle f_n, \nu \rangle $ does not converge to $ 0 $. We may assume that
$ |\langle f_n, \nu \rangle| \ge 1 $ for all $ n $. Denoting $ L^2 := \{ f \Colon
\R^d \to \R \sth \sigmaxika[f] < \infty \} $ and noting that $ \B(\R^d) \subseteq
L^2 $, we find that $ \sigmaxika[f] \to 0 $ by the dominated convergence theorem.
Therefore $ \Jxikameas(\nu) = + \infty $.

\textit{Case~3: $ \nu $ is a measure, but is not absolutely continuous with respect
to $ \Vxi(\ka(x)) \, \ka(x) \,\dl x $,} where $ \Vxi $ is as in \eqref{eq:Vxi}. In this
case, there exists a measurable
set $ A $ with $ \int_A \Vxi(\ka(x)) \, \ka(x) \,\dl x = 0 $, but $ \nu(A) \ne 0 $.
In other words, $ \sigmaxika[1_A] = 0 $, but $ \langle \1_A, \nu \rangle \ne 0 $,
so that again $ \Jxikameas(\nu) = + \infty $.

\textit{Case~4: $ \nu \ll \Vxi(\ka(x)) \, \ka(x) \,\dl x $, but $ \sigmaxika[\rho] =
\infty $, where $ \rho(x) := \nu(\dl x)/(\Vxi(\ka(x)) \, \ka(x) \,\dl x) $.}
In this case, there exists a sequence $ \rho_1, \rho_2, \ldots \in L^2 $ which
converges pointwise to $ \rho $ and satisfies $ \rho \rho_n \ge 0 $ and $ |\rho_1| \le
|\rho_2| \le \ldots $ By the monotone convergence theorem, we have
$ \sigmaxika[\rho_n] \uparrow \infty $; we may assume that
$ \sigmaxika[\rho_1] > 0 $. Then the functions
$ g_n := \rho_n/\sigmaxika^2[\rho_n] $ satisfy $ \sigmaxika[g_n] = 1/\sigmaxika[\rho_n]
\to 0 $ but $ \langle g_n, \nu \rangle \ge 1 $, so that again
$ \Jxikameas(\nu) = + \infty $.

\textit{Case~5: $ \nu \ll \Vxi(\ka(x)) \, \ka(x) \,\dl x $ and
$ \rho(x) := \nu(\dl x)/(\Vxi(\ka(x)) \, \ka(x) \,\dl x) $ satisfies
$ \sigmaxika[\rho] < \infty $.} In this case we may write:
\begin{align*}
 \Jxikameas(\nu)
 &= \sup_{f \in \B(\R^d)}
  \frac{\displaystyle \left( \int_{\R^d} f(x) \, \rho(x) \, \Vxi(\ka(x))
    \, \ka(x) \,\dl x \right)^2}%
   {\displaystyle 2 \int_{\R^d} \rho^2(x) \, \Vxi(\ka(x)) \, \ka(x) \,\dl x}
 = \sup_{f \in L^2}
  \frac{\displaystyle \left( \int_{\R^d} f(x) \, \rho(x) \, \Vxi(\ka(x))
    \, \ka(x) \,\dl x \right)^2}%
   {\displaystyle 2 \int_{\R^d} \rho^2(x) \, \Vxi(\ka(x)) \, \ka(x) \,\dl x} =
\\
 &= \frac 1 2 \, \sigmaxika^2[\rho] = \Ixikameas(\nu) \, ,
\end{align*}
where the latter is defined in \eqref{eq:rateemp}. The second equality holds because
$ \B(\R^d) $ is dense in $ L^2 $; the third one is due to the Cauchy--Schwarz
inequality. This completes Step~3.

\medskip
\textit{Step~4: restrict the MDP.} To see that we may replace $ \bigl( \B(\R^d) \bigr)' $ and $ \Jxikameas $
with $ \Meas(\R^d) $ and $ \Ixikameas $, we apply Lemma~4.1.5 of \cite{DZ}, noting that
$ \Jxikameas $ agrees with $ \Ixikameas $ on $ \Meas(\R^d) $ and is infinite outside
$ \Meas(\R^d) $. This completes the proof.
\end{proof}

\subsection{Proof of non-degeneracy of the limiting variance}
\label{ssc:NDeg:Proof}

Throughout this subsection, we stick to the conventions on $ \xi $, $ \Hxi $, $ \Deltaxi $,
$ R $, $ \xi_\la $, $ \Deltaxi_\la $, $ R_\la $ and $ \Omega $ specified in Subsection~\ref{NDeg}.
Before proving Theorem~\ref{NDeg}, we need the following auxiliary result.

\begin{lemm}
\label{PowSum}
Let $ f $ be a non-negative locally integrable function. Take $ 1 < a < b $. Then there exists
a universal constant $ C $, such that for every non-negative geometric functional $ g $,
\begin{align*}
 \left[ \E \Biggl(
   \sum_{\marked{x} \in \marked{\Po}_f} g(\marked{x}, \marked{\Po}_f)
  \Biggr)^a \right]^{1/a}
  &\le
  C \left( \frac{(a-1) b}{b - a} \right)^{\frac{a-1}{a}} \Biggl\{
    \frac{b-1}{b-a} \left[
      \int_{\marked{\R}^d} \E \bigl( g(\marked{x}, \marked{\Po}_f) \bigr)^b f(x) \,\dl \marked{x}
    \right]^{1/b} + \null
\\ & \kern 11em \null
    + \int_{\marked{\R}^d} \Bigl[
     \E \bigl( g(\marked{x}, \marked{\Po}_f) \bigr)^b
   \Bigr]^{1/b} f(x) \,\dl \marked{x}
  \Biggr\} \, .
\end{align*}
\end{lemm}

\begin{proof}
Let $ m $ be a non-negative measurable function on $ \marked{\R}^d $, such that
for each $ \marked{x} $ with $ m(\marked{x}) = 0 $, $ g(\marked{x}, \marked{\Po}_f) $
almost surely vanishes. For each $ t \ge 0 $, let
$ \marked{M}(t) := \{ \marked{x} \sth m(\marked{x}) > t \} $. Write:
\begin{equation*}
\begin{split}
 J := \left[ \E \Biggl(
   \sum_{\marked{x} \in \marked{\Po}_f} g(\marked{x}, \marked{\Po}_f)
  \Biggr)^a \right]^{1/a}
 &= \left[ \E \Biggl(
   \sum_{\marked{x} \in \marked{\Po}_f \cap \marked{M}(0)}
   \frac{g(\marked{x}, \marked{\Po}_f)}{m(\marked{x})} \int_0^{m(\marked{x})} \dl t
  \Biggr)^a \right]^{1/a} =
\\
 &= \left[ \E \Biggl(
   \int_0^\infty \sum_{\marked{x} \in \marked{\Po}_f \cap \marked{M}(t)}
   \frac{g(\marked{x}, \marked{\Po}_f)}{m(\marked{x})} \, \dl t
  \Biggr)^a \right]^{1/a} \, .
\end{split}
\end{equation*}
By Minkowski's inequality for integrals and then by Jensen's inequality, we can estimate:
$$
 J \le \int_0^\infty \left[ \E \Biggl(
   \sum_{\marked{x} \in \marked{\Po}_f \cap \marked{M}(t)}
   \frac{g(\marked{x}, \marked{\Po}_f)}{m(\marked{x})}
  \Biggr)^a \right]^{1/a} \, \dl t
 \le \int_0^\infty \left[ \E \Biggl(
   \sum_{\marked{x} \in \marked{\Po}_f \cap \marked{M}(t)}
   |\marked{\Po}_f \cap \marked{M}(t)|^{a-1}
   \left( \frac{g(\marked{x}, \marked{\Po}_f)}{m(\marked{x})} \right)^a
  \Biggr) \right]^{1/a} \, \dl t \, .
$$
By the Palm formula \eqref{eq:Palm}, we can rewrite this estimate as:
$$
 J \le \int_0^\infty \left[ \int_{\marked{M}(t)}
   \E \Biggl\{
   \Bigl( 1 + |\marked{\Po}_f \cap \marked{M}(t)| \Bigr)^{a-1}
   \left( \frac{g(\marked{x}, \marked{\Po}_f)}{m(\marked{x})} \right)^a
   \Biggr\} f(x) \,\dl \marked{x} \right]^{1/a} \, \dl t \, .
$$
Applying H\"older's inequality, we obtain:
$$
 J \le \int_0^\infty \left[ \int_{\marked{M}(t)}
   \Biggl\{ \E
   \Bigl( 1 + |\marked{\Po}_f \cap \marked{M}(t)| \Bigr)^{(a-1)b/(b-a)} \Biggr\}^{(b-a)/b}
   \Biggl\{ \E
   \left( \frac{g(\marked{x}, \marked{\Po}_f)}{m(\marked{x})} \right)^b
   \Biggr\}^{a/b} f(x) \,\dl \marked{x} \right]^{1/a} \, \dl t \, .
$$
Now set $ m(\marked{x}) := \Bigl[ \E \bigl( g(\marked{x}, \marked{\Po}_f) \bigr)^b \Bigr]^{1/b} $.
In addition, observe that $ |\marked{\Po}_f \cap \marked{M}(t)| $ is
Poisson with expectation $ I(t) := \int_{\marked{M}(t)} f(x) \,\dl \marked{x} $.
If $ X \sim \Pois(\la) $, then for any $ n \in \N $, we may express the $ n $-th moment
of $ 1 + X $ in terms of a contour integral:
$$
 \E (1 + X)^n = \frac{n!}{2 \pi i} \oint_K \frac{e^{\la(e^z - 1) + z}}{z^{n+1}} \,\dl z \, .
$$
Choosing $ K $ to be the circle with radius $ 1/(1 + \la) $ centered at the origin, we can estimate
$ \E (1 + X)^n \le n! (1 + \la)^n f(\la) $, where $ f(\la) = e^{\la(e^{1/(\la + 1)} - 1) + 1/(\la + 1)} $.
Noting that $ f $ is bounded in $ \la > 0 $ and applying Stirling's formula,
we obtain that there exists a universal constant $ A $, such that $ \E (1 + X)^\gamma
\le (A \gamma(1 + \la))^\gamma $ for all $ \gamma > 0 $. Therefore,
$$
% \begin{equation}
% \label{eq:PowSum:I}
 J \le \left( \frac{A(a-1) b}{b - a} \right)^{\frac{a-1}{a}}
     \int_0^\infty \left[ \bigl( 1 + I(t) \bigr)^{a-1} I(t) \right]^{1/a} \,\dl t
   \le A \left( \frac{(a-1) b}{b - a} \right)^{\frac{a-1}{a}}
     \int_0^\infty \Bigl[ \bigl( I(t) \bigr)^{1/a} + I(t) \Bigr] \,\dl t \, .
% \end{equation}
$$
Observe that:
$$
% \begin{equation}
% \label{eq:PowSum:IntI}
 \int_0^\infty I(t) \,\dl t = \int_{{\R}^d} m(\marked{x}) \, f(x) \,\dl x \, .
% \end{equation}
$$
For the rest, apply Young's inequality:
\begin{equation}
\label{eq:PowSum:Young}
 \int_0^\infty \bigl( I(t) \bigr)^{1/a} \,\dl t \le
  \frac{a-1}{a} \int_0^\infty \phi(t) \,\dl t +
  \frac{1}{a} \int_0^\infty \bigl( \phi(t) \bigr)^{1-a} I(t) \,\dl t \, ,
\end{equation}
where we choose $ \phi(t) = \min \bigl\{ 1, (c/t)^{(b-1)/(a-1)} \bigr\} $;
$ c > 0 $ will be chosen later. It is easy to see that:
\begin{equation}
\label{eq:PowSum:Int1min}
 \int_0^\infty \phi(t) \,\dl t = \frac{b-1}{b-a} \, c \, .
\end{equation}
For the second term, we have:
$$
% \begin{equation}
% \label{eq:PowSum:IntphiPow}
 \int_0^\infty \bigl( \phi(t) \bigr)^{1-a} I(t) \,\dl t
  = \int_{\marked{\R}^d} \int_0^{m(\marked{x})} \max \biggl\{ 1, \left( \frac{t}{c} \right)^{b-1} \biggr\}
      \,\dl t \, f(x) \,\dl x
  \le \int_{\marked{\R}^d} \left( m(\marked{x}) + \frac{\bigl( m(\marked{x}) \bigr)^b}{b c^{b-1}} \right)
      f(\marked{x}) \,\dl x \, .
% \end{equation}
$$
Choosing $ c := \Bigl[ \int_{\marked{\R}^d} \bigl( m(\marked{x}) \bigr)^b \, f(x) \,\dl \marked{x} \Bigr]^{1/b} $,
combining with \eqref{eq:PowSum:Int1min} and plugging into \eqref{eq:PowSum:Young},
we obtain:
$$
 \int_0^\infty \bigl( I(t) \bigr)^{1/a} \,\dl t \le
   B \left[ \int_{\marked{\R}^d} \bigl( m(\marked{x}) \bigr)^b \, f(x) \,\dl \marked{x} \right]^{1/b}
  + \frac{1}{a} \int_{\marked{\R}^d} m(\marked{x}) \, f(x) \,\dl \marked{x} \, ,
$$
where $ B = \frac{a-1}{a} \, \frac{b-1}{b-a} + \frac{1}{ab} \le \frac{b-1}{b-a} $.
Collecting all together, the result follows.
\end{proof}

\begin{coro}
\label{PowMH}
Let $ (g_\la)_{\la > \la_0} $ and $ (R^*_\la)_{\la > \la_0} $ be two families of
non-negative geometric functionals. Define:
\begin{align*}
 h_\la(\marked{x}, \marked{\X}) &:= \sum_{\marked{y} \in \marked{\X}}
  g_\la(\marked{y}, \marked{\X})
  \1 \bigl( R_\la^*(\marked{y}, \marked{\X}) \ge \la^{1/d} \| y - x \| \bigr)
 \, ,
\\
 h^-_\la(\marked{x}, \marked{\X}) &:= \sum_{\marked{y} \in \marked{\X}}
  g_\la(\marked{y}, \marked{\X} \setminus \{ \marked{x} \})
  \1 \bigl( R_\la^*(\marked{y}, \marked{\X}) \ge \la^{1/d} \| y - x \| \bigr)
\end{align*}
(assuming $ \marked{x} \in \marked{\X} $).
Next, let $ 0 < \tau < \infty $ and let $ p, s, q > 0 $ with $ s/p + d/q < 1 $.
Suppose that the family $ (R^*_\la)_{\la > \la_0} $ satisfies
Assumption~$\MomHom(q, \tau, \Omega)$. Then, if the family
$ (g_\la)_{\la > \la_0} $ satisfies Assumption~$\MomHom(p, \tau, \Omega)$,
the family $ (h^-_\la)_{\la > \la_0} $ satisfies Assumption~$\MomHom(s, \tau, \Omega)$;
if the family $ (g_\la)_{\la > \la_0} $ satisfies Assumption~$\MomHomOne(p, \tau, \Omega)$,
the family $ (h_\la)_{\la > \la_0} $ satisfies Assumption~$\MomHom(s, \tau, \Omega)$.
\end{coro}

\begin{proof}
For arbitrary non-negative geometric functional $ \eta $ and any $ \tau, u > 0 $, set:
$$
 m_{u, \tau}(\eta) := \esssup_{\1(x \in \Omega) \,\dl \marked{x}}
  \Bigl[ \E \bigl( \eta(\marked{x}, \marked{\Po}_{\tau} \cap \Omega) \bigr)^u \Bigr]^{1/u} \, , \quad
 m^+_{u, \tau}(\eta) := \esssup_{\1(x, y \in \Omega) \,\dl \marked{x} \otimes \dl \marked{y}}
\Bigl[ \E \bigl( \eta(\marked{x}, (\marked{\Po}_{\tau} \cup \{ \marked{y} \} ) \cap \Omega) \bigr)^u \Bigr]^{1/u} \, .
$$
Thus, we have to bound $ m_{s, \la \tau}(h_\la) $ and $ m_{s, \la \tau}(h^-_\la) $.
We shall apply Lemma~\ref{PowSum}. Observe that one can choose $ 0 < r < q $ and $ t > s $
with $ t/p + d/r = 1 $. To bound $ m_{s, \la \tau}(h^-_\la) $, we first use
H\"older's and Markov's inequality to estimate:
\begin{equation*}
\begin{split}
 & \biggl[ \E \Bigl(
  g_\la(\marked{y}, \marked{\Po}_{\la \tau} \cap \Omega)
  \1 \bigl( R_\la^*(\marked{y}, \marked{\Po}_{\la \tau} \cap \marked{\Omega}) \ge \la^{1/d} \| y - x \| \bigr)
 \Bigr)^t \biggr]^{1/t} \le
\\
 & \kern 3em \null
 \le m_{p, \la \tau}(g_\la)
  \Bigl[ \P \bigl(
    R_\la^*(\marked{y}, \marked{\Po}_{\la \tau} \cap \marked{\Omega}) \ge \la^{1/d} \| y - x \|
  \bigr) \Bigr]^{d/r}
 \le \frac{m_{p, \la \tau}(g_\la)\bigl( m_q(1 + R^*_\la) \bigr)^{qd/r}}%
           {\bigl( 1 + \la^{1/d} \| y - x \| \bigr)^{qd/r}}
\end{split}
\end{equation*}
for almost all $ \marked{x}, \marked{y} \in \marked{\Omega} $. By Lemma~\ref{PowSum}, we have:
$$
 m_{s, \la \tau}(\marked{x}, h_\la^-) \le K_1 \left[ \int_\Omega
  \frac{\bigl( m_{p, \la \tau}(g_\la) \bigr)^t \bigl( m_q(1 + R^*_\la) \bigr)^{qtd/r}}%
       {\bigl( 1 + \la^{1/d} \| y - x \| \bigr)^{qtd/r}}
 \, \la \tau \,\dl y \right]^{1/t} \kern -3pt
 + K_2 \int_\Omega
  \frac{m_{p, \la \tau}(g_\la)\bigl( m_q(1 + R^*_\la) \bigr)^{qd/r}}%
       {\bigl( 1 + \la^{1/d} \| y - x \| \bigr)^{qd/r}} \, \la \tau \,\dl y \, .
$$
Substituting $ z = \la^{1/d} (y - x) $, observe that the both integrals converge uniformly
in $ x $, leading to the desired bound on
$ m_{s, \la \tau}(h^-_\la) $. Similarly, we can bound $ m_{s, \la \tau}(h_\la) $,
using $ m^+_{p, \la \tau}(g_\la) $ instead of $ m_{p, \la \tau}(g_\la) $. This
completes the proof.
\end{proof}

\begin{proof}[Proof of Theorem~\ref{NDeg}]
First, observe that there exists $ \rho > 0 $, such that for notably
many pairs $ (t, \marked{\X}) $,
$ \Deltaxi \bigl( (\0, t), \marked{\X} \bigr) \ne 0 $,
$ \marked{\X} $ is $ \rho $-externally stable at $ (\0, t) $
and $ \marked{\X} \subset B_\rho(\0) $. This also means
that with nonzero probability,
$ \Deltaxi \bigl( (\0, T), \marked{\Po}_1 \cap B_\rho(\0) \bigr) \ne 0 $ and
$ \marked{\Po}_1 \cap B_\rho(\0) $ is $ \rho $-externally stable at $ (\0, T) $, where
$ T $ is a generic random mark with distribution $ \P_\Marks $, independent
of $ \marked{\Po}_1 $ (see also Lemma~4.2 of \cite{PW}). However,
from the latter, we can deduce that
$ \Deltaxi \bigl( (\0, T), \marked{\Po}_1 \bigr) \ne 0 $ and
$ \marked{\Po}_1 $ is $ \rho $-externally stable at $ (\0, T) $:
see Remarks~\ref{ExtStabDelta} and \ref{ExtStabMatch}.

Now let $ v := \vol(\Omega) $ and $ \Omega^* := (\tau v)^{-1/d} \Omega $.
Define $ \ka $ to be the uniform density on $ \Omega^* $
(that is, $ \ka(x^*) := \tau \1( x^* \in \Omega^*) $) and let $ f $ to be the
indicator function of $ \Omega^* $. Then condition~\eqref{eq:NDeg:PW:xiMom}
need not be verified, while, by Remark~\ref{MomHomUnif}, the equivalent
conditions \eqref{eq:NDeg:PW:DeltafMom} and \eqref{eq:NDeg:PW:DeltaMom} hold
provided that the family $ (\Deltaxi_\la)_{\la > \la_0} $
satisfies Assumption~$\MomHom(s, \tau, \Omega)$. To verify the latter,
first estimate (assuming that $ \marked{x} \in \marked{\X} $):
$$
 \bigl| \Deltaxi_\la(\marked{x}, \marked{\X}) \bigr|
  \le \bigl| \xi_\la(\marked{x}, \marked{\X}) \bigr| +
   \sum_{\marked{y} \in \marked{\X}}
    \Bigl( \bigl| \xi_\la(\marked{y}, \marked{\X}) \bigr|
         + \bigl| \xi_\la(\marked{y}, \marked{\X} \setminus \{ \marked{x} \}) \bigr|
    \Bigr) \1 \bigl( R(\marked{y}, \marked{\X}) \ge \la^{1/d} \| y - x \| \bigr) \, .
$$
Clearly, there exists $ s > 2 $ with $ s/p + d/q < 1 $.
Since $ p > s $, the family $ (\xi_\la)_{\la > \la_0} $ satisfies
Assumption~$\MomHom(s, \tau, \Omega)$. To verify the latter for the
remaining sum, apply Corollary~\ref{PowMH}.

We have now verified all the conditions required for Theorem~\ref{NDeg:PW}.
To complete the proof, observe that $ \xi $ satisfies Assumptions~$\ConvergVar(\ka)$
by Remark~\ref{MomHomUnif}. Formula \eqref{eq:Varlim} applied to our choice of
$ f $ and $ \ka $ yields $ \Vxi(\tau) = \lim_{\la \to \infty} \sigmaxilaka^2[f]/\la > 0 $
and the proof is complete.
\end{proof}

\paragraph{Acknowledgements} The authors are grateful to J.~E. Yukich for
many inspiring discussions related to the subject of this paper. They would
also wish to thank the referees for pointing out several deficiencies of
the previous version. Their thoughtful remarks led to significant improvement
of the results.

Peter Eichelsbacher and Martin Rai\v{c} dedicate their contribution to
this paper to the memory of Tomasz Schreiber.

% BibTeX was used to generate references, but they are currently incorporated.
% 
% \bibliography{ERS}

\begin{thebibliography}{99}

\bibitem{BESY}
Y.~Baryshnikov, P.~Eichelsbacher, T.~Schreiber and J.~E. Yukich.
Moderate deviations for some point measures in geometric probability.
\textit{Ann. Inst. Henri Poincar\'e Probab. Stat.}
\textbf{44}, no.~3 (2008) 422--446.

\bibitem{BY1}
Y.~Baryshnikov and J.~E. Yukich.
Gaussian fields and random packing.
\textit{J. Statist. Phys.}
\textbf{111}, no.~1--2 (2003) 443--463.

\bibitem{BY2}
Y.~Baryshnikov and J.~E. Yukich.
Gaussian limits for random measures in geometric probability.
\textit{Ann. Appl. Probab.}
\textbf{15}, no.~1A (2005) 213--253.

% Removed.
% \bibitem{BY4}
% Y.~Baryshnikov, J.~E. Yukich, Gaussian fields and maximal points, preprint
%   (2006).

\bibitem{BR}
V.~Bentkus and R.~Rudzkis.
On exponential estimates of the distribution of random variables. (Russian)
\textit{Litovsk. Mat. Sb.}
\textbf{20}, no.~1 (1980) 15--30.
%.% No translation seems to be available.

\bibitem{BIV}
T.~Bodineau, D.~Ioffe and Y.~Velenik.
Rigorous probabilistic analysis of equilibrium crystal shapes.
\textit{J. Math. Phys.}
\textbf{41}, no.~3 (2000) 1033--1098.
  

\bibitem{CQ}
S.~N. Chiu and M.~P. Quine.
Central limit theory for the number of seeds in a growth model in {$\mathbb R\sp d$} with
  inhomogeneous {P}oisson arrivals.
\textit{Ann. Appl. Probab.}
\textbf{7}, no.~3 (1997) 802--814.

\bibitem{CFJP}
E.~G. Coffman, Jr., L.~Flatto, P.~Jelenkovi{\'c} and B.~Poonen.
Packing random intervals on-line. Average-case analysis of algorithms.
\textit{Algorithmica}
\textbf{22}, no.~4 (1998) 448--476.

\bibitem{DV1}
D.~J. Daley and D.~Vere-Jones.
\textit{An introduction to the theory of point processes, Volume~I: Elementary Theory and Methods.}
Second Edition.
% \textit{Probability and its Applications}.
Springer, New York, 2003.

\bibitem{DV2}
D.~J. Daley and D.~Vere-Jones.
\textit{An introduction to the theory of point processes, Volume~II: General Theory and Structure.}
Second Edition.
% \textit{Probability and its Applications}.
Springer, New York, 2008.

\bibitem{DZ}
A.~Dembo and O.~Zeitouni.
\textit{Large deviations techniques and applications}.
Second Edition.
% Vol.~38 of \textit{Applications of Mathematics (New York)}.
Springer-Verlag, New York, 1998.

\bibitem{DKS}
R.~Dobrushin, R.~Koteck{\'y} and S.~Shlosman.
\textit{Wulff construction. A global shape from local interaction.}
% Translated from the Russian by the authors.
% Vol.~104 of \textit{Translations of Mathematical Monographs}. 
American Mathematical Society, Providence, RI, 1992,

\bibitem{DR}
A.~Dvoretzky and H.~Robbins.
On the ``parking'' problem.
\textit{Magyar Tud. Akad. Mat. Kutat\'o Int. K\"ozl.}
\textbf{9} (1964) 209--225.

\bibitem{ES10}
P.~Eichelsbacher and T.~Schreiber.
Process level moderate deviations for stabilizing functionals.
\textit{ESAIM Probab. Stat.}
\textbf{14} (2010) 1--15.

\bibitem{Gorch}
A.~B. Gorchakov.
Upper bounds for cumulants of the sum of multi-indexed random variables.
\textit{Discrete Math. Appl.}
\textbf{5}, no. 4 (1995) 317--331.

\bibitem{HardyFDB}
M.~Hardy.
Combinatorics of partial derivatives.
\textit{Electron. J. Combin.}
\textbf{13}, no.~1 (2006), Research Paper 1, 13 pp. (electronic).

\bibitem{LH2005}
L.~Heinrich.
\textit{Large deviations of the empirical volume fraction for stationary Poisson grain models}.
\textit{Ann. Appl. Probab.}
\textbf{15}, no.~1A (2005) 392--420.

\bibitem{LH2009}
L.~Heinrich and M.~Spiess.
Berry--Ess\'een bounds and Cram\'er-type large deviations for the volume distribution of Poisson cylinder processes.
\textit{Lith. Math. J.}
\textbf{49}, no.~4 (2009) 381--398.

\bibitem{Krk}
K.~Krickeberg.
Moments of point processes.
In \textit{Probability and information theory, II}, pp.~70--101.
% Vol.~296 of \textit{Lecture Notes in Math.}
Springer, Berlin, 1973. 

\bibitem{MM}
V.~A. Malyshev and R.~A. Minlos.
\textit{Gibbs random fields}.
% Vol.~44 of \textit{Mathematics and its Applications (Soviet Series)}.
Kluwer Academic Publishers Group, Dordrecht, 1991.

\bibitem{Pe}
M.~D. Penrose.
\textit{Random geometric graphs}.
% Vol.~5 of Oxford Studies in Probability.
Oxford University Press, Oxford, 2003.

\bibitem{Pe1}
M.~D. Penrose.
Random parking, sequential adsorption, and the jamming limit.
\textit{Comm. Math. Phys.}
\textbf{218}, no.~1 (2001) 153--176.

\bibitem{Pe2}
M.~D. Penrose.
Multivariate spatial central limit theorems with applications to percolation and spatial graphs.
\textit{Ann. Probab.}
\textbf{33}, no.~5 (2005) 1945--1991.

%
%C% As far as we know, it has not been published.
%
\bibitem{Pe2005}
M.~D. Penrose.
Convergence of random measures in geometric probability.
Preprint (2005).
arXiv:math/0508464.

\bibitem{Pe2007LLN}
M.~D. Penrose.
Laws of large numbers in stochastic geometry with statistical applications
\textit{Bernoulli}
\textbf{13}, no.~4 (2007) 1124--1150.

\bibitem{Pe2007CLT}
M.~D. Penrose.
Gaussian limits for random geometric measures.
\textit{Electron. J. Probab.}
\textbf{12} (2007) 989--1035 (electronic).

\bibitem{PW}
M.~D. Penrose and A.~R. Wade.
Multivariate normal approximation in geometric probability.
\textit{J. Stat. Theory Pract}
\textbf{2}, no.~2 (2008) 293--326.

\bibitem{PY1}
M.~D. Penrose and J.~E. Yukich.
Central limit theorems for some graphs in computational geometry.
\textit{Ann. Appl. Probab.}
\textbf{11}, no.~4 (2001) 1005--1041.

\bibitem{PY2}
M.~D. Penrose and J.~E. Yukich.
Limit theory for random sequential packing and deposition.
\textit{Ann. Appl. Probab.}
\textbf{12}, no.~1 (2002) 272--301.

\bibitem{PY4}
M.~D. Penrose and J.~E. Yukich.
Weak laws of large numbers in geometric probability.
\textit{Ann. Appl. Probab.}
\textbf{13}, no.~1 (2003) 277--303.

\bibitem{PY6}
M.~D. Penrose and J.~E. Yukich.
Normal approximation in geometric probability.
In \textit{Stein's method and applications}, pp.~37--58.
% Vol.~5 of \textit{Lect. Notes Ser. Inst. Math. Sci.}
Singapore Univ. Press, Singapore, 2005.

\bibitem{Re}
A.~R{\'e}nyi.
Th{\'e}orie des {\'e}l{\'e}ments saillants d'une suite d'observations.
In \textit{Colloquium on Combinatorial Methods in Probability Theory}, pp.~104--115.
Mathematical Institut, Aarhus Universitet, Denmark, 1962.

\bibitem{RSS}
R.~Rudzkis, L.~Saulis and V.~Statulevi{\v{c}}ius.
\textit{A general lemma on probabilities of large deviations}.
\textit{Lith. Math. J.}
\textbf{18}, vol.~2 (1978) 226--238.

\bibitem{SS}
L.~Saulis and V.~Statulevi\v{c}ius.
\textit{Limit theorems on large deviations}.
Kluwer Academic Publishers Group, Dordrecht, 1991.

\bibitem{SY}
T.~Schreiber and J.~E. Yukich.
Large deviations for functionals of spatial point processes
with applications to random packing and spatial graphs.
\textit{Stochastic Process. Appl.}
\textbf{115}, no.~8 (2005) 1332--1356.

\bibitem{SY2}
T.~Schreiber and J.~E. Yukich.
Variance asymptotics and central limit theorems for
generalized growth processes with applications to convex hulls and maximal points.
\textit{Ann. Probab.}
\textbf{36}, vol.~1 (2008) 363--396.

% \bibitem{WA}
% A.~R. Wade, Explicit laws of large numbers for random nearest-neighbour-type
%   graphs, Adv. in Appl. Probab. 39~(2) (2007) 326--342.

\end{thebibliography}
% 

\end{document}